\documentclass[final,reqno,12pt]{amsart}
\usepackage{rotating,pdflscape,color,amssymb,hyperref,subfigure,psfrag,amsmath,eufrak,bbm}

\newtheorem{theorem}{Theorem}

\newtheorem{corollary}[theorem]{Corollary}

\newtheorem{lemma}[theorem]{Lemma}
\newtheorem{proposition}[theorem]{Proposition}

\newtheorem{remark}[theorem]{Remark}

\numberwithin{theorem}{section}
\newcommand{\cC}{\ensuremath{\mathcal C}}

\newcommand{\cE}{\ensuremath{\mathcal E}}
\newcommand{\cF}{\ensuremath{\mathcal F}}

\newcommand{\cH}{\ensuremath{\mathcal H}}

\newcommand{\cK}{\ensuremath{\mathcal K}}

\newcommand{\cP}{\ensuremath{\mathcal P}}

\newcommand{\cX}{\ensuremath{\mathcal X}}

\def\ind{{\mathbf{1}}}

\def\dist{{\mathrm{dist}}}
\def\tr{{\mathrm{tr}}}
      \let\e=\varepsilon

\def\bZ{{\Bbb Z}}

\def\bR{{\Bbb R}}
\def\bC{{\Bbb C}}
\def\bE{{\Bbb E}}
\def\bP{{\Bbb P}}

\def\({\left(}
\def\){\right)}
\def\stab{\mathrm{Stab}}
\renewcommand{\Im}{\mathrm{Im}}
\renewcommand{\Re}{\mathrm{Re}}

\title[Eigenvectors of Heavy-tailed random matrices]{Localization and delocalization of eigenvectors for heavy-tailed random matrices}
\author{Charles Bordenave}
\email{charles.bordenave(at)math.univ-toulouse.fr}
\address{CNRS \& Universit\'e de Toulouse, Institut de Math\'ematiques de Toulouse, 118 route de Narbonne, 31062 Toulouse, France}

\author{Alice Guionnet}
\email{alice.guionnet(at)ens-lyon.fr.}
\address{CNRS \& \'Ecole Normale Sup\'erieure de Lyon, Unit\'e de math\'ematiques pures et appliqu\'ees, 46 all\'ee d'Italie, 69364 Lyon Cedex 07, France}

\begin{document}

\begin{abstract}
 Consider an $n \times n$ Hermitian random matrix with, above the diagonal, independent entries with $\alpha$-stable symmetric distribution and $0 < \alpha < 2$. We establish new bounds on the rate of convergence of the empirical spectral distribution of this random matrix as $n$ goes to infinity. When $1 < \alpha < 2$ and $ p > 2$, we give vanishing bounds on the $L^p$-norm of the eigenvectors normalized to have unit $L^2$-norm. On the contrary, when $0 < \alpha < 2/3$, we prove that these eigenvectors are localized. 
 \end{abstract}

\maketitle

\section{Introduction}

We consider an array $(X_{ij})_{1 \leq i \leq j}$ of  i.i.d. real random variables and set, for $i > j$, $X_{ij} = X_{ji}$. Then, for each integer $n \geq 1$, we may define the random symmetric matrix:
$$X  = ( X_{ij} )_{1 \leq i , j \leq n}.$$
The eigenvalues of the matrix $X$ are real and are denoted by $\lambda_n (X) \leq \cdots \leq \lambda_1(X)$. In the large $n$ limit, the spectral properties of this matrix are now well understood as soon as $X_{ij}$ has at least two finite moments see e.g. \cite{Bai,BaiSil,BAG21,AGZ10,erdossurvey,taovusurvey4} for reviews, or \cite{johan,ESY10,ESY11,TV10,TV11} for recent results on universality. The starting point of this analysis is the Wigner's semi-circular law, which asserts that if the variance of $X_{ij}$ is normalized to $1$, then the empirical spectral measure
$$
\frac 1 n \sum_{i=1}^n \delta_{\lambda_i (X) /  \sqrt n } 
$$
converges almost surely for the weak convergence topology to 
the semi-circular law $\mu_2$ with support $[-2,2]$ and density $f_2 ( x) = \frac 1 {2 \pi} \sqrt { 4 - x ^2}$. As already advertised, many more properties of the spectrum are known. For example, if the entries are centered and have a subexponential tail, then, see  \cite{ESY09a, ESY09b}, for any $p >2$ and $\e > 0$,
$$
\max \left\{Ê\|Êv \|_p  : \hbox{ $v$ eigenvector of $X$ with $\|Êv \|_2 =1$}\right\}
$$
is $O ( n^{ 1 / p - 1/2 + \e })$, where  $\|Êv \|_p =  \left( \sum_{i=1} ^p |v_i|^p  \right)^{\frac 1 p}$. This implies that the eigenvectors are strongly delocalized. 

When the second moment is no longer finite, much  less is known and the picture is different. Let $0 < \alpha < 2$ and assume for simplicity that 
\begin{equation}\label{eq:tailX11}
\bP ( |X_{11} |\geq t ) \sim_{t \to \infty} t^{-\alpha}.
\end{equation}
Then, we are not anymore in the basin of attraction of Wigner's semi-circular law: now  the empirical spectral measure
$$
\frac 1 n \sum_{i=1}^n \delta_{\lambda_i (X)  /  n^{ 1 / \alpha }}  
$$
converges a.s. for the weak convergence topology to  a new limit law $\mu_\alpha$, see \cite{BG08} and also \cite{BDG09}, \cite{BCC}. It is known that $\mu_\alpha$ is symmetric, has full support, a bounded density $f_\alpha$ which is analytic outside a finite set of points. Moreover, $f_\alpha (0)$ has an explicit expression and as $x$ goes to $\pm \infty$, $f_\alpha ( x) \sim (\alpha / 2) |x|^{\frac \alpha 2 -1}$. Finally, as $\alpha$ goes to $2$, $\mu_\alpha$ converges  for the weak convergence topology to $\mu_2$. One of the difficulty of this type of random matrices is the lack of an exactly solvable model
 as in the Gaussian Unitary Ensemble or the Gaussian Orthogonal Ensemble in the finite variance case.

In the present paper, we give a rate of local convergence to $\mu_\alpha$ and investigate the behavior of the eigenvectors of $X$. In a fascinating article \cite{bouchaudcizeau}, Bouchaud and Cizeau have made some prediction for the eigenvectors of $X$. They argue that the situation is different for $0< \alpha < 1$ and $1 < \alpha < 2$. They quantify the localized nature of a vector $v$ with $\| v \|_2 =1$ by two scalars: $ \| v \|_4$ and $ \| v \|_1$. If $\| v \|_4 = o (1)$ the vector is said to be \emph{delocalized}, if $\| v \|_4 \ne o (1)$ but $\| v \|_1  \gg 1$ then $v$ is \emph{weakly delocalized} (we might also say weakly localized), while if $\| v \|_1  = O(1)$ then the vector is \emph{localized}. Now suppose that $v$ is an eigenvector of $n^{-1/\alpha} X$ associated to an eigenvalue $\lambda$.  For $1 < \alpha < 2$,  we have proved that all
 but $o(n)$ of the eigenvectors are  delocalized (this disproved the prediction of \cite{bouchaudcizeau}). Our result is in fact stronger as we can show
that the $L^\infty$ norm of these vectors go to zero, which insures that
the $L^p$ norm goes to zero for all $p>2$ and to infinity for all $p<2$
by duality.

For $0 < \alpha < 1$, Bouchaud and Cizeau predict that with high probability, if $|\lambda| < E_\alpha$ then $v$ is weakly delocalized, while $|\lambda| > E_\alpha$, $v$ is localized. It is reasonable to predict that $E_\alpha$ goes to $0$ as $\alpha \downarrow 0$  and goes to infinity as 
 $\alpha \uparrow 2$. It is not clear whether this threshold $E_\alpha$ depends on the choices of the norms $L^1$ and $L^4$ to quantify localization and delocalization. We are far from proving the existence of such a threshold within the spectrum.  Nevertheless, for $0 < \alpha < 2/3$, we have proved that there exists $E_\alpha >0$ such that if $|\lambda| \geq E_\alpha$  then a localization occurs : the mass of $v$ is carried by at most $ n^{1-\delta_\alpha}$ entries, for some $\delta_\alpha >0$.

This heavy-tailed matrix model is in some sense similar to the adjacency matrix of 
Erd\H{o}s-R\'enyi graphs with parameter $p/n$ since its entries are of order one
only with probability of order $1/n$. In the regime where $p$ is going to infinity faster than $n^{2/3}$, this adjacency matrices were shown to belong to the university class of Wigner  random matrices \cite{EKY11},\cite{EKYY11}. If $pn / ( \log n )^c$ goes to infinity for some
constant $c$,  the delocalization of eigenvectors was also proved in these articles. In the related model of the adjacency matrix of uniformly sampled $d$-regular graphs on $n$ vertices, the delocalization of eigenvectors has been studied in Dumitriu and Pal \cite{DumitriuPal}
and Tran, Vu and Wang \cite{tranvuwang}. It was also conjectured by Sarnak that as soon as $d \geq 3$, this model also belongs to the university class of Wigner  random matrices. %We may note finally that heavy-tailed random matrices are at the interface between random matrix theory and random Schr\"odinger operators, indeed there is a random self-adjoint operator which is, in some sense, the limit of the matrices $A$, see \cite{BCC}. 

\subsection{Main results}

Let us now be more precise. Throughout the paper, the array $(X_{ij})_{1 \leq i \leq j}$ will be real i.i.d. symmetric $\alpha$-stable random variable such that for all $t \in \bR$, 
$$
\bE  \exp ( i t X_{11}  ) = \exp ( - w_\alpha |t |^\alpha ),
$$
for some $0 < \alpha < 2$ and $w_\alpha = \pi /   ( \sin(\pi \alpha / 2)\Gamma(\alpha)) $. With this choice, the random variables $(X_{ij})$ are normalized in the sense that \eqref{eq:tailX11} holds. The assumption that the random variables follow an $\alpha$-stable law should not be a crucial for our results, it will however simplify substantially some proofs. We define the hermitian matrix
$$ A_n = a_n^{-1} X \quad \hbox{ with } \quad a_n = n^{1 / \alpha}. $$ 
The eigenvalues of the matrix $A$ are denoted by $\lambda_n ( A) \leq \cdots \leq \lambda_1 (A)$. The empirical spectral measure of $A$ is defined as
$$
\mu_A = \frac 1 n \sum_{i=1}^n \delta_{\lambda_i ( A)}  = \frac 1 n \sum_{i=1}^n \delta_{\lambda_i (X)  /  n^{ 1 / \alpha }}. 
$$
The resolvent of $A$ will be denoted by
$$R(z) = ( A - z ) ^{-1},$$
where $z \in \bC_+ = \{z \in \bC : \Im (z) >0\}$. The Cauchy-Stieltjes transform of $\mu_A$  is easily recovered from the resolvent: 
\begin{equation}\label{eq:resCS}
g_{\mu_A} ( z) =  \int  \frac {1}{x - z } \mu_A (dx)  = \frac 1 n \tr ( R (z ) ) .
\end{equation}

From \cite{BG08,BCC}, for any fixed interval $I \subset \bR$, a.s. as $n \to \infty$, 
\begin{equation}\label{eq:convmuA}
\frac { \mu_{A_n} ( I )}{ | I | }  - \frac { \mu_{\alpha} ( I ) }{ | I | } \; \rightarrow \; 0,
\end{equation}
where $| I |$ denotes the length of the interval $I$. As in \cite{ESY10,ESY11}, the opening move for proving statements about the eigenvectors of $A$ is to reinforce the convergence \eqref{eq:convmuA} for small intervals whose length vanishes with $n$. We will express our main results in terms of a scalar $\rho$ depending on $\alpha$:
$$
\rho = \left\{ \begin{array}{lcl}
\frac 1 2 & \hbox{ if } & \frac 8 5 \leq \alpha < 2 \\
\frac \alpha { 8 - 3 \alpha }  & \hbox{ if } & 1 < \alpha < \frac 8 5  \\
\frac {\alpha} { 2+ 3 \alpha }  & \hbox{ if } & 0 < \alpha \leq 1.
\end{array}  \right.
$$
The scalar $\rho$ depends continuously on $\alpha$ and is non-decreasing. Roughly speaking we are able to prove that the convergence \eqref{eq:convmuA} holds for  all intervals of size larger than $n^{ - \rho + o(1)}$. A precise statement is the following. 

\begin{theorem}[Local convergence of the empirical spectral distribution]\label{th:mainconvloc}
Let $0 < \alpha < 2$.  There exists a finite set $\cE_\alpha \subset \bR$ such that if $K \subset \bR \backslash \cE_\alpha  $ is a compact set and $\delta >0$, the following holds. There are constants $c_0, c_1 >0$ such that for all integers $n \geq 1$, if $I\subset K$ is an interval of  length $ |I| \geq c_1  n^{-\rho}  ( Ê\log  n)^2$, then
$$
 \left| \mu_{A} (I)   - \mu_\alpha ( I)  \right| \leq  \delta  |I |,
$$
with probability at least $1 - 2  \exp \left( - c_0  n  \delta^2 | I | ^2 \right)$.
\end{theorem}
In the forthcoming Theorem \ref{th:boundStieltjes}, we will give a slightly stronger form of Theorem \ref{th:mainconvloc}: we will allow the parameter $\delta$ to depend explicitly on $n$ and $| I |$ and the logarithmic correction in front of $n^{-\rho}$ will be reinforced. The proof of Theorem \ref{th:mainconvloc} will be based on estimates of the diagonal coefficients of the resolvent matrix $R(z)$ as $z = E + i \eta$ gets close to the real axis with $\eta = n^{ - \rho + o(1)}$. For technical reasons, we have only been able to establish \eqref{eq:convmuA} for intervals outside the finite set $\cE_\alpha$ which contains $0$. The same type of result should hold for all sufficiently large intervals. In Proposition \ref{prop:supNI}, we will give an upper bound on $\mu_{A} (I)  $ (i.e. a Wegner's estimate) which will be valid for all intervals of size larger than $n^{-  (\alpha + 2) / 4} $. The threshold $\rho\le \frac{1}{2}$ may be optimal, eventhough for Wigner's matrices it is simply one, since the spectral measure of  heavy tails random matrices fluctuates like $O(n^{-1/2})$
rather than like $O(n^{-1})$ for Wigner's matrices (see \cite{BGM, ShTi}).

Theorem \ref{th:mainconvloc} will have the following corollary on the delocalization of the eigenvectors.

\begin{theorem}[Delocalization of eigenvectors]\label{th:delocvect}
Let $1 < \alpha < 2$.  There exist a finite set $\cE_\alpha \subset \bR$ and a constant $c >0$ such that if $K \subset \bR \backslash \cE_\alpha  $ is a compact set, with probability tending to $1$, 
\begin{equation}\label{eq:delocvect}
\max \{ \| v_k \|_\infty : 1 \leq k \leq n  , \lambda_k (A) \in K\} \leq  n^{ -\rho (1 - \frac 1   \alpha)} ( \log n )^{c} , 
\end{equation}
where $v_1, \cdots, v_n$ is an orthogonal basis of eigenvectors of $A$ associated to the eigenvalues $\lambda_1 (A), \cdots, \lambda_n (A)$.
\end{theorem}

Notice that for $p > 2$, $\| v \|_p \leq \| v \|^{2/p}_2 \| v \|^{1 - 2/p}_\infty$. Hence, Theorem \ref{th:delocvect} implies that the $L^p$-norm of any eigenvector associated to an eigenvalue in $K$ goes in probability to $0$ as soon as $p > 2$. Similarly, from $\|v \|^2_2 \leq \| v \|_1 \| v \|_\infty$, we have a lower bound of order $n^{ \rho (1 - \frac 1   \alpha) + o(1)}$ on the $L^1$-norm of the eigenvectors. Note that our estimate becomes trivial as $\alpha \downarrow 1$ and give upper bound of order $n ^ { - 1 / 4 + o (1)}$ as $\alpha \uparrow 2$. For any $\kappa > 0$, in the proof of Theorem \ref{th:delocvect}, we will see that by increasing suitably $c$, the probability that the event \eqref{eq:delocvect} holds is at least $1 - n^{-\kappa}$.

We now present our result on localization of eigenvectors. We are not able to prove localization for all eigenvectors but only for "typical" eigenvectors associated to an eigenvalue in a small interval. More precisely, we consider  $v_1, \cdots, v_n$ an orthogonal basis of eigenvectors of $A$ associated to the eigenvalues $\lambda_1 (A) , \cdots, \lambda_n (A)$. If $I$ is an interval of $\bR$, we define $\Lambda_I$ as the set of eigenvectors whose eigenvalues are in $I$. Then, if $\Lambda_I$ is not empty, for $ 1 \leq i \leq n$, set
$$
W_I  (i) =  \frac { n} {| \Lambda_I | } \sum_{v \in \Lambda_I} \langle v , e_i \rangle ^2,
$$
where, throughout this paper, 
\begin{equation}\label{eq:defNI}
|\Lambda_I | = n \mu_A ( I) = N_I
\end{equation} is the cardinal of $\Lambda_I$. $W_I(i) / n$ is the average amplitude of the $i$-th coordinate of eigenvectors in $\Lambda_I$. By construction, the average amplitude of $W_I$ is $1$:
$$
\frac 1 n \sum_{i =1} ^ n  W_I  (i) =  1.
$$
If the eigenvectors in $\Lambda_I$ are localized and $I$ contains few eigenvalues, then we might expect that for  some $i$, $W_I  (i) \gg 1$, while for most of the others $W_I(i) = o(1)$. More quantitatively, fix $0 < \delta < 1$ and assume that for some $0 < \kappa < 1$ and $0 < \e < 1$
$$
\frac 1 n \sum_{i =1} ^ n  W_I  (i)^\kappa \leq  \e,
$$
then, setting
 $J=\{ i : W_I  (i) \geq  (\delta^{-1} \e )^{ - \frac 1  { 1 - \kappa} } \}$, we find
$$
\frac 1 n \sum_{i \in J} ^ n  W_I  (i) \geq 1 - \delta.
$$
In particular, all but a proportion $\delta$ of the mass of $W_I$ is carried by a set $J$ of cardinal at most %Al
$|J| \leq n ( \delta^{-1} \e )^{ \frac 1  { 1 - \kappa} }$. If $\e$ goes to $0$ with $n$, this indicates a localization phenomenon.  With this in mind, we can state our result.

\begin{theorem}[Localization of eigenvectors]\label{th:locvect}
Let $0 < \alpha < 2/3$, $0 < \kappa < \alpha /2$ and $\rho$ be as above. There exists $E_{\alpha,\kappa}$ such that for any compact $K \subset [-E_{\alpha,\kappa},E_{\alpha,\kappa}]^c$, there are constants $c_0,c_1 >0$ and for all integers $n \geq 1$, if $I\subset K$ is an interval of  length $ |I| \geq  n^{-\rho}  ( Ê\log  n)^2$,
$$
\frac 1 n \sum_{i =1} ^ n  W_I  (i)^{\frac \alpha 2} \leq c_1 |I|^ \kappa,
$$
with probability at least $1 -   2 \exp \left( - c_0  n | I | ^4 \right)$.
\end{theorem}

This result is interesting when $I = [ E - n^{-\rho+ o(1)}, E + n^{-\rho+ o(1)} ]$ is a small neighborhood around some large $E$. Then it shows that  for any $0 < \kappa < \alpha /2$, the mass of the eigenvectors around $E$ is concentrated around order $n^{ 1 -  2 \rho \kappa / ( 2 - \alpha ) }$  entries as long as $|E|$ is large enough.
 The proof of Theorem \ref{th:locvect} will be done by showing that 
\begin{equation}\label{eq:locvectres}
\frac 1 n \sum_{i=1} ^ n  \left(  \Im R(E + i \eta)_{ii} \right)^ {\frac \alpha 2 } 
\end{equation}
vanishes to $0$ if $\eta =n^{-\rho + o(1)} $ even though that  
$$
\frac 1 n \sum_{i=1} ^ n   \Im R(E + i \eta)_{ii} 
$$ 
stays bounded away from $0$. This phenomenon will have an interpretation in terms of a random self-adjoint operator introduced in \cite{BCC} which is, in some sense, the limit of the matrices $A$. We will prove that the imaginary part of its resolvent vanishes at $E + i \eta$, with $\eta = o(1)$ and $|E|$ large enough, while its expectation does not, see Theorem \ref{th:unicityRDE}.  Note that if $0 < \eta \ll n^{-1}$, then we necessarily have that for almost all $E$, $\Im R(E + i \eta)_{ii}$ converges to $0$.  The fact that our estimate $ |I| \geq  n^{-\rho}$ gets worse as $\alpha$ goes to $0$ is an artifact of the proof : our rate of convergence of $R(E + i \eta)_{ii}$ to its limit gets worse as $\alpha$ gets small. It is however intuitively clear that the localization should be stronger when $\alpha$ is smaller.  However, in the forthcoming Theorem \ref{th-finn}, we will prove that, for any $0 < \alpha < 2/3$, the expression \eqref{eq:locvectres} goes to $0$ if $\eta =n^{-\frac 1 6} $. Finally, it is worth to notice that computing the fractional moments of the resolvent matrix as a way to prove localization is already present in the literature on random Schr\"odinger operators, see e.g. \cite{MR1868998}.

The remainder of the paper is organized as follows. In Section \ref{sec:Wegner} we establish general upper bounds on  $N_I$  defined by  \eqref{eq:defNI}.  Section \ref{sec:apfp} contains the proof of Theorem \ref{th:mainconvloc}. Section \ref{sec:delocvect} is devoted to the proof of Theorem \ref{th:delocvect}.  The arguments developped in these two sections  are based on  ideas from  the seminal work of Erd\H{o}s, Schlein, Yau (see e.g. \cite{ESY09a, ESY09b, ESY10}) concerning 
Wigner's matrices with enough moments, as well as on the analytic approach of heavy tailed matrices as initiated in \cite{BG08, BDG09}. However, there was a technical gap due to the lack of  concentration inequalities, 
as well as of simple loop equations, that hold for finite second moment Wigner matrices. A few of the required  new  estimates due to the specific nature of heavy tailed matrices
 are contained in the appendix on concentration inequalities and  stable laws.
In Section \ref{sec:RDE} we prove Theorem \ref{th:locvect}, which is based on the representation of the asymptotic
spectral measure given in \cite{BCC} and a new  fixed point argument which allows
to prove the vanishing    of the imaginary part of the resolvent in the regime $\alpha\in (0,2/3)$.

 The whole article is quite technical,  but hopefully shall be useful for 
 further local study of the spectrum of random matrices which do not
 belong to the universality class of Wigner's semi-circle law.

\section{Upper bound on the spectral counting measure}

\label{sec:Wegner}

For $\eta > 0$ and $ E \in \bR$, we set  $I = [E - \eta , E + \eta]$. The goal of this section 
is to provide a rough 
 upper bound on  $N_I$ when $\eta$ is large enough, where $N_I$ was defined by  \eqref{eq:defNI}. Let 
\begin{equation}\label{eq:defgamma}
 \gamma = \left( \frac 1 2 + \frac 1 \alpha\right)^{-1}.
 \end{equation}

\begin{proposition}[Upper bound on counting measure]\label{prop:supNI}
Let $0 < \alpha < 2$. 
There exist $c,c'>0$ depending only on $\alpha$ such that if $\eta \geq n^{- \frac{\alpha + 2}{ 4}}$, then, for all integers $n$,  for all $t\ge c$,
$$
\bP\left( N_I \ge  t  n \eta^{\gamma} \right)\le
c\exp(-c' t^{\frac{2}{2-\alpha}})+ 2n\exp(-c' t^{\frac{4}{2+\alpha}})\,.
$$
\end{proposition}

This bound will later be refined in the forthcoming Proposition \ref{prop:boundStieltjes}. The upper bound on the  eigenvalues counting measure  implies an upper bound for the trace of the resolvent.  
\begin{corollary}[Trace of resolvent] \label{cor:supNI}
Let $0 < \alpha < 2$ and $z = E + i \eta \in \bC_+$.  There exists $c >0$ depending only on $\alpha$ such that if $\eta \geq n^{- \frac{\alpha + 2}{ 4}}$, then, for all integers $n$, 
$$
 \bE  \tr R(z) R^*(z)  \leq c(\log n)^{\frac{2+\alpha}{4}}  n  \eta^{- \frac 4 {2 + \alpha}}.
$$
\end{corollary}
\begin{proof}
By the spectral theorem  
$$
 \tr R(z) R^*(z) = \sum_{j=1} ^n |\lambda_j (A)- z|^{-2}.  
$$
Let $\eta$ be the imaginary part of $z$.
Define $I_0 = [E - \eta , E + \eta]$ and  for integer $k \geq 0$, $I_{k+1}  = [E - 2^{k+1} \eta , E + 2^{k+1} \eta]  \backslash [E - 2^{k} \eta , E + 2^{k} \eta] $. By construction, if $\lambda_j (A) \in I_k$ then $|\lambda_j (A)- z|^{-2} \leq 2^{- 2k + 1} \eta^{-2}$. Therefore, if $N_{I_k} = n \mu_{A} (I_k)$ is the number of eigenvalues in $I_k$, then 
\begin{equation}\label{toto}
 \tr R(z) R^*(z) \leq \sum_{k\geq 0 }    2^{- 2k + 1} \eta^{-2} N_{I_k}.
\end{equation} 
We write $I_k = I^+_ k \cup I^-_k$, where $I^\pm _k = \bR_{\pm} \cap I_k$. To estimate $\bE[ N_{I_k^\pm}]$
we apply 
 Proposition \ref{prop:supNI}. Namely it yields that
for each interval $I$ of length $\eta\ge n^{-\frac{\alpha+2}{4}}$,  for any $\tau\ge c$
\begin{eqnarray*}
\bE[ N_{I}]&=&\int_0^\infty \bP\left(N_{I}\ge t\right) dt \\
&\le& \tau n \eta^\gamma+ n\eta^\gamma \int_\tau^\infty  \bP\left(N_{I}\ge t n \eta^\gamma\right) dt\\
&\le& n \eta^\gamma\left(\tau+\int_\tau^\infty\left(\exp(-c' t^{\frac{2}{2-\alpha}})+ 2n\exp(-c' t^{\frac{4}{2+\alpha}})\right)\right)dt\\
&\le & n \eta^\gamma\left(\tau + c_0(1+n\exp(-c' \tau^{\frac{4}{2+\alpha}}))\right),
\end{eqnarray*}
for some finite constant $c_0 > 0 $. Therefore, taking $\tau$ of order $(\log n)^{\frac{2+\alpha}{4}}$, we deduce
that there exists some finite constant $c_1 > 0 $ such that 
$$\bE[ N_{I}]\le  c_1 (\log n)^{\frac{2+\alpha}{4}} n \eta^\gamma\,.$$
Therefore, we deduce from \eqref{toto} that
$$
\bE  \tr R(z) R^*(z) \leq   2c_1 (\log n)^{\frac{2+\alpha}{4}} \sum_{k\geq 0 }   2^{- 2k} \eta^{-2}  n (\eta 2^k)^{\gamma} \leq 2c_1(\log n)^{\frac{2+\alpha}{4}}  n   \eta^{- \frac 4 {2 + \alpha}} \sum_{k \geq 0} 2^{- \frac {4k} {2 + \alpha}} .
$$ 
\end{proof}
The rest of this section is devoted to the proof of Proposition \ref{prop:supNI}.

\subsection{A geometric upper bound}

In this paragraph, we recall a general upper bound for $N_I$ that is due to Erd\H{o}s-Schlein-Yau \cite{ESY10}, namely
if we let $A^{(k)}$ be the principal minor matrix of $A$ where the $k$-th row and column have been removed, $W^{(k)}$ be the vector space generated 
by the eigenvectors of $A^{(k)}$ correponding to eigenvalues at distance
greater than $\eta$ from $E$, we have 
\begin{equation}\label{eq:ESY}
N_I \leq 4 \eta^2 \, a_n^2 \, \sum_{k =1}^n \dist(X_k,W^{(k)})^{-2}\,.
\end{equation}
Let us prove \eqref{eq:ESY}. We start with the resolvent formula, 
\begin{equation}\label{eq:resolventformula}
R(z)_{kk}  =  - \left( z - a_n^{-1} X_{kk}  +   a_n ^{-2} \langle  X_k , R^{(k)}X_k \rangle  \right)^{-1},
\end{equation}
where $X_k = ( X_{k1}, \cdots ,X_{k k-1}, X_{k k+1}, \cdots , X_{kn}) \in \bR^{n-1}$
and  $R^{(k)} = ( A^{(k)} - z ) ^{-1}$. 

Identifying the real and imaginary part, we get, using the fact that $z$ and the eigenvalues of $R^{(k)}$ are in $\bC_+$, 
\begin{eqnarray*}
\Im R(z)_{kk}  &\leq &\left( \Im \left(z - a_n^{-1} X_{kk} +    a_n ^{-2} \langle  X_k , R^{(k)}X_k  \rangle  \right) \right)^{-1} \\
& \leq & a_n ^{2}    \langle  X_k , \Im  R^{(k)}X_k  \rangle ^{-1}.
\end{eqnarray*}
Let  $(\lambda_{i}^{(k)})_{1 \leq i \leq n-1}$ and $(u_{i}^{(k)})_{1 \leq i \leq n-1}$ be the eigenvalues and eigenvectors of $A^{(k)}$. Choosing $z = E + i \eta$, we have the spectral decomposition
$$
\Im  R^{(k)} = \sum_{i = 1} ^{n-1} \frac{\eta}{ ( \lambda_i^{(k)} - E)^2 + \eta^2 } u_i^{(k)} {u_i^{(k)}}^*. 
$$
If $ |\lambda_i^{(k)} - E| \leq \eta$, then  $ \frac{\eta}{ ( \lambda_i - E)^2 + \eta^2 } \geq 1 / ( 2 \eta)$. Therefore,
we deduce
$$
\Im R (z )_{kk} \leq 2 \eta a_n^{2} \left(\sum_{i=1}^{n-1} 1_{|\lambda_i^{(k)} - E| \leq \eta } \langle  X_k , u_i ^{(k)} \rangle  ^2 \right)^{-1}\,.
$$
We rewrite the above expression as
\begin{equation}\label{lkj}
\Im R (z )_{kk} \leq 2 \eta a_n^{2} \dist^{-2}(  X_k, W^{(k)} ),
\end{equation}
where $$W^{(k)} = \mathrm{vect}\left\{ u_{i}^{(k)} : 1 \leq i \leq n-1 , \lambda_i^{(k)} \notin [E - \eta , E + \eta] \right\} .$$ 
Since $I = [E - \eta , E + \eta]$, the inequalities \eqref{lkj} and 
$$
\tr \, \Im R(z) = n  \int \frac {\eta}{(E - \lambda)^2 + \eta^2} \mu_A ( dx) \geq   n \int_I  \frac {\eta}{(E - \lambda)^2 + \eta^2} \mu_A ( dx) \geq \frac{N_ I }{2 \eta}
$$
 give \eqref{eq:ESY}.
 We set 
$$N^{(k)}_I  =  | \{1 \leq i \leq n - 1 : \lambda^{(k)}_i \in I \} |= n  - 1 -  \mathrm{dim} (W^{(k)})  = (n -1) \mu_{A^{(k)}} ( I).$$ From Weyl interlacement theorem,
\begin{equation}\label{eq:weylNI}
N_I - 1 \leq N^{(k)}_I =  n  - 1 -  \mathrm{dim} (W^{(k)})  \leq N_I +1. 
\end{equation}

\subsection{Proof of Proposition \ref{prop:supNI}.}

We note that up to increasing the constant $c$ and $1/c'$, it is sufficient to prove the proposition only for all $\eta \geq c_1 n^{- \frac{\alpha + 2}{ 4}}$ for some $c_1$ (indeed, $N_{I'} \leq N_I$ if $I' \subset I$ and if $N_I \leq t n | I |^\gamma$ then $N_{I'} \leq  t n |I'|^\gamma  (|I | / |I' |Ê )^\gamma $). 

In the sequel we denote in short  $\dist_k = \dist(  X_k , W^{(k)} )$. From \eqref{eq:ESY}-\eqref{eq:weylNI}, we write 
\begin{equation}\label{eq:NIt}
N_I \leq \ind_{ N_I \leq \lfloor t n \eta^{\gamma}\rfloor } t n \eta^{\gamma} +    4 \eta^2 a_n^2 \, \sum_{k =1}^n \dist^{-2}_k  \ind_{ N^{(k)}_I \geq \lfloor t n \eta^{\gamma}\rfloor} \,.
\end{equation}
We have $ \dist^2 _k  = \langle X_k , P_k X_k \rangle$, where $P_k$ is the orthogonal projection on  $W^{(k)}$. We note that $X_k$ is independent of $W^{(k)}$. From Lemma \ref{le:formulealice},  there exists a positive $\alpha/2$-stable random variable $S_k$ and a standard Gaussian vector $G_k$, independent from $S_k$,  such that 
$$
\dist^2 _k = \|P_k G_k \|_\alpha ^2 S_k. 
$$
Note that  $\eta \geq c n^{- \frac{\alpha + 2}{ 4}}$ is equivalent to $n \eta^{\frac{2\gamma}{\alpha}} \geq c^{\frac{2\gamma}{\alpha}}$.
By Corollary \ref{cor:concnorm} (applied to $A=P_k$), there exists universal constants
$C,\delta$ so that if $N^{(k)}_I\ge t c^{-\gamma}
 n \eta^{\gamma}  \geq t n^{1-\frac{\alpha}{2}}$
for $t \ge C c^{\gamma}$, 
 with probability at least 
\begin{equation}\label{eq:Fn1}
1 - 2 \exp ( -  \delta     n  ( t  \eta )^{\frac {2 \gamma }{ \alpha}} /2) \geq 1 - 2 \exp ( - c_0    t  ^{\frac{4}{2 + \alpha}} ) 
\end{equation}
we have,
$$
\|P_k G_k \|_\alpha \geq \delta  \left( t n \eta^\gamma \right)^{\frac 1 \alpha}. 
$$ 
Hence, if $F_n$ denotes the event that $N_I >  \lfloor t n \eta^{\gamma} \rfloor$ and for all $k$, $\|P_k G_k \|_\alpha \geq \delta  \left( t n \eta^\gamma \right)^{\frac 1 \alpha}$, we have from \eqref{eq:NIt}
$$
N_I \ind_{F_n}  \leq  4 \eta^2  \delta^{-2}  \left( t  \eta^\gamma \right)^{-\frac 2 \alpha} \, \sum_{k =1}^n S_k^{-1}.
$$
With our choice of $\gamma$, $2  - 2  \gamma / \alpha = \gamma$, hence, with $c_1 = 4 \delta^{-2}$, we deduce
$$
N_I \ind_{F_n}  \leq  c_1 n \eta^\gamma t^{-\frac 2 \alpha}   \left( \frac 1 n \sum_{k =1}^n S_k^{-1} \right).
$$
The variables $(S_k)_{1 \leq k \leq n}$ have the same distribution but are correlated. Nevertheless, note that the function $x \mapsto \exp ( x^\delta )$ is convex on $[b_\delta, +\infty) $ with $b_\delta  = 0$ if $\delta \geq 1$, and $b_\delta = Ê(1/\delta-1)^{1/\delta}$ if $0 \leq \delta \leq 1$. Hence from the Jensen inequality, 
\begin{equation*}\label{sumi}
\exp \left( \frac 1 n \sum_{k =1}^n S_k^{-1} \right)^{\delta } \leq \exp \left( \frac 1 n \sum_{k =1}^n   S_k^{-1} \vee b_\delta \right)^{\delta } \leq   \frac 1 n \sum_{k =1}^n  \exp (   S_1 ^{-\delta} \vee b^\delta_\delta). 
 \end{equation*}
In particular for every $c_2 > 0$,
\begin{equation*}\label{sumi}
 \bE  \exp \left\{c_2 \left( \frac 1 n \sum_{k =1}^n S_k^{-1} \right)^{\frac{\alpha}{2 - \alpha} }\right\} \leq c_3 \bE \exp \left\{ c_2 \left( S_1 ^{-\frac{\alpha}{2 - \alpha} } \right) \right\}, 
 \end{equation*}
 where $c_3 = \exp ( c_2 b_{\alpha / (2 - \alpha) } ^{\alpha / (2 - \alpha)} )$. By  Lemma \ref{le:tailS}, for $c_2$ small enough, the above is finite. Thus, from the Markov inequality, for some constants $c_4,c_5 >0$,
$$
\bP \left( N_I \ind_{F_n}  > t  n \eta^\gamma \right) \leq \bP \left( \frac 1 n \sum_{k =1}^n S_k^{-1}  > c_1^{-1} t^{  \frac 2 \alpha} \right) \leq c_4 \exp ( - c_5 t^{ \frac{2}{2 - \alpha}}). 
$$
Therefore, by \eqref{eq:Fn1}, we deduce 
\begin{eqnarray*}
\bP \left( N_I > t  n \eta^\gamma \right)& \leq&\bP \left( \{N_I  > t  n \eta^\gamma\}\cap F_n^c \right) +  \bP \left( N_I \ind_{F_n}  > t  n \eta^\gamma \right) \\
&\le& 2n \exp ( - c_0    t  ^{\frac{4}{2 + \alpha}} )  + c_4 \exp ( - c_5 t^{ \frac{2}{2 - \alpha}})
\end{eqnarray*}
which completes the proof of the proposition.

\section{Local convergence of the spectral measure}
\label{sec:apfp}
To prove the local convergence of the spectral measure, we shall prove that
an observable of the resolvent satisfies nearly a fixed point equation, which also
entails an approximate equation for the resolvent. Such an equation was already derived in \cite{BG08,BDG09}
but the error terms are here carefully estimated. This step will be crucial  to obtain, in the second part
of this section,  a rate of convergence of the Stieltjes transform of the spectral measure toward its limit. 
The range of convergence  will  be first derived roughly, and then improved  for $\alpha>1$ thanks
to bootstraps arguments.

\subsection{Approximate fixed point equation}

The observables we shall be interested in 
will be 
\begin{equation}\label{defXY}
Y(z) := \bE [( - i R(z)_{11}  ) ^\frac{\alpha}{2} ] \quad \hbox{and } \quad 
X(z) := \bE [- i R (z)_{11}   ]\,.
\end{equation}
(For $1 \leq k , \ell \leq n$, we will write indifferently $R_{ k \ell} (z)$ or $R (z)_{ k \ell}$). For $\beta \in [0,2]$, we define $\cK_\beta = \{z \in \bC : | \arg (z) | \leq \frac {\pi \beta}{2} \}$. By construction $ - i R_{11} ( z) \in \cK_{1}$ for $z\in \mathbb C_+$, so that 
$ Y(z)  \in \cK_{\alpha / 2}$ and  $
X(z) \in \cK_{1}.
$
On $\cK_{\alpha / 2}$, we may define the entire functions 
$$
\varphi_{\alpha,z} (x) =  \frac{1}{\Gamma(\frac \alpha 2)} \int_0^\infty t^{\frac \alpha 2 - 1}e^{itz} e^{- \Gamma ( 1 - \frac{\alpha }{2} ) t^{\frac \alpha 2} x} dt 
$$
and 
$$
\psi_{\alpha,z} (x) =   \int_0^\infty e^{itz} e^{- \Gamma ( 1 - \frac{\alpha }{2} ) t^{\frac \alpha 2} x} dt. 
$$

For further use, we define, with the notation of \eqref{eq:resolventformula},
\begin{equation}\label{defM}
M_n (z) =\frac 1 {n-1} \bE \tr  \left\{  R^{(1)} (z) (R^{(1)}(z) )^*\right\} .
\end{equation}
Note that, writing explicitly the dependence in $n$,  $R_n^{(1)} = ( A_n^{(1)} - z ) ^{-1}$ and $\frac{a_n}{a_{n-1}} A_n^{(1)}$ has the same distribution than $A_{n-1}$. We may thus apply Corollary \ref{cor:supNI} to  $R^{(1)}$ (it can be checked the difference between $a_{n}$ and $a_{n-1}$ is harmless). For some constant $c>0$, we therefore have the upper bound for all $z = E + i \eta$ and $\eta \geq n^{ -\frac{\alpha+2}{4}}$,
\begin{equation}\label{eq:boundMn}
M_n (z) \leq c(\log n)^{\frac{2+\alpha}{4}}\eta^{-\frac{4}{2+\alpha}}.
\end{equation}
The main result of this paragraph is the following approximate fixed point  equations.
\begin{proposition}[Approximate fixed point equation]\label{prop:fixpoint}
Let $0 < \alpha < 2$ and $z = E + i \eta \in \bC_+$. There exists $c >0$ such that if $n^{ -\frac{\alpha+2}{4}} \leq \eta \leq 1$ and 
\begin{equation}\label{eq:defeps}
\e =  \eta^{-\alpha \wedge 1} n ^{-\frac{1}{\alpha \vee 1}}   + \frac{( \log n) \ind_{\alpha =1} }{\eta n}
   + \left( \eta^{-1} \sqrt{ \frac{M_n(z)}{n}  } \right)^{1 \wedge \alpha}  \left( 1 +    ( \log  n  )  \ind_{1 < \alpha \leq 4/3} +  ( \log  n  )^2  \ind_{0 < \alpha \leq 1} \right),
\end{equation}
  then, for any integer $n \geq 1$, 
$$
 |Y ( z) - \varphi_{\alpha,z}( Y ( z) )|\leq  c \eta^{-\frac \alpha 2} \e +  c \eta^{-\frac \alpha 2} n^{-\frac \alpha 4},
$$
and 
$$
 |X ( z) - \psi_{\alpha,z}( Y ( z) )|\leq  c \eta^{-1} \e +  c \eta^{-\frac \alpha 2} n^{-\frac \alpha 4} .
$$
\end{proposition}

We note that we could use the bound \eqref{eq:boundMn} to get an explicit upper bound on $\e$. In the forthcoming Proposition \ref{prop:boundStieltjes} we will however improve this bound for some range of $\eta$. 

In the first step of the  proof, we compare $Y(z)$ with an expression which gets rid of the off-diagonal terms of $R^{(1)}$ in \eqref{eq:resolventformula}. More precisely, with the notation of  \eqref{eq:resolventformula}, we define 
$$
I (z) : =  \bE \left[\left( (- i  z ) +  a_n ^{-2} \sum_{k=2}^n   X_{1k}^2  ( - i R^{(1)}_{kk} ) \right) ^{-\frac{ \alpha}{2} }\right]\; \in \cK_{\alpha / 2},
$$
and similarly, 
$$
J (z) : =  \bE \left[\left( (-i  z ) +    a_n ^{-2} \sum_{k=2}^n   X_{1k}^2  ( - i R^{(1)}_{kk}  )  \right) ^{-1}\right]\; \in \cK_{1}.
$$
We start with a technical lemma. 
\begin{lemma}[Off-diagonal terms]\label{le:offdiag}
Let $B$ be an hermitian matrix with resolvent $G = ( B - z)^{-1}$. For any $0 < \alpha < 2$, there exists a constant $c = c(\alpha) >0$ such that for $n \geq 2$,
$$
\bP \left( \left| a_n ^{-2}  \sum_{2 \leq k \ne \ell \leq n } X_{1k} X_{1\ell}  G_{k\ell}\right| \geq  \sqrt{ \frac {\tr  ( GG^* ) } { n^2 } }  t \right)  \leq c t^{ - \alpha }   \log \left(n \left( 2 \vee t   \right)  \right)\log   \left( 2 \vee   t   \right)    , 
$$
and if $1 < \alpha < 2$, 
 $$
 \bE \left| a_n ^{-2}  \sum_{2 \leq k \ne \ell \leq n } X_{1k} X_{1\ell}  G_{k\ell}\right| \leq c \sqrt{ \frac {\tr  ( GG^* ) }{n^2 }}   \left(1 + \ind_{1<  \alpha \leq 4/3}  \log n   \right). 
 $$
\end{lemma}

\begin{proof}
Let $ 0 < \alpha \leq 1$. We use a decoupling technique : from  \cite[Theorem 3.4.1]{MR1666908}, there exists a  universal constant $c > 0$ such that 
\begin{eqnarray}
\bP  \left( \left| a_n ^{-2}  \sum_{2 \leq k \ne \ell \leq n } X_{1k} X_{1\ell}  G_{k\ell}\right| \geq t  \right) & \leq & c  \, \bP  \left( \left| a_n ^{-2}  \sum_{2 \leq k \ne \ell \leq n } X_{1k} X'_{1\ell} G_{k\ell}\right| \geq  t / c   \right)  \label{eq:Gdecouptail} \\
& \leq  & c  \, \bP  \left( \left| a_n ^{-2}  \sum_{2 \leq k \ne \ell \leq n } X_{1k} X'_{1\ell} \Re( G_{k\ell} ) \right| \geq  t / {2c}   \right) \nonumber  \\
& & \quad \quad +  c  \, \bP  \left( \left| a_n ^{-2}  \sum_{2 \leq k \ne \ell \leq n } X_{1k} X'_{1\ell} \Im( G_{k\ell} ) \right| \geq  t / {2c}   \right),\nonumber 
\end{eqnarray}
where $X'_1$ is an independent copy of $X_1$.  From the stable property of $X'_1$, we deduce that 
\begin{align}\label{eq:Redecoup}
  \sum_{2 \leq k \ne \ell \leq n } X_{1k} X'_{1\ell}  \Re ( G_{k\ell} )    & \stackrel{d}{=}    X'_{11}   \left( \sum_{\ell }  \left|  \sum_{k\ne \ell}     X_{1k}  \Re ( G_{k\ell} ) \right| ^ \alpha  \right)^{\frac 1 \alpha},
\end{align}
and similarly for the imaginary part. From the stable property of $X_1$, 
\begin{align}\label{eq:Redecoup2}
\sum_{\ell }  \left|  \sum_{k\ne \ell}     X_{1k}  \Re ( G_{k\ell} ) \right| ^ \alpha  \stackrel{d}{=} \sum_{\ell }  | \hat X_\ell | ^ \alpha  \sum_{k\ne \ell}  |  \Re ( G_{k\ell} ) | ^ \alpha ,
\end{align}
where $(\hat X_\ell)_\ell $ is a random vector whose marginal distribution is again the law of $X_{11}$ (note however that the entries of  $(\hat X_\ell)_\ell $ are correlated). Let $\rho_\ell = \sum_{k\ne \ell}  |  \Re ( G_{k\ell} ) | ^ \alpha$ and $\rho = \sum_\ell \rho_\ell$. 
\iffalse However, for any $\lambda > 0$, by \cite[Corollary 6.2.5]{MR1666908}, we have
\begin{align*}
\bE \exp \left\{  - \lambda \sum_{\ell }   |X'_{11} | ^ \alpha  | \hat X_\ell | ^ \alpha  \rho_\ell \right\}   & \leq  \sqrt{  \bE \exp \left\{   - 2 \lambda \sum_{\ell }   |X'_{1\ell} | ^ \alpha  |  X_{1 \ell}  | ^ \alpha   \rho_\ell \right\} } . 
\end{align*}
On the right hand, there is a sum of independent variables, hence 
\begin{align*}
\bE \exp \left\{  - \lambda \sum_{\ell }   |X'_{11} | ^ \alpha  | \hat X_\ell | ^ \alpha  \rho_\ell \right\}   & \leq  \prod_{\ell}  \sqrt{  \bE \exp \left\{   - 2 \lambda   |X'_{1\ell} | ^ \alpha   |  X_{1 \ell}  | ^ \alpha   \rho_\ell \right\} } . 
\end{align*}
For $s \leq 1 / \lambda$, we deduce from $e^{ -x } \leq 1 -  x /2 $ on $[0,1]$ and the union bound that
$$
\bE  e^{  - \lambda   |X'_{1\ell} | ^ \alpha   |  X_{1 \ell}  | ^ \alpha  } \leq 1 - \frac \lambda 2 \bE [  |X'_{1\ell} | ^ \alpha   |  X_{1 \ell}  | ^ \alpha \ind_ { |X'_{1\ell} |^\alpha |X'_{1\ell} |^\alpha    \leq s } ] +  \bP ( | X_{1\ell}|^ \alpha  |X'_{1\ell} |^\alpha  \geq s ) 
$$
\fi
For $ s \geq 2$ to be chosen later, we define $Y_\ell = \hat X_{\ell}  \ind ( |\hat X_{\ell}| \leq s a_n )$ and $Y'_1 = X'_{11}  \ind ( | X'_{11}| \leq s)$. It is straightforward to check that 
$$
 \bE | Y_\ell |^ \alpha   \leq c   \log ( s^\alpha n  )  \quad \hbox{ and } \quad  \bE | Y'_1 |^ \alpha   \leq c   \log ( s  ).
$$
 Hence, from \eqref{eq:Redecoup}-\eqref{eq:Redecoup2}
\begin{eqnarray*}
\bP  \left( |X'_{11} |  \left(  \sum_{\ell }  | \hat X_\ell | ^ \alpha  \rho_\ell \right) ^{\frac 1 \alpha}  \geq t  \right) & \leq & \bP  \left(  |Y'_{1} |   \left(  \sum_{\ell }  | Y_\ell | ^ \alpha  \rho_\ell \right) ^{\frac 1 \alpha}   \geq t   \right) \\
&& \quad \quad + \; \bP \left( \max_{\ell}  |\hat X_{\ell}|    \geq s a_n \right)  + \bP \left( |\hat X'_{11}|    \geq s  \right) \\
& \leq & c s^{-\alpha} + \frac{c  \rho  \log (s^ \alpha n ) \log(s) }{t^\alpha},
\end{eqnarray*}
where we have used the Markov inequality and the union bound 
$$\bP ( \max_{\ell}  |\hat X_{\ell}|    \geq s a_n  )  \leq  n \bP ( |X_{11}|\geq s a_n  )  \leq c s^{-\alpha}.$$ 
We choose $s = 2 \vee  ( t / \rho ^{1 / \alpha})$, we find that 
$$
\bP  \left( |X'_{11} |  \left(  \sum_{\ell }  | \hat X_\ell | ^ \alpha  \rho_\ell \right) ^{\frac 1 \alpha}  \geq t \rho ^{1 / \alpha}  \right)  \leq  \frac{c   \log ( (2 \vee  t )n ) \log (  2 \vee   t) }{t^\alpha  } .  
$$
The same statement holds for $\Im (G)$.  To sum up, we deduce from \eqref{eq:Gdecouptail} that 
$$
\bP \left( \left| a_n ^{-2}  \sum_{2 \leq k \ne \ell \leq n } X_{1k} X_{1\ell}  G_{k\ell}\right| \geq a_n ^{-2}  \rho^{  1 / \alpha}  t \right)  \leq c t^{ - \alpha }   \log \left(n \left( 2 \vee t   \right)  \right)\log   \left( 2 \vee   t   \right)  .
$$
Then, the first statement of the lemma follows H\"older's inequality which asserts that 
\begin{align}\label{eq:holderdecoup}
a_n ^{-2}  \rho^{1  / \alpha}  &  \leq   n^{ - \frac 2 \alpha}    \left( \sum_{\ell} \sum_{k }  |    G_{k\ell}   | ^ \alpha \right)^ { \frac 1 \alpha } 
 \leq  n^{ - \frac 2 \alpha}    \left( \sum_{\ell} \sum_{k }  |    G_{k\ell}   | ^ 2 \right)^ { \frac 1 2} n^{\frac 2 \alpha -1} 
= \sqrt{ \frac {\tr  ( GG^* ) }{n^2 }},
\end{align}
(recall  that $G_{k\ell} = G_{\ell k}$). 

Now, assume $\alpha > 1$. By integrating the above bound, we find easily
$$
\bE \left| a_n ^{-2}  \sum_{2 \leq k \ne \ell \leq n } X_{1k} X_{1\ell}  G_{k\ell}\right| = \int_0 ^ \infty \bP \left( \left| a_n ^{-2}  \sum_{2 \leq k \ne \ell \leq n } X_{1k} X_{1\ell}  G_{k\ell}\right|   \geq t  \right) dt  \leq c \sqrt{ \frac {\tr  ( GG^* ) }{n^2 }} \log n. 
$$
(In fact, with slightly more care, we may replace $\log n$ by $(\log n)^{\frac 1 \alpha} \log (\log n)$).  It remains to check that we can remove the term $\log n$ for $\alpha > 4/3$. It will come easily from the bound 
\begin{equation}\label{boundTbis}
\bP \left( \left| a_n ^{-2}  \sum_{2 \leq k \ne \ell \leq n } X_{1k} X_{1\ell}  G_{k\ell}\right| \geq t  \right)  \leq c \left( \frac {\tr  ( GG^* ) }{n^2 t ^2 } \right)^{ \frac{\alpha }{ 4 - \alpha}}. 
\end{equation}
Let $s \geq 1$ to be chosen later. We now set $Y_k = X_{1k}  \ind ( |X_{1k}| \leq s a_n ) /  a_n $. We write 
$$
\left|  \sum_{2 \leq k \ne \ell \leq n } Y_k Y_\ell G_{k\ell} \right|^2 =   \sum_{ k_1 \ne \ell_1 , k_2 \ne \ell_2 } Y_{k_1} Y_{\ell_1}  Y_{k_2} Y_{\ell_2}G_{k_1\ell_1} G_{k_2 \ell_2 }^*.
$$
The variables $Y_k$ are iid and by symmetry $\bE Y_1 = \bE Y_1^3 = 0$. Hence, since $G_{k\ell} = G_{\ell k}$, taking expectation we obtain 
$$
\bE \left|  \sum_{2 \leq k \ne \ell \leq n } Y_k Y_\ell G_{k\ell} \right|^2 = 2     \sum_{ k \ne \ell } \bE [ Y^2_{1} ]^2   G_{k\ell} G_{\ell k}^*  \leq   2 \bE [ Y^2_{1} ]^2   \tr GG^*.
$$
It is routine to check that, for $p > \alpha$,  $ \bE [ |Y _{1} |^p ]  \leq c(p)  s^{p - \alpha} n^{-1}$.  Hence, arguing as above, we find
\begin{eqnarray*}
\bP  \left( \left| a_n ^{-2}  \sum_{2 \leq k \ne \ell \leq n } X_{1k} X_{1\ell}  G_{k\ell}\right| \geq t  \right) & \leq & c   \frac{ s^{2 ( 2 - \alpha ) }   \tr GG^* } {n^{2} t ^2 }  +  c s ^{-\alpha}.
\end{eqnarray*}
We conclude by choosing $s =  1 \vee ( n^2 t^2 / \tr GG^*  ) ^{1 / (4 - \alpha)}$. 
\end{proof}

We may now compare $I(z)$ and $J(z)$  to $Y(z)$ and $X(z)$.
\begin{lemma}[Diagonal approximation]\label{le:diagapprox}
Let $0 < \alpha < 2$ and $z = E + i \eta \in \bC_+$. There exists $c  >0$ such that if $\e$ is given by \eqref{eq:defeps}
then 
$$
| Y(z) - I (z) | \leq c \eta^{-\frac \alpha 2} \varepsilon \quad  \hbox{and } \quad 
| X(z) - J (z) | \leq c \eta^{-1} \varepsilon . 
$$
\end{lemma}
\begin{proof}
Define
\begin{equation}\label{defT}
T(z) = - a_n^{-1} X_{11} + a_n ^{-2}  \sum_{2 \leq k \ne \ell \leq n } X_{1k} X_{1\ell}  R^{(1)}_{k\ell}.
\end{equation}
We notice that for any, $z, z' \in \bC_+$ and $\alpha > 0$, 
 $$| (iz)^{-\alpha /2} - (iz')^{-\alpha /2} |\leq \frac{\alpha}{2} |z - z'| (  \Im(z) \wedge \Im (z'))  ^{-\alpha/2-1}.$$ 

With the notation of  \eqref{eq:resolventformula}, we also note that $\Im (\sum_{k=1}^n   X_{1k}^2  R^{(1)}_{kk} ) \geq 0$ and  $\Im ( - a_n^{-1} X_{11}+ \langle  X_1 , R^{(1)}X_1 \rangle )  \geq 0$. Hence, from  \eqref{eq:resolventformula}, for any event $\Omega$, 
\begin{eqnarray}
\left| Y(z) - I (z) \right|& \leq  &   \frac \alpha 2   \eta ^{-\frac{\alpha}{2} - 1}   \bE[ |T(z)| \ind_{\Omega}] + 
 \eta^{-\frac \alpha 2} \bP ( \Omega^c). \label{eq:diffYI} 
\end{eqnarray}
Applying the same argument with $X(z),J(z)$, we get 

$$\left| Y(z) - I (z) \right| \leq   \frac \alpha 2   \eta ^{-\frac{\alpha}{2} }D(z),\qquad
\left| X(z) - J (z) \right| \leq    \eta^{-1} D(z) $$
with
\begin{equation}\label{defD}
 D(z)= \eta ^{- 1}   \bE[ |T(z)| \ind_{\Omega}] +  \bP ( \Omega^c).\end{equation}
We may bound $D(z)$ by using that by Lemma \ref{le:offdiag}, since  $R^{(1)} $ is independent of $X^{(1)}$, which gives
\begin{equation}\label{boundT}
\bP \left( \left| a_n ^{-2}  \sum_{2 \leq k \ne \ell \leq n } X_{1k} X_{1\ell}  G_{k\ell}\right| \geq  \sqrt{ \frac {M_ n } { n } }  t \right)  \leq c t^{ - \alpha }   \log \left(n \left( 2 \vee t   \right)  \right)\log   \left( 2 \vee   t   \right) .
\end{equation}
and  for  $ 1 < \alpha < 2$,
\begin{equation}\label{boundT2}
 \bE \left| a_n ^{-2}  \sum_{2 \leq k \ne \ell \leq n } X_{1k} X_{1\ell}  G_{k\ell}\right| \leq c \sqrt{ \frac {M_ n } { n } }   \left(1 + \ind_{1<  \alpha \leq 4/3}  \log n  \right). 
\end{equation}
\begin{itemize}
\item {\em
Let us first assume that  $1 < \alpha < 2$}, then taking $\Omega^c = \emptyset$ in \eqref{eq:diffYI} and using \eqref{boundT2}, it shows that  for some constant $c>0$, 
\begin{eqnarray*}
D(z)
 \leq    c  \eta ^{ - 1}  \left(  n^{-\frac 1 \alpha} +  \sqrt{\frac{M_n}{n}}   \left(1 + \ind_{1<  \alpha \leq 4/3} \log n  \right) \right).
\end{eqnarray*}
\item 
{\em Assume that $0< \alpha \leq 1$}, we take in \eqref{defD}
 $$\Omega = \left\{ a_n^{-1} |X_{11}| \leq t  \, ; a_n ^{-2}  |Ê\sum_{1 \leq k \ne \ell \leq n } X_{1k} X_{1\ell}  R_{k\ell}^{(1)}|  \leq t \right\}.$$ 
 Then, we have
\begin{eqnarray*}
\bE[ |T(z)| \ind_{\Omega}]  & \leq &   \int_0^t \bP( a_n^{-1} |X_{11}| \ge y )dy  +  \int_0^t \bP(| a_n ^{-2}  \sum_{1 \leq k \ne \ell \leq n } X_{1k} X_{1\ell}  R_{k\ell}^{(1)}| \ge y )dy.
\end{eqnarray*}
Assume that $1 / n \leq t \leq 1$. Then, using \eqref{boundT}, we find that $ \bE[|T(z)|\ind_{\Omega}] $ is bounded up to multiplicative constant (depending on $\alpha$) by
\begin{eqnarray*}
 \frac{t^{1 - \alpha}}{n} +  \frac{\log (n)}{n} \ind_{\alpha =1}  +  \left(\frac{M_n}{n}\right)^{\frac \alpha 2 } \left( \log n \right)^2  t^{1 - \alpha}. 
\end{eqnarray*}
(using \eqref{eq:defeps} we may safely bound the terms $ \log (  t n / M_n  )$ by $\log n$). With $t = \eta$ we find
$$
D(z)  \le  c    \eta^{-\alpha} n ^{-1}  +c \frac{1_{\alpha=1} ( \log n) }{\eta n} 
+ c \left(\frac{M_n}{n\eta^2}\right)^{\frac{\alpha}{2}} \left( \log n \right)^2 .
$$
\end{itemize}
This yields the claimed bounds. \end{proof}

We next relate $I$ and $J$ with the functions $\varphi_{\alpha,z}$ and $\psi_{\alpha,z}$
by using the well known identities 
\begin{equation}\label{eq:gamma}
x^{-\delta} = \frac{ 1}{ \Gamma ( \delta)} \int_0^\infty t^{\delta - 1} e^{ -x t} dt \quad \hbox{ and } \quad x^{\delta} = \frac{ \delta}{ \Gamma (1 - \delta)} \int_0^\infty t^{-\delta - 1}  ( 1 - e^{ -x t} )  dt 
\end{equation}
valid for $x \in \cK_1$, $\delta > 0$ and $0 <\delta < 1$ respectively.  We get 
$$
I (z) =   \frac{ 1}{ \Gamma ( \frac \alpha 2 )} \int_0^\infty t^{\frac \alpha 2 - 1} \bE \exp \left\{  i  t  \left( z + a_n ^{-2} \sum_{k=2}^n   X_{1k}^2  R^{(1)}_{kk}   \right)  \right\} dt\,.$$
We may apply Corollary \ref{cor:formulealice} to take the expectation over $X_1$ and get 
\begin{eqnarray*}
I (z) &=&  \frac{ 1}{ \Gamma ( \frac \alpha 2 )} \int_0^\infty t^{\frac \alpha 2 - 1} e^{itz} \bE \exp \left\{ - w_\alpha ^{\alpha} (  2 t ) ^{\frac \alpha 2 } \frac{1}{n}  \sum_{k=2}^n   \left(  - i R^{(1)}_{kk}\right)^{\frac \alpha 2} |g_k|^\alpha    \right\} dt\,,\\
&=&\bE\left[\varphi_{\alpha,z}\left( 
\frac{1}{n}  \sum_{k=2}^n   \left(  -i R^{(1)}_{kk}\right)^{\frac \alpha 2}\frac{ |g_k|^\alpha   }{\bE[|g_k|^\alpha]}
\right)\right]
\end{eqnarray*}
where $w_\alpha > 0$ was defined in the introduction and $(g_i)_{i \geq 1}$ are iid standard gaussian variables. Similarly,
we find that 
\begin{eqnarray}
J (z) &=&  \int_0^\infty e^{itz} \bE \exp \left\{ - w_\alpha ^{\alpha}
 (  2 t ) ^{\frac \alpha 2 } \frac{1}{n}  \sum_{k=2}^n   \left(  - i R^{(1)}_{kk}\right)^{\frac \alpha 2} |g_k|^\alpha    \right\} dt \label{eq:expJ}\\
&=& \bE
\left[ \psi_{\alpha,z} \left( \frac{1}{n}  \sum_{k=2}^n   \left(
 -i R^{(1)}_{kk}\right)^{\frac \alpha 2}\frac{ |g_k|^\alpha   }{\bE[|g_k|^\alpha]}
\right)\right]\,.
\end{eqnarray}

The next lemma due to Belinschi, Dembo and Guionnet \cite{BDG09} will be crucial in the sequel. 
\begin{lemma}\cite[Lemma 3.6]{BDG09}\label{le:36BDG}
For any $z \in \bC_+$,  the functions $\varphi_{\alpha,z}$ and $\psi_{\alpha,z}$ are Lipschitz with constant $c = c(\alpha) |z|^{ - \alpha} $ and $c = c(\alpha) |z|^{ - \alpha/2} $  on $\cK_{\alpha/2}$. Moreover $\varphi_{\alpha,z}$ maps $\cK_{\alpha/2}$ into $\cK_{\alpha/2}$ and $\psi_{\alpha,z}$ maps $\cK_{\alpha/2}$ into $\cK_{1}$. \end{lemma}

\begin{proof}
The first statement follows from \cite[Lemma 3.6]{BDG09} by a change of variable. For the second, we note that if $x \in \cK_{\alpha/2}$ then $x = ( - i w ) ^{\alpha/2}$ with $w \in \bC_+$ and from \eqref{eq:gamma} $\varphi_{\alpha,z} ( x) = \bE ( i z + i w S )^{- \alpha /2}$ where $S$ is non-negative $\alpha/2$ stable law with Laplace transform, for $x > 0$, $\bE \exp ( - x S) = \exp ( - \Gamma ( 1- \alpha /2) x^{\alpha/2} )$. Similarly, $\psi_{\alpha,z} (x) = \bE ( i z + i w S )^{- 1}$. In particular, $\varphi_{\alpha,z} (x)  \in \cK_{\alpha /2}$ and  $\psi_{\alpha,z} (x)  \in \cK_{\alpha /2}$. 
\end{proof}

We are now able to prove Proposition \ref{prop:fixpoint}. 
\begin{proof}[Proof of Proposition \ref{prop:fixpoint}]
The point is that the Lipschitz constant in Lemma \ref{le:36BDG} depends on $|z|$ and not only $\Im (z)$.  Hence since $ \rho_k := \left(  - i R^{(1)}_{kk}\right)^{\frac \alpha 2} \in \cK_{\alpha/2}$, using exchangeability, we deduce that 
\begin{align*}
& \left| I(z) -\varphi_{\alpha,z}  (\bE  \rho_2)  \right| \\
& = \left|\bE\left[\varphi_{\alpha,z}\left( 
\frac{1}{n}  \sum_{k=2}^n   \left(  -i R^{(1)}_{kk}\right)^{\frac \alpha 2}\frac{ |g_k|^\alpha   }{\bE[|g_k|^\alpha]}
\right)\right] -\varphi_{\alpha,z}\left( 
\frac{1}{n-1}  \sum_{k=2}^n   \bE[\left(  -i R^{(1)}_{kk}\right)^{\frac \alpha 2}]
\right)\right]\\
&\leq c   \bE  \left|\frac{1}{n-1}  \sum_{k=2}^n    \rho_k  |g_k|^\alpha    -  \bE \frac{1}{n-1}  \sum_{k=2}^n   \rho_k |g_k|^\alpha \right| +\frac{c\bE[|\rho_2|]}{n}\\
& \leq c \left(  \bE  \left|\frac{1}{n-1}  \sum_{k=2}^n    \rho_k  |g_k|^\alpha    -  \frac{1}{n-1}  \sum_{k=2}^n   \rho_k  \bE  |g_k|^\alpha \right|  +       \bE  \left|\frac{1}{n-1}  \sum_{k=2}^n    \rho_k    -  \bE \frac{1}{n-1}  \sum_{k=2}^n   \rho_k  \right|+\frac{\bE[|\rho_2|]}{n}\right) .
\end{align*}
By the Cauchy-Schwarz inequality and Lemma \ref{le:concres2}, we obtain
\begin{align*}
& \left| I(z) -\varphi_{\alpha,z}  (\bE  \rho_2)  \right| \leq
 c  n^{-1} \sqrt{\bE[\sum_{k=2}^n  |\rho_k |  ^ 2]   } 
+c \left(  n^{-1} \sqrt{\bE[\sum_{k=2}^n  |\rho_k |  ^ 2]   }\right)^2+
 c \eta^{-  \frac{\alpha}{2} } n^{ -  \frac{\alpha}{4} }.
\end{align*}
By applying the Jensen inequality, we also notice that since $0 < \alpha\le 2$, 
\begin{align*}
&\bE[ \frac 1 {n-1} \sum_{k=2}^n  |\rho_k |  ^ 2]    = \bE[ \frac 1 {n-1} \sum_{k=2}^n   |R^{(1)}_{kk}  | ^ \alpha ] \\
&\quad\quad  \leq   \left ( \bE[\frac 1 {n-1} \sum_{k=2}^n   |R^{(1)}_{kk}  | ^ 2 ] \right)^{ \frac  \alpha 2} \leq\left( 
\bE[ \frac 1 {n-1}  \tr \left\{R ^{(1)} {R^{(1)}}^* \right\}] \right)^{ \frac \alpha 2}=M_{n}^{\frac \alpha 2}.
\end{align*}
Hence we obtain an error of
\begin{equation}\label{eq:Irho2}
 \left| I(z) -\varphi_{\alpha,z}  (\bE  \rho_2)  \right|   \leq  c n^{-\frac 1 2} M_n^{\frac{\alpha}{4}} + c n^{-1}
M_n^{\frac{\alpha}{2}}
+c\eta^{-\frac{\alpha}{2}} n^{-\frac{\alpha}{4}}.
\end{equation}
In the forthcoming computations, we shall always consider  $\eta$
so that  $\e$ of \eqref{eq:defeps} is
smaller than one so that $ n^{-1} M_n^{\frac{\alpha}{2}}$ vanishes 
and is neglectable compared to $ n^{-\frac 1 2} M_n^{\frac{\alpha}{4}}$.
However $\bE \rho_2$ and $Y(z)$ are close. More precisely, by equation \eqref{eq:rankineq} (in appendix) applied with $f(x)=(-ix)^{\frac \alpha 2}$
we find that
$$
\left|\sum_{i= 1}^n  \left(  - i R_{kk}\right)^{\frac \alpha 2} -  \sum_{i=2} ^{n}  \left(  - i R^{(1)}_{kk}\right)^{\frac \alpha 2}  + ( - iz ) ^{\frac \alpha 2} \right| \leq 2
n (n\eta)^{-\frac \alpha 2}. 
$$
Taking the expectation, we get 
$$
\left|\bE \rho_2 - Y(z) \right| \leq c (n\eta)^{-\frac \alpha 2}. 
$$
By Lemma \ref{le:36BDG}, the function $\varphi_{\alpha,z}$ is Lipschitz for some constant $c = c(\alpha,|z|)$ on $\cK_{\alpha/2}$. We deduce from \eqref{eq:Irho2} that
\begin{eqnarray*}
  \left| I(z) -\varphi_{\alpha,z}  (Y(z))  \right| 
& \leq  &  c n^{-\frac 1 2} M_n^{\frac{\alpha}{4}}
+c\eta^{-\frac{\alpha}{2}} n^{-\frac{\alpha}{4}} +  c (n\eta)^{-\frac \alpha 2}  \\
& \leq & c'   \eta^{-\frac{\alpha}{2+\alpha}}  n^{-\frac 1 2}   (\log n)^{\frac{\alpha(2+\alpha)}{16}} 
+2 c\eta^{-\frac{\alpha}{2}} n^{-\frac{\alpha}{4}}, \end{eqnarray*}
where we used the upper bound on $M_n$ given by \eqref{eq:boundMn}. We note finally that the first term is always smaller for $n$ large enough and $\eta \leq 1$ than the second term. We have thus proved that there exists $c >0$, such that for all $n^{- \frac{\alpha + 2}{ 4}} \leq \eta \leq 1$ and all integers, 
\begin{equation*}\label{eq:finalIphi}
  \left| I(z) -\varphi_{\alpha,z}  (Y(z))  \right| \leq c \eta^{-\frac{\alpha}{2}} n^{-\frac{\alpha}{4}}. 
\end{equation*}
The statement of Proposition \ref{prop:fixpoint} on $Y(z)$ follows by applying Lemma \ref{le:diagapprox}: we find 
\begin{eqnarray*}
  \left| Y(z) -\varphi_{\alpha,z}  (Y(z))  \right| & \leq &  c \eta^{ -  \frac{ \alpha}{2}} n^{ -  \frac{\alpha}{4} }  +  |ÊY (z) - I(z) |Ê \label{eq:finalYphi}
\end{eqnarray*}
Finally, we observe that the bound on $X(z)$ follows similarly from \eqref{eq:expJ}: 
\begin{eqnarray*}
  \left| X(z) -\psi_{\alpha,z}  (Y(z))  \right| & \leq &  c \eta^{ -  \frac{ \alpha}{2}} n^{ -  \frac{\alpha}{4} }  +  |ÊX (z) - J(z) |Ê \label{eq:finalXphi}.
\end{eqnarray*}
We now use Lemma \ref{le:diagapprox} and Proposition \ref{prop:fixpoint} is proved.  \end{proof}

\subsection{Rate of convergence of the resolvent}

We will now use Proposition  \ref{prop:fixpoint} as the stepping stone to obtain a quantitative rate of convergence of the spectral measure $\mu_A$ toward its limit. 

By Lemma \ref{le:36BDG}, if $z = E + i \eta$ and $| z|$ is large enough, say $E_0$, then $\varphi_{\alpha,z}$ and $\psi_{\alpha,z}$ are Lipschitz with constant $L < 1$ and in particular, it has a unique fixed point 
$$
y(z) = \varphi_{\alpha,z} ( y(z) ), \quad y (z) \in \cK_{\frac \alpha 2}.  
$$
From \cite[Theorem 1.4]{BG08}, the empirical measure $\mu_A$ converges a.s. to a probability measure $\mu_\alpha$ (for the topology of weak convergence). The Cauchy-Stieltjes transform of the limit measure $\mu_\alpha$ is equal to 
$$
g_{\mu_\alpha}  (z) =  \int \frac {\mu_\alpha (dx) }{x - z} =  i \psi_{\alpha, z} ( y(z)). 
$$
The above identity characterizes  the probability measure $\mu_\alpha$. 
 
\begin{theorem}[Convergence of Stieltjes transform] \label{th:boundStieltjes}
For all $0 < \alpha < 2$, there exists a finite set $\cE_\alpha \subset \bR$ such that if $K$ is a compact set with $K \cap \cE_\alpha = \emptyset$ the following holds for some constant $c = c ( \alpha, K)$.

\begin{enumerate}
\item[(i)] If $1 < \alpha < 2$ : for any integer $n \geq 1$, $z = E + i \eta$ with $E \in I$, $c    \sqrt{ \frac{Ê\log n}{  n }} \vee \left( n^{ - \frac \alpha {Ê8 - 3 \alpha} }  ( 1 + \ind_{Ê1 < \alpha < 4/3} ( \log n ) ^{ \frac {2 \alpha}{ 8 - 3 \alpha} })\right)  \leq \eta \leq 1$, 
\begin{equation}\label{eq:gmuAdelta}
 \left|\bE g_{\mu_A} (z)  - g_{\mu_\alpha} ( z)  \right| \leq c \delta,
\end{equation}
where $ \delta = \eta^{-\frac{\alpha}{2}} n^{-\frac{\alpha}{4}} + \eta ^{- \frac{8 - 3 \alpha}{2\alpha} } n^{-\frac 1 2}   ( 1 + \ind_{Ê1 < \alpha < 4/3}  \log n  ) +  \eta^{-1} \exp ( - \delta n \eta^2 )$.

\item[(ii)] If $0 < \alpha \leq 1$, the same statement holds with $c n^{-\frac{\alpha}{2 +3 \alpha}} ( \log n )^{\frac{4}{2 + 3 \alpha} }  \leq \eta\leq 1$  and $\delta =   \eta^{-\frac{\alpha}{2}} n^{-\frac{\alpha}{4}} +    \eta ^{ - \frac{2 + 3 \alpha}{2} }   n^{-\frac \alpha 2}( \log n ) ^2$. 

\end{enumerate}

Moreover for any interval $I \subset K$ of length $|I |\geq \eta \,( 1 \vee  \delta    |Ê\log  ( \delta )  |^{-1}) $,
$$
 \left|\bE \mu_{A} (I)   - \mu_\alpha ( I)  \right| \leq  c  |I |.
$$
\end{theorem}

This result implies Theorem \ref{th:mainconvloc}. Indeed the presence of the expectation of $\mu_A(I)$ instead of $\mu_A(I)$ does not pose a problem due to Lemma \ref{le:concspec} in Appendix. We start the proof of Theorem \ref{th:boundStieltjes} with a weaker statement. 

\begin{proposition}[Convergence of Stieltjes transform : weak form] \label{prop:boundStieltjes}
Statement (ii) of Theorem \ref{th:boundStieltjes} holds  and 
\begin{enumerate}
\item[(i')] If $1 < \alpha < 2$ and  $ c n^{-1/5}  \left( 1 + \ind_{ 1 < \alpha \leq \frac 4 3} ( \log n )^{\frac {2 }{5}} \right) \leq \eta\leq 1 $ then \eqref{eq:gmuAdelta} holds with 
$
 \delta = \eta^{-\frac{\alpha}{2}} n^{-\frac{\alpha}{4}}  +  \eta^{-  5/2 Ê} n^{-1/2}  \left( 1 + \ind_{ 1 < \alpha \leq \frac 4 3}  \log n  \right).  
$
\end{enumerate}
\end{proposition}

\begin{proof}
Assume first that $z = E + i \eta$ with $| E | \geq E_0$ and $\eta \leq 1$. If we apply Lemma \ref{le:36BDG} to $\varphi_{\alpha,z}$ we find 
\begin{eqnarray*}
|Y(z) - y (z) |  \leq  \frac {1} { 1 - L}  |Y(z) - \varphi_{\alpha,z} (Y(z))| .
\end{eqnarray*}
Also, by exchangeability $\bE g_{\mu_A} (z) = \bE G_{11} (z) = i X(z)$. Hence applying Lemma \ref{le:36BDG}  to $\psi_{\alpha,z}$, we deduce 
\begin{equation}\label{eq:boundgXY}
|\bE g_{\mu_A} (z) - g_{\mu_\alpha} ( z)  |  \leq  |X(z) - \psi_{\alpha,z} (Y(z)) | +  \frac {L} { 1 - L} |Y(z) - \varphi_{\alpha,z} (Y(z))|.
\end{equation}
Also, the Cauchy-Weyl interlacing theorem implies that the same type of bounds holds for the minor $A^{(1)}$ instead of $A$. Indeed applying Lemma \ref{le:Cauchy} to $f (x) = ( x - z)^{-1}$, we have 
$$
| g_{\mu_A} (z)  - g_{\mu_{A^{(1)}} } (z) |Ê \leq 2 (n \eta )^{-1}.  
$$
We recall that $\mu_{\alpha}$ has a bounded density (see \cite{BG08,BDG09,BCC}). Hence $\Im ( g_{\mu_\alpha} (z))$ is uniformly bounded. We get for any $z = E + i \eta$, $| E | \geq E_0$, $\eta \geq  n ^{-\frac{\alpha+2}{4}}$, 
\begin{equation}\label{bn}
\bE \Im ( g_{\mu_{A^{(1)} }} (z) ) \leq  \Im ( g_{\mu_\alpha} (z))  + \left|\bE g_{\mu_{A^{(1)}}} (z)  - g_{\mu_\alpha} ( z)  \right| \leq c +  \left|\bE g_{\mu_{A^{(1)}}} (z)  - g_{\mu_\alpha} ( z)  \right|. 
\end{equation}
On the other hand, the spectral theorem implies the important identity
$$
M_n = \frac 1 {n-1} \bE \tr  \left\{  R^{(1)} (R^{(1)})^*\right\} = \eta^{-1} \bE \Im ( g_{\mu_{A^{(1)} }} (z) ). 
$$
Then, by \eqref{eq:boundgXY},
\begin{equation*}
\eta M_n \leq 
2 (n \eta )^{-1} + c +  |X(z) - \psi_{\alpha,z} (Y(z)) | +  \frac {L} { 1 - L} |Y(z) - \varphi_{\alpha,z} (Y(z))|.
\end{equation*}
We first consider  the case $1 < \alpha < 2$.
Then, by Proposition \ref{prop:fixpoint}, we obtain  for $\eta \geq n^{-1/{2\alpha}} \ge n^{-1/2}$, 
$$
\eta M_n \le  c + c\eta^{-2} \sqrt{\frac{M_n}{n}}  \left( 1 + \ind_{ 1 < \alpha \leq \frac 4 3}  \log n \right) .
$$
By monotonicity, we find that $\eta M_n $ is upper bounded by $x^*$ where $x^*$ is the unique fixed point of 
$$
x  =  c + c  \eta^{-\frac 5 2 } n^{-\frac 1 2}  \left( 1 + \ind_{ 1 < \alpha \leq \frac 4 3} \log n \right)  \sqrt x. 
$$
It is easy to check that the unique fixed point of $x = a + b x^\beta$, with $a,b >0$ and $0 < \beta < 1$  is upper bounded by $\kappa(\beta) a$ if $a \geq b^{\frac 1 {1 -  \beta}}$. We deduce  that, for some constant $c_1 > 0$ and all $n^{-1/5} \left( 1 + \ind_{ 1 < \alpha \leq \frac 4 3} ( \log n )^{\frac {2 }{5}} \right)  \leq \eta\leq 1  $,  
$$
\eta M_n \le  c_1.
$$
So finally, from Proposition \ref{prop:fixpoint}, we find that for all $n^{-1/5} \left( 1 + \ind_{ 1 < \alpha \leq \frac 4 3} ( \log n )^{\frac {2 }{5}} \right)  \leq \eta\leq 1  $,
$$
|\bE g_{\mu_A} (z) - g_{\mu_\alpha} ( z)  | \leq  c_2 \eta^{-\frac{\alpha}{2}} n^{-\frac{\alpha}{4}} + c_2 \eta^{-2}n^{-\frac 1 \alpha} +  c_2 \eta^{-\frac 5 2 } n^{-\frac 1 2} \left( 1 + \ind_{ 1 < \alpha \leq \frac 4 3}  \log n \right)   . 
$$
We notice that the middle term is negligible compared to the last for our range of $\eta$.

Assume finally that $0 < \alpha \leq 1$, then arguing as above for $\eta \geq n^{-1/2} \geq n^{-1/{2\alpha}} $, 
\begin{equation}\label{po}
\eta M_n \leq c  + c  \eta^{ - 1}  \left(\frac{M_n}{n\eta^2}\right)^{\frac{\alpha}{2}} ( \log n ) ^2.
\end{equation}
We deduce  that, for some $c_1 >0$ and all $n^{-\frac{\alpha}{2 +3 \alpha}} ( \log n )^{\frac{4}{2 + 3 \alpha} }  \leq \eta\leq 1$,  
$ \eta M_n \leq  c_1$.
 We find that for all $n^{-\frac{\alpha}{2 +3 \alpha}} ( \log n )^{\frac{4}{2 + 3 \alpha} }  \leq \eta\leq 1$,
$$
|\bE g_{\mu_A} (z) - g_{\mu_\alpha} ( z)  | \leq  c_2 \eta^{-\frac{\alpha}{2}} n^{-\frac{\alpha}{4}} +  c_2 \eta^{-1 - \alpha}n^{-1} ( 1 + (\log n) \ind_{\alpha = 1}  ) + c_2  n^{-\frac \alpha 2} \eta ^{ - \frac{2 + 3 \alpha}{2} }  ( \log n ) ^2. 
$$
Again the middle term is negligible compared to the first  for our range of $\eta$.

We have proven the proposition if $z = E + i \eta$ and $|E| \geq E_0$ is large enough. It remains to prove the statement for all $E$ outside of a finite set. It is proven in \cite{BDG09} that $\varphi_{\alpha,z} ( y ) = c z^{-\alpha} g (y)$ where $g$ is an entire function and $c$ is a constant (both depends on $\alpha$). It follows that the set of $z \in \bC$ such that $\varphi'_{\alpha, z} (y (z) ) = 1$ is finite (for details see  \cite[\S 5.3]{BDG09}). We define $\cE'_\alpha$ as the set of real $E$ such that there exists $0 \leq \eta \leq   E_0 + 1$ with $\varphi_{\alpha, E + i \eta} ' ( y (  E + i \eta  ))  = 1$. We set finally $\cE_\alpha = \{Ê0 \} \cup \cE'_\alpha$. This set is finite. Let $K$ be a compact interval which does not intersect $\cE_\alpha$. From the implicit function theorem, there exist $\tau , c_0 > 0$ such that for $t \geq 0 $, $0 \leq \eta \leq   E_0 + 1$, $E \in K$, 
\begin{equation}\label{eq:loccontr}
\hbox{if $| y - y (E + i \eta) |Ê\leq \tau$ and $|Êy - \varphi_{\alpha, E + i \eta} ( y ) | \leq t$ then $| y - y (E + i \eta) |Ê\leq c_0 t $. }
\end{equation}
Therefore, we may use Lemma \ref{le:36BDG} and an alternative version of \eqref{eq:boundgXY} : for any $z = E + i \eta $ with $E \in K $ and $0 \leq \eta \leq   E_0 + 1$, if $| Y(z) - y (z) |Ê\leq \tau$  then
\begin{equation}\label{eq:boundgXY2}
|\bE g_{\mu_A} (z) - g_{\mu_\alpha} ( z)  |  \leq  |X(z) - \psi_{\alpha,z} (Y(z)) | +  \frac { c_0 c(\alpha) } { | z| ^{\frac \alpha 2 } } |Y(z) - \varphi_{\alpha,z} (Y(z))|.
\end{equation}

To apply Proposition \ref{prop:fixpoint}, we shall
 use  an inductive argument to insure that the hypothesis $| Y(z) - y (z) |Ê\leq \tau$ is satisfied. 
 We set for integer $\ell$, $\eta_0 = \eta$, $\eta_{\ell+1} = \eta_\ell + \frac \tau 3  ( 1 \wedge \eta_\ell ) ^2$ and $z_\ell= E + i  \eta_\ell$. There exists $k$ such that $E_0 \leq \eta_k \leq E_0 + \tau $.  Then $ \varphi_{\alpha, z_k}$ is a contraction and the above argument proves that 
$$
|\bE g_{\mu_A} (z_k) - g_{\mu_\alpha} ( z_k)  | \leq  c \delta \quad \hbox{ and } \quad | Y(z_k) - y (z_k) |Ê\leq c \delta,
$$
(note that $\delta$ is a pessimistic bound since $\Im ( z_k ) $ is bounded away from $0$). We notice that it is sufficient to prove the statement of the proposition in the range, for $1 < \alpha < 2$, $ \kappa n^{-1/5}  \left( 1 + \ind_{ 1 < \alpha \leq \frac 4 3} ( \log n )^{\frac {2 }{5\alpha}} \right) \leq \eta\leq 1 $ and  for $0 < \alpha \leq 1$, $ \kappa n^{-\frac{\alpha}{2 +3 \alpha}} ( \log n )^{\frac{4}{2 + 3 \alpha} }$, where $ \kappa > 0$ is any fixed constant. Hence, up to increasing $ \kappa$, we may assume that $s c \delta \leq \tau/3$, where $s \geq 1$ is large number that will be chosen later on. 

To obtain a priori bounds for $Y(z)-y(z)$ and $|\bE g_{\mu_A} (z) - g_{\mu_\alpha} ( z)  |$ at $z=z_{k-1}$
from those at $z=z_{k}$ observe that for any probability measure $\mu$ on $\bR$ and $0 \leq \beta \leq 1$,  
\begin{equation}\label{gloupy}
|Êg_{\mu} (E + i \eta_\ell)  ^\beta  - g_{\mu } ( E + i \eta_{\ell +1} ) ^\beta   | \leq \frac{ |Ê\eta_ \ell - \eta_{\ell +1} |Ê } { \eta_\ell ^{1 + \beta }} \leq   \frac { \tau}{3}.
\end{equation}
Using the above control with $\mu=\sum_{k=1}^n \langle v_k, e_1\rangle ^2 \delta_{\lambda_k}$
so that $g_\mu(z)=R_{11}(z)$
we deduce by applying \eqref{gloupy} with $\beta=\alpha/2$ that
$$|(-iR_{11}(E + i \eta_\ell) )^{\frac{\alpha}{2}}- (-iR_{11}(E + i \eta_{\ell+1}) )^{\frac{\alpha}{2}}|\le \frac{\tau}{3}$$
and thus  $| Y (z_{k-1}) -  y ( z_{k-1})  | \leq  \tau$. We get a similar control for $\bE g_{\mu_A} (z_{k-1})$
by applying \eqref{gloupy} with $\beta=1$ so that
$$M_n (z_{k-1} )  \leq \eta^{-1} \left(  \frac {\tau} 3 + c \delta + \sup_\ell \Im g_{\mu_\alpha } (z_\ell) \right) \leq  \eta^{-1} ( \tau    +  \sup_\ell \Im g_{\mu_\alpha } (z_\ell) ) =  c_1 \eta^{-1}. $$
Therefore, using  Proposition \ref{prop:fixpoint}, we find for some constant $c' >0$. 
$$
| Y (z_{k-1}) -  \varphi_{\alpha, z_{k-1}} ( Y ( z_{k-1}) )   |  \vee | X (z_{k-1}) -  \psi_{\alpha, z_{k-1}} ( Y ( z_{k-1}) )   |  \leq  c' \delta.
$$
From what precedes, it implies that $| Y (z_{k-1}) -    y ( z_{k-1})    | \leq   c_0 c' \delta$. We choose $s$  large enough so that $ c'c_0 \leq s c$, so that we have $ c' c_0 \delta \leq \tau /3$.
Also, we may use \eqref{eq:boundgXY2}. We find for some new constant $c''$,
$$
|\bE g_{\mu_A} (z_k) - g_{\mu_\alpha} ( z_k)  | \leq  c'' \delta.
$$
Finally, if $s$ was also chosen large enough so that $ c'' \leq s c$, then $c'' \delta \leq \tau /3$, and we may repeat the above argument down to $\ell = 0$. 
\end{proof}

When $\alpha \in (1,2)$, a bootstrap argument allows to improve significantly Proposition \ref{prop:boundStieltjes}. 
The idea
is, in the spirit of \cite{ESY10} that if the imaginary part of $\left\langle X, R^{(1)} X\right\rangle$ is a priori bounded
below by something going to zero more slowly than $\eta$, then we can improve the result of the key Lemma
\ref{le:diagapprox}. Before moving to the proof, we state a classical deconvolution lemma. 
\begin{lemma}[From Stieltjes transform to counting measure]\label{le:deconvolution}
Let $L > 0$, $0 < \e < 1$, $K$ be an interval of $\bR$ and $\mu$ be a probability measure on $\bR$. We assume that for some $\eta  > 0 $ and all $E \in K$, either
$$ \Im g_\mu ( E + i \eta ) \leq  L \quad \hbox { or }  \quad \mu \left( Ê[ÊE - \frac \eta 2 , E + \frac \eta 2 ] \right) \leq L \eta.$$
 Then, there exists a universal constant $c$ such that for any interval $ I \subset K$ of size at least $\eta$ and such that $\dist ( I , K^c ) \geq \e$, we have
$$
\left|Ê\mu (  I  ) - \frac 1 \pi \int_{I}  \Im g_{\mu} (E + i \eta)  dE \right|  \leq c ( L \vee \e^{-1})   \eta  \log \left( 1 + \frac{|ÊI |}{\eta} \right) . 
$$
\end{lemma}

\begin{proof}
Let us prove the first statement. We observe that
$$
\frac 1 \pi \Im ( g_{\mu}  (y + i \eta) ) =  \frac 1 \pi  \int_{\bR} \frac{\eta}{ (y - x )^2 + \eta^2} \mu(dx) = P_\eta * \mu (y),   
$$
where $P_\eta$ is the Cauchy law with parameter $\eta$. We thus need to perform a classical deconvolution. We may for example adapt Tao and Vu \cite[Lemma 64]{TV11} (see also e.g. \cite[p.15]{GuionnetBourbaki}). Define 
$$
F (y) = \frac 1 \pi  \int_{I} \frac{\eta}{ (y - x )^2 + \eta^2} dx = P_\eta (  I - y ). 
$$
In particular
$$
\left|Ê\mu (  I  ) - \frac 1 \pi \int_{I}  \Im g_{\mu} (E + i \eta)  dE \right|  = \left| \mu ( I ) -  \intÊF(y)  \mu (dy)  \right|.
$$
Now, the Cauchy law has density, $ P_{\eta} ( t  ) =   \frac 1 \pi \frac{\eta }{Ê \eta^2+  t^2 }.$
It follows that for $\{  y \in I\}$, $\{Êy \in I^c, \dist ( y , I) \leq |I |Ê\}$ and  $\{Êy \in I^c, \dist ( y , I) \geq |I |Ê\}$ we may use respectively the bounds
$$
\left|ÊF(y)  - 1 \right| \leq  \frac{ c }{Ê 1+  \dist( y , I^c ) \eta^{-1} } \; \hbox{, } \quad  | F(y) | \leq  \frac{ c }{Ê 1+  \dist( y , I) \eta^{-1} } \quad \hbox{Êand } \quad  | F(y) | \leq  \frac{ c Ê|ÊI | \etaÊ}{Ê \dist( y , I)^2  }. 
$$
We write if $I = [a , b]$, $I_1 = I^c \cap [Êa - |ÊI |Ê, b  + |ÊI |Ê] \cap K$ and $I_2 = I^c \cap [Êa - |ÊI |Ê, b + |ÊI |Ê]^c  \cap K $, 
\begin{align*}
\left|Ê\mu (  I  ) - \frac 1 \pi \int_{I}  \Im g_{\mu} (E + i \eta)  dE \right| & \leq  \int_I \frac{ c }{Ê 1+  \dist( y , I^c ) \eta^{-1} } \mu (dy)\\
& \quad \;  + \int_{I_1Ê}\frac{ c }{Ê 1+  \dist( y , I ) \eta^{-1} } \mu (dy) \\
&  \quad \quad \; +  \int_{I_2Ê} \frac{ c Ê|ÊI | \etaÊ}{Ê \dist( y , I)^2  }  \mu (dy) \\
&  \quad \quad\quad \; + \int_{K^cÊ}\frac{ c }{Ê 1+  \dist( y , I ) \eta^{-1} } \mu (dy).
\end{align*}
However,  by assumption if $J = [ÊE - \eta/2, E + \eta/2] $ is an interval of size $\eta$ with $E \in K$, 
$$
L \geq \Im g_\mu ( E + i \eta )  =  \int_{\bR} \frac{\eta}{ (y - x )^2 + \eta^2} \mu(dx) \geq \frac{3}{Ê4 \eta} \mu ( J). 
$$
We deduce that 
$
\mu ( J) \leq \frac{ 4L}{3} \eta.$
Now, consider a partition $\cP$ of $\bR$ into intervals of size $ \eta$. We get from this last upper bound
$$
 \int_I \frac{ c }{Ê 1+  \dist( y , I^c ) \eta^{-1} } \mu (dy) \leq \sum_{ J \in \cP \cap I} \frac{ c \mu ( J)  } {Ê1 + \dist( J , I^c )} \leq  \sum_{  k = 0 }^{Ê| I |Ê\eta ^{-1}}  \frac{ c'  L \eta } {Ê1 + k} \leq c'' L \eta \log ( 1 + |I |Ê\eta^{-1} ). 
$$
The other terms are bounded similarly.
\end{proof}

\begin{proof}[Proof of Theorem \ref{th:boundStieltjes}]

In view of Proposition \ref{prop:boundStieltjes} and Lemma \ref{le:deconvolution}  applied to $\bE \mu_A$ and $\mu_\alpha$, it remains to prove statement $(i)$ of Theorem \ref{th:boundStieltjes}. We thus assume in the sequel that $1 < \alpha < 2$. The proof is divided into five steps. Throughout the proof, we assume that $n \geq 3$ (without loss of generality) and we denote by 
$\bE_1 [Ê\cdot ]Ê$ and $\bP_1 ( \cdot) $ the conditional expectation and probability given $\cF_1$, the $\sigma$-algebra generated by the random variables $(X_{ij})_{ i \geq j \geq 2}$.

\subsubsection*{Step one : Lower bound on the Stieltjes transform}

Let $K = [ a , b]$ be an interval which does not intersect the finite set $\cE_\alpha$, defined in  Proposition \ref{prop:boundStieltjes}. The limit spectral measure $\mu_\alpha$ has a positive density on $\bR$. In particular, there exists a constant $ c_0 = c_0(K_0,\alpha) > 0$ such that for all $0 \leq \eta \leq 1$ and $x \in K$, 
$$
\Im g_{\mu_{\alpha}}( x + i \eta ) \geq c_0.
$$
Consequently, if there exists  $0 \leq \eta \leq 1$ such that for all $x \in K$,
\begin{equation}
\label{eq:BNI}
|Ê\bE g_{\mu_A} (x + i \eta ) - g_{\mu_\alpha} (x + i \eta) | \leq \frac {c_0} 2 
\end{equation}
then 
\begin{equation*}
\bE \Im g_{\mu_A}( x + i \eta ) \geq \frac {c_0} 2.
\end{equation*}

Note that Proposition \ref{prop:boundStieltjes} already proves that \eqref{eq:BNI} holds if $n \geq n_0$ is large enough  and 
$$\eta_0 = n^{-\e},$$  
for some $\e > 0$. By an inductive argument, we aim at proving that \eqref{eq:BNI}  holds for the same constants $n_0$  but for some $\eta \ll \eta_0$.

For some constant $\delta > 0 $ to be defined later on, we set for $1 < \alpha < 2$, 
$$
\eta_\infty = \sqrt{ \frac{Ê\log n}{ 2 \delta n }} \vee \left( n^{ - \frac \alpha {Ê8 - 3 \alpha} } ( 1 + \ind_{Ê1 < \alpha < 4/3} ( \log n ) ^{ \frac {2 \alpha}{ 8 - 3 \alpha} })\right). 
$$
Note that $\eta_\infty \geq n ^{- \frac { \alpha +2 } { 4}}$ for all $n$ large enough (say again $n_0$).
 
\subsubsection*{Step two : Start of the induction}

We assume that \eqref{eq:BNI} holds for some $\eta_1 \in [Ên ^{- \frac { \alpha +2 } { 4}}  , \eta_0]$ and that 
\begin{equation}
\label{eq:BNIy}
|ÊY (x + i \eta_1 ) - y (x + i \eta_1) | \leq \frac \tau  3, 
\end{equation} 
where $\tau$ was defined in \eqref{eq:loccontr}. Let $0 < \tau' < \tau$ to be chosen later on. We are going to prove that \eqref{eq:BNI}-\eqref{eq:BNIy}  hold also for 
$$
\eta  \in \left[Ê  \eta_1  - \tau' \eta_1 ^2  , \eta_1 \right]. 
$$
provided that $\eta_1 \geq  t \eta_\infty$, $n \geq n_0$ and $t$ large enough. As in the proof of Proposition \ref{prop:boundStieltjes}, cf. \eqref{gloupy},
 if $\tau' $ is small enough, we note that \eqref{eq:BNIy} implies that 
\begin{equation}\label{eq:BNIYtau}
|ÊY (x + i \eta ) - y (x + i \eta) | \leq 2\frac{ |Ê\eta - \eta_1 | ^2 } { \eta_1 ^ 2 }   + |ÊY (x + i \eta_1 ) - y (x + i \eta_1) |  \leq  \tau. 
\end{equation}
First, by Weyl's interlacing property \eqref{eq:BNI}  holds  for $A^{(1)}$ with $c_0/2$ replaced by $c_0 /4$ ($\eta_1\gg n^{-1}$). Also it follows by Jensen's inequality that for $ z = x + i \eta$, with $x \in K$,   
\begin{eqnarray*}
\left(\Im R^{(1)} _{kk} (z)   \right)^{\frac \alpha 2} &=&  \left(  \sum_{ i = 1} ^{n-1}  \frac{ \eta} {Ê(\lambda^{(1)}_i - x)^2 + \eta^2 } \langle v^{(1)}_ i , e_k \rangle^2      \right)^{\frac \alpha 2} \\
&\geq&  \sum_{ i = 1} ^{n-1}   \left(  \frac{ \eta} {Ê(\lambda^{(1)}_i - x)^2 + \eta^2 } \right)^{\frac \alpha 2} \langle v^{(1)}_ i , e_k \rangle^2  \\
&=&\left( \left(\Im R^{(1)} (z)  \right)^{\frac \alpha 2}\right)_{kk}\\
& \geq & \eta^{1-\frac \alpha 2}  \sum_{ i = 1} ^{n-1}    \frac{ \eta} {Ê(\lambda^{(1)}_i - x)^2 + \eta^2 }   \langle v^{(1)}_ i , e_k \rangle^2 \\
&=& \eta^{1-\frac \alpha 2} \Im R^{(1)}_{kk} (z),
\end{eqnarray*}
where we have used  the fact $ \eta / ( Ê(\lambda^{(1)}_i - x)^2 + \eta^2 )\leq \eta^{-1}$ which implies   $ ( \eta /  ( (\lambda^{(1)}_i - x)^2 + \eta^2 )  ) ^{1 - \alpha / 2} \leq \eta^{\alpha / 2 -1 }$.  Note also that 
$$
|Ê\Im R^{(1)} _{kk}  (z)- \Im R^{(1)}_{kk} (E  + i \eta_1) |Ê\leq \frac { |Ê\eta - \eta_1 |Ê} {\eta^2} \leq \frac { \tau'Ê} { 1 - \tau' }. 
$$
Hence, if $\tau'$ is chosen small enough so that the above is less that $c_0  / 8$, we find with $c_1 =c_0 / 16$,
\begin{eqnarray*}
\bE \frac {1}{n-1} \sum_{k=1}^{n-1} \left(\Im R^{(1)}_{kk}  (z)
  \right)^{\frac \alpha 2} & \geq  &\bE \frac {1}{n-1} \sum_{k=1}^{n-1}\left( \left(\Im R^{(1)} (z)  \right)^{\frac \alpha 2}\right)_{kk}  \\
& \geq &  2 c_1 \eta^{1-\frac \alpha 2}.
\end{eqnarray*}
Now, for bounding from below
\begin{equation}\label{eq:Rtra}
\frac {1}{n-1} \sum_{k=1}^{n-1} \left(\Im R^{(1)} _{kk} (z) \right)^\frac \alpha 2 \geq \frac{1}{n-1} \tr \left\{ (\Im R^{(1)}(z))^{\alpha/2} \right\},
\end{equation} we observe that
by Lemma \ref{le:concspec}, since $ x^{\alpha/2}$ has
total variation on $[0,\eta^{-1}]$ equal to $\eta^{-\alpha/2}$,  that
\begin{eqnarray*}
\bP\left( \frac{1}{n-1}\tr  \left\{  ( \Im R^{(1)}(z) )^{\alpha/2} \right\}-\bE \frac{1}{n-1}\tr\left\{ (\Im R^{(1)}(z))^{\alpha/2}\right\} \le  r \right)
&\le& \exp\left(-\frac {(n-1) r^2\eta^{\alpha}}{2} \right) \,.
\end{eqnarray*}
Applying the above with 
$r=c_1 \eta ^{Ê1 - \frac \alpha 2}$ shows that for some $\delta >0$, 
$$\bP(\Lambda(z)^c)\le e^{-\delta n \eta^{2}},$$
where 
$$\Lambda(z):=
\left\{ \frac {1}{n-1} \tr \left\{ \left(\Im R ^{(1)} (z)  \right)^{\frac \alpha 2}  \right\}\geq c_1 \eta ^{Ê1 - \frac \alpha 2} Ê\right\}.
$$
Note that this probabilistic bound is non trivial only if $\eta_1 \geq n^{-\frac 1 2} \geq n ^{- \frac { \alpha +2 } { 4}}$ (recall that $1 < \alpha < 2$).

\subsubsection*{Step three : Gaussian concentration for quadratic forms} For any $z  = x + i \eta \in \bC_+$, we may bound from below  the imaginary part of 
$$Q(z)=a_n^{-2}\sum_{k=1}^n R^{(1)}_{kk}(z)X_{1k}^2 $$
on the event in $\Lambda(z) \in \cF_1$. 
Indeed, as the $\Im R^{(1)}_{kk},1\le k\le n-1,$ are non negative, we can use Lemma \ref{le:formulealice} to see that conditionnaly on $\cF_1$,
\begin{equation}\label{eq:id}
\Im Q(z) \stackrel{d}{=}  \left(\frac {1}{n-1} \sum_{k=1}^{n-1} \left(\Im R_{kk} ^{(1)} (z) G_k^2 \right)^{\frac \alpha 2}\right)^{\frac 2 \alpha} S  =  L(z) S , \end{equation}
where the equality holds in law and $S$ is a positive $\alpha/2$-stable law whereas the $G_k$ are independent standard Gaussian variables, independent from $S$. 
Moreover, if  $\Lambda (z)$ holds then from \eqref{eq:Rtra}
 $$ \sum_{k=1}^{n-1} \left(\Im R_{kk} ^{(1)} (z)  \right)^{\frac \alpha 2}\ge 
c_1  (n-1) \eta^{1 - \frac \alpha 2}\ge c_1  (n-1)\eta 
 \max_{k} \left(\Im R_{kk} ^{(1)} (z)  \right)^{\frac \alpha 2}.$$
Hence by Corollary \ref{cor:concnorm}, for some universal constants $c,\delta >0$, if 
$\Lambda(z)\in \cF_1$ holds, then 
$$\bP_1 \left( (\sum_{k=2}^n |\sqrt{\Im R_{kk} ^{(1)} (z)} G_k|^\alpha)^{\frac{1}{\alpha}}
\le \delta ((n-1) \eta^{1 - \frac \alpha 2})^{1/\alpha}\right)
\le e^{-\delta n\eta^{\frac 2 \alpha}}$$
which yields
\begin{eqnarray}\label{cont2}
\bP_1 \left(
L (z) \le \delta   \eta^{\frac{2}{\alpha}-1} \right) &\le& 
 e^{-\delta n \eta^{\frac 2 \alpha}}.
 \end{eqnarray}
Finally, we observe similarly that by Lemma \ref{le:formulealice},
$$\Im(Q(z)+T(z))=\langle X_1, \Im R^{(1)} X_1\rangle \stackrel{d}{=} \tilde L (z) \tilde S,$$
where conditonnally on $\cF_1$, $\tilde S$ is a positive $\alpha/2$-stable law, independent of $\tilde L(z)$. Moreover, if $\eta\ge c  n^{-\frac \alpha 2}$, the random variable $\tilde L(z)$ satisfies if
$ \Lambda(z)$ holds,
the probabilistic bound \begin{eqnarray*}\label{cont2bis}
\bP_1 \left(  \tilde L (z) \le \delta      \eta^{\frac{2}{\alpha}-1} \right) &\le& 
 e^{-\delta n  \eta^{\frac 2 \alpha}}.
 \end{eqnarray*}
We may thus summarize the last two steps by stating that if $n ^{-1/2} \leq \eta_1 \leq 1$  holds then 
$$
\bP \left( \Pi (z) ^c  \right) \leq 3 \exp ( - \delta n \eta^2 ). 
$$
where $z = x + i \eta$, $x \in K$ and
$$
\Pi (z) =Ê\Lambda (z) \cap  \left\{ L (z) \ge \delta      \eta^{\frac{2}{\alpha}-1} \right\}  \cap \left\{ \tilde L (z) \ge \delta      \eta^{\frac{2}{\alpha}-1}\right\}. 
$$
(recall that $1 < \alpha < 2$). 

\subsubsection*{Step four : Improved convergence estimates}  We next improve the results of Proposition \ref{prop:fixpoint} for our choice of $z = x+ i \eta$, $x \in K$. We  
write instead of \eqref{eq:diffYI}
\begin{eqnarray}
\left| Y(z) - I (z) \right|& \leq  &  
\frac \alpha 2  ( \delta \eta ^ { \frac{2}{\alpha} -1 } )^{- \frac{\alpha}{2} - 1}  \bE \left[Ê( S \wedge \tilde S )^{Ê- \frac{\alpha}{2} - 1}  |T(z)|  \right]   + \eta^{-\frac \alpha 2} \bP ( \Pi(z)^c), \label{eq:diffYI2}  \\
\left| X(z) - J (z) \right|& \leq  &  
( \delta \eta ^ { \frac{2}{\alpha} -1 } )^{-2} \bE \left[Ê( S \wedge \tilde S )^{Ê- 2}  |T(z)| \right]  +  \eta^{-1} \bP ( \Pi(z)^c). \label{eq:diffXI2} 
\end{eqnarray}
Then from  Lemma \ref{le:offdiag} and \eqref{boundTbis}, there exists $p > 1$ (depending on $\alpha$) such that 
$$
( \bE | T (z) |^p  )^{\frac 1 p}  \leq c  \left( n^{-\frac 1 \alpha} +  \sqrt{\frac{M_n}{n}}( 1 + \ind_{1 < \alpha \leq 4/3} \log n )   \right) .
$$
From H\"older's inequality and Lemma \ref{le:tailS}, we deduce that for some new constant $c>0$, 
\begin{eqnarray*}
\bE \left[Ê( S \wedge \tilde S )^{Ê- 2}  |T(z)| \right] \leq c  \left( n^{-\frac 1 \alpha} +  \sqrt{\frac{M_n}{n}}   ( 1 + \ind_{1 < \alpha \leq  4/3} \log n ) \right). 
 \end{eqnarray*}
From \eqref{eq:diffYI2}-\eqref{eq:diffXI2}, it follows that for $n ^{-1/2} \leq \eta_1 \leq 1$ and some new constant $c>0$, 
\begin{align*}
& \left| Y(z) - I (z) \right| \vee \left| X(z) - J (z) \right| \\
&\quad \leq \;  c  \eta ^{- 2(\frac{2}{\alpha} -1) }  \left( n^{-\frac 1 \alpha} +  \sqrt{\frac{M_n}{n}}  ( 1 + \ind_{1 < \alpha < 4/3} \log n )  \right)  + 3 \eta^{-1} \exp ( - \delta n \eta^2 ).
 \end{align*}
Note that $ \eta ^{- 2(\frac{2}{\alpha} -1) }n^{-\frac 1 \alpha} \leq 1$ if $\eta \geq n^{- \frac 1 {2 ( 2 - \alpha )}}\geq   n ^{-1/2} $ while $\eta^{-1} \exp ( - \delta n \eta^2 ) \leq 1$ if $( 2 \delta n / \log n )^{-1/2} \leq \eta$ . Note also that this last expression improves upon Lemma \ref{le:diagapprox} and then Proposition \ref{prop:fixpoint} can be improved into
\begin{align}\label{eq:newfixpoint}
& \left| Y(z) - \varphi_{\alpha, z} ( Y (z) )  \right| \vee \left| X(z) - \psi_{\alpha, z} ( Y (z) )  \right| \\
 & \quad  \leq  \;  c  \eta ^{- 2(\frac{2}{\alpha} -1) }  \left( n^{-\frac 1 \alpha} +  \sqrt{\frac{M_n}{n}}( 1 + \ind_{1 < \alpha \leq 4/3} \log n )   \right)  + c \eta^{-\frac \alpha 2} n^{- \frac \alpha 4} + c \eta^{-1} \exp ( - \delta n \eta^2 ). \nonumber
 \end{align}
Then, by \eqref{eq:BNIYtau} we may use the bound \eqref{eq:boundgXY2}. From \eqref{bn}, we thus obtain, for $( 2 \delta n / \log n )^{-1/2} \leq \eta_1 \leq 1$,
$$
\eta M_n (z) \le  c + c\eta ^{- 2(\frac{2}{\alpha} -1) } \sqrt{\frac{M_n}{n}}  ( 1 + \ind_{1 < \alpha \leq 4/3} \log n ) =  c + c\eta ^{- \frac{8 - 3 \alpha}{2\alpha} } n^{-\frac 1 2} \sqrt{\eta M_n} ( 1 + \ind_{1 < \alpha < 4/3} \log n ).
$$
We deduce  that, for some constant $c > 0$,  if $ \eta_\infty \leq \eta_1 \leq 1$,  
$$
\eta M_n \le  c.
$$
So finally, from  \eqref{eq:boundgXY2}-\eqref{eq:newfixpoint}, we find that for  $\eta_\infty \leq \eta_1 \leq 1$,
\begin{align}\label{eq:boundSt2}
& |\bE g_{\mu_A} (z) - g_{\mu_\alpha} ( z)  | \\
& \quad \leq \,  c_3 \eta^{-\frac{\alpha}{2}} n^{-\frac{\alpha}{4}} + c_3 \eta ^{- \frac{8 - 3 \alpha}{2\alpha} } n^{-\frac 1 2}  ( 1 + \ind_{1 < \alpha \leq 4/3} \log n )   + c_3 \eta^{-1} \exp ( - \delta n \eta^2 ). \nonumber
\end{align}

\subsubsection*{Step five : End of the induction}  

From \eqref{eq:loccontr}-\eqref{eq:boundSt2}, we deduce that if $ t \eta_\infty \leq \eta_1 \leq 1$ and $t$ large enough, then
$$
|\bE g_{\mu_A} (z) - g_{\mu_\alpha} ( z)  | \leq \frac {c_0} 2 \quad \hbox{ and }  \quad \left| Y(z) - y(z) \right|  \leq \frac \tau 3 . $$
We have thus proved that \eqref{eq:BNI}-\eqref{eq:BNIy}  holds also for our choice of $\eta$.

The argument is completed as follow. Let $K = [a , b]$ be a compact interval that does not intersect $\cE_\alpha$. Starting for $\eta_0 = n^{-\e}$, by applying $m$ times the induction, we deduce that  \eqref{eq:BNI} holds for $\eta_m = \eta_{m-1} - \tau' \eta_{m-1}^2$. Since this sequence vanishes as $m$ goes
to infinity, and, for some $m$, we have $t \eta_\infty \leq \eta_m  < 2 t \eta_\infty$. We deduce that for all $n$ large enough (say $n_0$), \eqref{eq:boundSt2} holds for $z = x + i \eta$, with $t \eta_\infty \leq \eta \leq 1$ and $K = [a , b]$. The statement follows.
\end{proof}

\section{Weak delocalization of eigenvectors}
\label{sec:delocvect}

Following Erd\H{o}s-Schlein-Yau \cite{ESY10}, from local convergence of the empirical spectral distribution (Theorem \ref{th:boundStieltjes}), it is possible to deduce the delocalization of eigenvectors. Using the union bound, Theorem \ref{th:delocvect} follows from the next proposition.

\begin{proposition}[Delocalization of the eigenvectors] \label{prop:delocvect}
For any $1 < \alpha < 2$, there exist $\delta, c >0$ and a finite set $\cE_\alpha \subset \bR$ such that if $I$ is a compact interval with $I \cap \cE_\alpha = \emptyset$, then for any unit eigenvector $v$ with eigenvalue $\lambda \in I$ and any $1 \leq i \leq n$, 
$$
|\langle v , e_i \rangle | \leq_{st} Z n ^ {-\rho (  1 - \frac 1   \alpha) } ( \log n )^c,
$$ 
where $\rho$ is as in Theorem \ref{th:mainconvloc} and $Z$ is a non-negative random variable whose law depends on $(\alpha, I)$  and which satisfies 
$$\bE \exp ( Z^\delta ) < \infty.$$
\end{proposition}

\begin{proof}
Let $\cE_\alpha$ be as in Theorem \ref{th:mainconvloc}. The density of $\mu_\alpha$ is uniformly lower bounded on $I$ by say $4\varepsilon >0$. We set $$\eta =  c_1  \left( \sqrt{ \frac{Ê\log n}{  n }} \vee \left( n^{ - \frac \alpha {Ê8 - 3 \alpha} }  ( 1 + \ind_{Ê1 < \alpha < 4/3} ( \log n ) ^{ \frac {2 \alpha}{ 8 - 3 \alpha} })\right) \right) ,$$ where the constant $c_1$ is large enough to guarantee that for any interval  $J$ of length at least $\eta$ in $I$ we have $|\bE \mu_A (J)  - \mu_\alpha (J)  | \leq 2 \varepsilon | J|$.  Then, we partition the interval $I = \cup_\ell I_\ell $ into $c_2  \eta^{-1} $ intervals of length $ \eta $. From what precedes we have for any $ 1 \leq \ell \leq c_2 \eta^{-1}$, $\bE \mu_A (I_\ell) > 2 \varepsilon | I_\ell|$. 

Now, by Lemma \ref{le:concspec}, the event $F_n$ that  for all $ 1 \leq \ell \leq c_2 \eta^{-1}$,
\begin{equation*}\label{eq:lowerboundmuA}
\mu_A (I_\ell) > \bE \mu_A (I_\ell) - \varepsilon | I_\ell| >  \varepsilon | I_\ell|,
\end{equation*}
has probability at least $ 1 - c_2  \eta^{-1} \exp ( - n \varepsilon^2 c^2_3 \eta^2  / 2 ) \geq  1 - c \exp ( - c n^{\delta})$ for some constants $c , \delta>0$.

Let $v$ be a unit eigenvector and $\lambda \in I$ such that $A v = \lambda v$. Set $v_i = \langle v , e_i \rangle$. We recall the formula 
$$
v_1 ^2 = ( 1 +  a_n^{-2} \langle X_1 , ( A^{(1)} - \lambda ) ^{-2} X_1 \rangle ) ^{-1} , 
$$
with $X_1 = ( X_{1 2}, \cdots , X_{1 n}) \in \bR^{n-1}$ and $A^{(1)}$ is the principal minor matrix of $A$ where the first row and column have been removed.  We may now argue as in the proof of Proposition \ref{prop:supNI} : for some $1 \leq \ell \leq c_2 \eta^{-1}$, $\lambda \in I_\ell$, and it follows
$$
v_1 ^2 \leq  a_n^{2} c_3 ^2 \eta^2   \left( \sum_{i : \lambda_i ^{(1)} \in I_\ell}    \langle X_1 , u_i ^{(1)} \rangle ^{2} \right)^{-1},
$$
where $(\lambda_i ^{(1)},u_i ^{(1)}), 1 \leq i \leq n-1$ denotes the eigenvalues and an eigenvectors basis of $A^{(1)}$. We rewrite the above expression as
\begin{equation}\label{eq:boundvi}
v_1 ^2   \leq  a_n^{2} c_3 ^2 \eta^2  \dist^{-2}(  X_1, W^{(1)} ) = a_n^{2} c_3 ^2 \eta^2  \langle  X_1, P_1 X_1 \rangle^{-1} ,
\end{equation}
where $W^{(1)} = \mathrm{vect}\left\{ u_{i}^{(1)} : 1 \leq i \leq n-1 , \lambda_i^{(1)} \notin I_\ell \right\},$
and $P_1$ is the orthogonal projection on the orthogonal of $W^{(1)}$. The rank of $P_1$ is equal to 
$$N^{(1)}_{I_\ell}  =  | \{1 \leq i \leq n - 1 : \lambda^{(1)}_i \in I_\ell \} |= n  - 1 -  \mathrm{dim} (W^{(1)}).$$ From Weyl interlacement theorem, we get
\begin{equation}\label{eq:weyl_NI}
n \mu_A (I_\ell) - 1  \leq N^{(1)}_{I_\ell}  \leq  n \mu_A (I_\ell) + 1. 
\end{equation}
From Lemma \ref{le:formulealice},  there exists a positive $\alpha/2$-stable random variable $S$ and a standard Gaussian vector $G$ such that 
$$
 \dist^{2}(  X_1, W^{(1)} ) \stackrel{d}{=} \|P_1 G \|_\alpha ^2 S. 
$$
By Corollary \ref{cor:concnorm} and \eqref{eq:weyl_NI}, if $n$ is large enough, on the event $F_n$, with probability at least 
\begin{equation*}\label{eq:Fn1v}
1 - 2 \exp \left( -  \delta    \frac{  (  \varepsilon c_3  n \eta )^{\frac {2}{ \alpha}}  } { n^{\frac 2 \alpha - 1}}  \right) \geq 1 - 2 \exp ( - c n^{\delta})
\end{equation*}
the lower bound $$
\|P_1 G \|_\alpha \geq \delta  \left(  \varepsilon c_3 n  \eta \right)^{\frac 1 \alpha} 
$$ 
holds. Let us denote this enlarged event by $\bar F_n$. Hence, for some $c >0$, on $\bar F_n$, we have from \eqref{eq:boundvi}
$$
v_1 ^2   \leq  c   \eta^{2( 1 - 1 /  \alpha)}  S^{-1} . 
$$
In summary, we have shown that
$$
|v_1  |   \leq  c   \eta^{ 1 - 1 /  \alpha}  S^{- 1 / 2} + \ind_{\bar F_n^c},
$$
where $\bar F_n^c$ has probability at most  $c \exp ( - c n^{\delta})$, for some $c >1$. For $0 < \delta' < 1$, it yields,
\begin{eqnarray*}
\bE \exp\left\{ \left( \frac{ |v_1| }{  c  \eta^{ 1 - 1 /  \alpha}  } \right)^ {\delta'} \right\} & \leq &  \bE \exp\left\{  S^{- \delta' / 2}  +   \eta^{ \delta'  ( 1 /  \alpha - 1 ) } \ind_{\bar F_n^c}  \right\} \\
& \leq &  \sqrt { \bE e^ { 2 S^{- \delta' / 2}  }  \bE e^ { 2 \eta^{ \delta'  ( 1 /  \alpha - 1 ) } \ind_{\bar F_n^c} }  } \\
& \leq &  \sqrt { \bE e^ { 2 S^{- \delta' / 2}  } ( 1 + e^ { 2 n ^ {\delta'/2} } c e^{  - c n^{\delta}}) },
\end{eqnarray*}
where we have used that $\eta \geq 1/n$ and $\alpha < 2$. Using, Lemma \ref{le:tailS}, if $\delta'$ is small enough, the above is uniformly bounded in $n$. This gives our statement for any $\delta'' < \delta'$.  \end{proof}

\section{Analysis of the limit recursive equation}

\label{sec:RDE}

We next turn to the analysis of the limiting equation describing the resolvent, in case $\alpha<1$.
Let $\cH$ be the set of analytic functions $h : \bC_+ \to \bC_+$ such that for all $z \in \bC_+$, $|h(z)| \leq \Im (z)^{-1}$. We also consider the subset $\cH_0$ of functions of $\cH$ such  that for all $z \in \bC_+$, $h( - \bar z) = - \bar h ( z)$ . For every $n$ and $1 \leq i \leq n$, the function $z \mapsto R (z)_{ii}$ is  in $\cH$. It is proved in \cite{BCC} that $R_{ii}$ converges weakly for the finite dimensional convergence to the random variable $R_0$ in $\cH_0$ which is the unique solution of the recursive distributional equation for all $z \in \bC_+$,
\begin{equation}\label{eq:RDE}
R_0(z) \stackrel{d}{=}  - \left(Êz + \sum_{k \geq 1} \xi_k R_k (z)\right)^{-1},
\end{equation}
where $\{\xi_k\}_{k \geq 1}$ is a Poisson process on $\bR_+$ of intensity measure $\frac{\alpha}{2} x^{Ê\frac \alpha  2 - 1} dx$, independent of $(R_k)_{k\geq 1}$, a sequence of independent copies of $R_0$. In \cite{BCC}, $R_0(z)$ is shown to be the resolvent at a vector of a random self-adjoint operator defined associated to Aldous' Poisson Weighted Infinite Tree. We define $\bar \bC_+ = \{z \in \bC : \Im (z) \geq 0Ê\} = \bC_+ \cup \bR$. In the following statement, we establish a new property of this resolvent. 

\begin{theorem}[Unicity for the resolvent recursive equation]\label{th:unicityRDE}
Let $0  < \alpha < 2/3$.  There exists $E_\alpha > 0 $ such that for any $z \in \bar \bC_+$ with $|z| \geq E_\alpha$, there is a unique random variable $R_0 (z)$ on $\bar \bC_+$ which satisfies
 the distributional equation \eqref{eq:RDE} and   $\bE | R_0(z) |^ {\frac \alpha 2 } < + \infty$. Moreover, for any $0 < \kappa < \alpha / 2$, there exists  $E_{\alpha,\kappa} \geq E_\alpha$ and $c>0$, such that  for any $z \in  \bC_+$ with $|z| \geq E_{\alpha,\kappa}$, 
$$
\bE \Im R_0(z) ^ {\frac \alpha 2 } \leq c  \, \Im ( z ) ^ {\kappa} . 
$$
In particular, if $\Im( z) = 0$, $R_0(z)$ is a real random variable. 
\end{theorem}

The main part of this section is devoted to 
the proof of Theorem \ref{th:unicityRDE}. We will then analyze its consequence on our random matrix and prove Theorem \ref{th:locvect}  in section \ref{sec:prooflocvect}. As usual, it is based on a fixed 
point argument. However, as $R_0$ is complex-valued, it is not enough to get a fixed point argument for the moments of $R_0$ as was done previously in
\cite{BG08,BDG09}. Instead, we prove that moments of  linear combinations
of $R_0$ and its conjuguate satisfy a fixed point equation. We then show that this new
fixed point equation is well defined and for sufficiently large $z$, has a unique solution.  

\subsection{Proof of Theorem \ref{th:unicityRDE}}
We shall give the proof of  Theorem \ref{th:unicityRDE} in this section,
but postpone the proofs of technical lemmas to the next subsection.
By construction $ H (z) = - i R_0 ( z) \in \cK_{1}$ as
well as   $\bar  H (z) = i \bar R_0 ( z) \in \cK_{1}$ (recall that for $\beta \in [0,2]$ , $\cK_\beta = \{z \in \bC : | \arg (z) | \leq \frac {\pi \beta}{2} \}$).
For ease of notation, we define the bilinear form $h.u$ for $ h \in \bC$ and $u \in \cK_1^+ = \cK_1 \cap \bar \bC_+$ given by
$$
h.u = \Re ( u )  h + \Im (u ) \bar h.
$$
Note that if $h \in \cK_1$ then $h.u \in \cK_1$. We set 
$$
\gamma_z ( u ) =   \Gamma (1 - \frac \alpha 2)  \bE ( H (z).u )^{\frac \alpha 2} \; \in \cK_{\alpha /2}.
$$
We let $\cC_\alpha$ (resp. $\cC_\alpha'$) denote the set of continuous functions $g$ from $ \cK_1^+$ to $\cK_{\alpha /2}$ (resp. $\bC$) such that $g( \lambda u ) = \lambda^{\alpha /2 } g(u)$, for all $\lambda > 0$. Then, for $\alpha/2 \le \beta \le 1$,
 we introduce the norm 
$$
\| g  \|_{\beta} = \max_{ u \in S^1_+} | g(u) | + \max_{ u \ne v \in S^1_+} \frac{ | g(u)  - g(v) | }{ |u - v | ^ {  \beta} } \left(  |  i.u | \wedge | i.u |  \right)^ {\beta - \frac  \alpha 2}
$$ 
where $S_+^1=\{u\in  \cK_1^+, |u|=1\}$. 
We then define $ \cH_{\beta}$ (resp. $ \cH'_{\beta}$) as the set of functions $g$ in $\cC_\alpha$ (resp. $\cC_\alpha'$) such that $\| g  \|_\beta$ is finite. Note that $\| g  \|_{\beta} $ contains two parts : the infinite norm and a weighted $\beta$-H\"older norm which get worse as the argument of $u$ or $v$ gets close to $\pi/4$. Notice also that $ \cH'_{\beta}$ is a real vector space and $ \cH_{\beta}$ is a cone.

The starting point of our analysis is that $\gamma_z$ belongs to $\cH_{\beta}$.
\begin{lemma}[Regularity of fractional moments]
\label{le:regFM}
Let $0 < \alpha < 2$ and $z \in \bar \bC_+$,
\begin{itemize} 
\item[-] Let $R_0(z)$ be a solution of \eqref{eq:RDE} such that $\bE | R_0(z) |^ {\frac \alpha 2 } < + \infty$. Then for all $0 < \beta<1$, all $z\in \bar \bC_+ \backslash \{ 0\} $,  $\bE | R_0(z)| ^\beta \leq c  |Ê\Re ( z) |Ê^{-\beta} $ for some constant $c = c(\alpha,\beta)$. 
\item[-]
 Let $H$ be a random variable in $\cK_{1}$ such that $\bE | H|^ { \frac\alpha  2} $ is finite. If we define for $u \in \cK_1 ^ +$, 
$ \gamma( u ) =   \bE (  H.u )^{\frac \alpha 2} $, 
then $\gamma \in \cH_{\beta}$ for all $\alpha/2 \le \beta \le 1$ and $\| \gamma \|_\beta \leq c   \bE | H|^ { \frac \alpha  2}$ for some universal constant $c > 0$. 
\end{itemize}\end{lemma}
Let $ h \in \cK_1$ and $g \in \cH_{\beta}$. We define formally the function given  for  $ u \in S^1_+$ by  
$$
F_h ( g ) (u) = Ê\int_0 ^ {\frac \pi 2}   \hspace{-3pt} d\theta (\sin 2 \theta  )^{\frac \alpha 2 -1}\, \int_{0}^\infty  \hspace{-3pt} dy\,  y^{-\frac{\alpha}{2}-1}Ê \int_0^\infty  \hspace{-3pt} dr \, r^{\frac{\alpha}{2}-1} e^{- r h .e^{Êi \theta }   }\left( e^{-r^{\frac \alpha 2 } g(e^{i\theta})} - e^{- y  r h .  u }e^{-r^{\frac \alpha 2}g(e^{i \theta} + y   u )}\right).
$$
We next see  that $F_{-iz}$ is closely related to a fixed point equation satisfied by $\gamma_z$.  
\begin{lemma}[Fixed point equation for fractional moments]\label{le:fpgamma}
Let $z \in \bar \bC_+$, $ 0 < \alpha  < 2$ and $R_0(z)$ solution of \eqref{eq:RDE} such that   $\bE | R_0(z) |^ {\frac \alpha 2 } < + \infty$. Then for all $u \in \cK_1^+$, $$\gamma_z (u )  = c_\alpha F_{-iz} ( \gamma_z )  (\check u ),$$
where 
$$c_\alpha =   \frac {  \alpha }  { 2^{ \frac \alpha 2 }  \Gamma ( \alpha / 2 ) ^2 \Gamma ( 1 - \alpha / 2 ) } \quad \hbox{ and } \quad \check u = \Im (u) + i \Re (u). $$
\end{lemma}
To prove this lemma, we will properly define and  study the function  $F_h$ at least for
some values of  $(h,\alpha,\beta)$. We shall prove  that 
\begin{lemma}[Domain of definition of $F_h$] \label{le:DefFh}
Let $h \in \cK_1$ with  $|h| \geq 1$, $ 0 < \alpha < 1$ and $\beta$ such that 
$$
\frac \alpha 2  < \beta < 1 - \frac{ \alpha}{ 2}. 
$$
Then $F_h$ defines a map from $\cH_{\beta}$ to $\cH'_{\beta}$, and there exists a constant $c = c(\alpha)$ such that  
$$
\| F_h ( g ) \|_\beta \leq c |h|^{-\frac \alpha 2} ( \|Êg \|_\beta + 1 ) .  
$$
\end{lemma}
We could not prove unfortunately that $F_h$ is a contraction for $\|.\|_\beta$
but  for a weaker and less appealing norm on $\cH'_{\beta}$ which is given for  $\e > 0$ by  : 
$$
\| g  \|_{\beta, \e} = \max_{ u \in S^1_+} | g(u) | | i.u |^ \e + \max_{ u \ne v \in S^1_+} \frac{ | g(u)  - g(v) | }{ |u - v | ^ {  \beta} } \left(  |  i.u | \wedge | i.v |  \right)^ {\beta+ \e}.
$$
It turns out that the map $F_h$ is Lipschitz for this new norm if $\alpha$ is small enough.

\begin{lemma}[Contraction property of $F_h$] \label{le:ContFh}
Let $h \in \cK_1$ with  $|h| \geq 1$, $ 0 < \alpha < 2/3$, $\alpha / 2 < \beta < 1 - \alpha/2$ and  $0 < \e < (1 -
  3 \alpha /2)\wedge (\beta-\alpha/2) $. Then there exists a finite
 constant $c = c(\alpha,\beta,\e)$ such that, for all $f, g \in \cH_{\beta}$, 
$$
\| F_h ( f )  - F_h ( g ) \|_{\beta,\e} \leq c |h|^{- \alpha } ( 1 + \| f \|_\beta + \| g \|_{\beta} ) \|Êf - g \|_{\beta,\e}.  
$$
\end{lemma}
We can now turn to the proof of  Theorem \ref{th:unicityRDE}. To this end 
 define the map $G_z$ on $\cH_{\beta}$ 
\begin{equation} \label{defGz}
G_z : g \mapsto  \left( u \mapsto  c_\alpha F_{-iz} ( g )  (\check u ) \right).
\end{equation}
Then by Lemma \ref{le:DefFh}, if $|z|$ is large enough, any fixed point $g$ of $G_z$ satisfies $\|g \|_\beta \leq c_0 /2$ for some constant $c_0= c_0(\alpha, \beta)$.  By Lemma \ref{le:ContFh}, for any $0 < \e < 1 -  3 \alpha /2 $,  if $|z| \geq E_{\alpha, \e}$ is large enough,  $G_z$ satisfies
$$
\|ÊG_z ( f) - G_z (g) \|_{\beta,\e} \leq    \frac{  1 + \| f \|_\beta + \| g \|_{\beta} }{1 + 2 c_0}  \|Êf - g \|_{\beta,\e}.  
$$
Thus, by Lemma \ref{le:fpgamma}, $\gamma_z$ is the unique solution in $\cH_\beta$ of the fixed point equation 
$
\gamma_z = G_z (\gamma_z). 
$
However, by Lemma \ref{le:gammatochi} below, the law of  $R_0(z)$ which satisfies \eqref{eq:RDE} is uniquely characterized by  its fractional moments $\gamma_z$. Therefore, there is a unique solution to this recursive distributional equation. 

To prove the estimate on $\bE[\Im R_0(z)^{\alpha/2}]$, we start by proving  that $\Im R_0(E)$ vanishes almost surely.
Indeed, we first note that when $z  = E \ne 0$ is real, there is a real solution of the fixed point equation $\gamma_z = G_z (\gamma_z)$. Let us seek for a probability distribution  $P_E$ in $\bR$ such that   \eqref{eq:RDE} holds. We recall that if $y_k$ are non-negative i.i.d. random variables, independent  of $\{\xi_k\}_{k\geq 1}$, then $\sum_{k} y_k \xi_k $ is equal in law to $ ( \bE y_1^{\frac \alpha  2} ) ^{\frac 2  \alpha } \sum_{k} \xi_k$ and $S = \sum_{k} \xi_k$ is a non-negative $\alpha/2$-stable law. Thus, using the Poisson thinning property, 
by definition $P_E$ has to be  the law of
$$-\left(E+  a^{2/\alpha} S - b^{2/\alpha} S' \right)^{-1}$$
if $S$ and $S'$ are independent $\alpha/2$-stable positive laws and  $a = \int \max( x, 0)^{\alpha/2}dP_E(x)$, $b = \int \max( -x, 0)^{\alpha/2}dP_E(x)$. We find the system of equations
\begin{eqnarray*}
a &  =&  \bE  \left((E +  a^{2/\alpha} S - b^{2/\alpha} S')^{-1} \right)_-^{\alpha/2},\\
b & = &\bE  \left((E +  a^{2/\alpha} S - b^{2/\alpha} S')^{-1} \right)_+^{\alpha/2},
\end{eqnarray*}
(where we have used the notation  $(x)_+ = \max ( x, 0)$,  $(x)_- = \max ( -x, 0)$).  Notice that $a S - b S'$ is an $\alpha/2$-stable variable, it has a bounded density. Hence, for any $0 < \alpha < 2$,  $  |ÊE +  a^{Ê\alpha /2} S - b^{Ê\alpha /2} S' |^{-\alpha/2}$ is perfectly integrable. Thus, by construction
$\tilde\gamma_E(u)=\Gamma(1-\frac{\alpha}{2})
\int  (-i u. x)^{\frac{\alpha}{2}}dP_E(x)$ belongs to $\cH_\beta$ and  it is a fixed point of $G_E$. This insures 
the  existence of $a,b\ge 0$ and 
 also the fact that $P_E$ is the law of $R_0(E)$ as soon as $E$ is large
enough so that $G_E$ is a contraction. 
To consider $\gamma_z$ with small imaginary part, we
need the additional lemma 
\begin{lemma}[Continuity of the maps $F_h$] \label{le:contFh}
Let $ 0 < \alpha < 2/3$, $\alpha / 2 < \beta < 1 - \alpha/2$, $0<\e \leq \beta-\frac{\alpha}{2}$ and $0 < \kappa < \alpha/2$. There exists a constant $c  = c( \alpha, \beta, \kappa) >0$, such that for any $h,k \in \cK_1$, $|h|,  |k| \geq 1$, and $g \in \cH_{\beta}$
$$
\|ÊF_h ( g) - F_k (g) \|_{\beta,\e} \leq c ( |h| \wedge |k| ) ^{-\frac \alpha 2 -\kappa}  |Êh - k |^{\kappa} (1 + \|Êg \|_{\beta} ).
$$
\end{lemma}
Set $z = E + i \eta$ with $ |E| \geq E_{\alpha, \e}$ and let $0 < \kappa < \alpha /2$.  Then, by Lemma \ref{le:contFh}, we have
\begin{eqnarray*}
\| \gamma_E -  \gamma_z  \|_{\beta,\e}Ê & =&  | G_E ( \gamma_E ) - G_z( \gamma_z ) | \leq   \|ÊG_E ( \gamma_E ) - G_z( \gamma_E )\|_{\beta,\e} + \|Ê G_z ( \gamma_E ) - G_z( \gamma_z )\|_{\beta,\e} \\
& \leq &  c E  ^{-\frac \alpha 2 -\kappa}  (1 + c_0 ) \eta^\kappa + \frac   1 2 \| \gamma_E -  \gamma_z  \|_{\beta,\e}. 
\end{eqnarray*}
Hence, for some $c' = c' (\alpha,\beta, \kappa)$, we deduce
$$
\| \gamma_E -  \gamma_z  \|_{\beta,\e} \leq c' \eta^\kappa. 
$$
Then, since $\gamma_E ( e^{i \frac \pi 4} )=0$ as $\gamma_E$ is real, for any $u \in S^1_+$, 
\begin{eqnarray*}
\Gamma(1-\frac{\alpha}{2})  \bE \Im  R_0(z) ^{\frac \alpha 2}  & = & | \gamma_z   ( e^{i \frac \pi 4} ) -  \gamma_E ( e^{i \frac \pi 4} )  |  \\
& \leq &   | \gamma_z   ( e^{i \frac \pi 4} ) -  \gamma_z  (u  )  | +  | \gamma_E   ( e^{i \frac \pi 4} ) -  \gamma_E ( u  )  | + | \gamma_z   ( u ) -  \gamma_E ( u  )  | \\
& \leq &  c | u - e^{i \frac \pi 4} |^{\frac \alpha 2} + \| \gamma_z    -  \gamma_E   \|_{\beta, \e} | i . u |Ê^{-\e} \\
& \leq & c''  | u - e^{i \frac \pi 4} |^{\frac \alpha 2} +
 c'' \eta^\kappa  | u - e^{i \frac \pi 4} |^{-\e}.
\end{eqnarray*}
Choosing $u$ such that $ | u - e^{i \frac \pi 4} |$ is of order $\eta ^{\frac{2Ê\kappa }{\alpha  + 2 \eÊ} } $, we deduce that for all $z = E + i \eta$ with $ |E| \geq E_{\alpha, \e}$, $\bE \Im  R_0(z) ^{\frac \alpha 2} $ is bounded by $\eta ^{\frac{Ê\kappa \alpha }{\alpha  + 2 \eÊ} } $ up to a multiplicative constant. Since $\e > 0$ can be arbitrarily small, this concludes the proof of Theorem \ref{th:unicityRDE}.

\subsection{Proofs of technical lemmas} 
We collect in this part the proofs of a few technical results used in the proof of  Theorem \ref{th:unicityRDE}. 

\subsubsection{Proof of Lemma \ref{le:regFM} (Regularity of fractional moments)} 
If $\Re(z)=E$,
$$|R_0(z)|\le |E-\sum \xi_k \Re(R_{k}(z))|^{-1}$$
where $\sum \xi_k \Re(R_{k}(z))$ is equal in law to $aS-bS'$
for two non-negative constants $a,b$ and two independent stable laws $S,S'$. Assume for example that $E > 0$. By conditioning on $S'$ and integrating over $S$, we deduce from Lemma \ref{momentinv} that there exists a finite
constant $c = c(\alpha,\beta)$ so that
$$\bE |R_0(z)|^ \beta \leq \bE | E -  aS + bS' |^{-\beta} \leq  c \bE | E + b S' |^{-\beta} \leq c  E ^{-\beta}.$$
In particular, as $\eta$ goes to $0$, any limit point $R_0(E)$ of $R_0(E + i \eta)$, solution of \eqref{eq:RDE}, satisfies the above inequality. The conclusion of the first point follows. 

To prove the second point, we notice that it is straightforward that $\gamma$ belongs to $\cC_\alpha$. Moreover, for any $\beta\in [\frac{\alpha}{2},1]$,
 there exists a constant $c = c(\alpha, \beta)$ such that for any $x, y$ in $\cK_1$, 
\begin{equation}\label{fondin}
| x^{\frac \alpha 2}  - y^ {\frac \alpha 2}  | \leq c |x - y | ^ { \beta} \left(  | x | \wedge | y |  \right)^ { \frac  \alpha 2 - \beta }.\end{equation}
Also, we have, for $u \in \cK_1 ^ +$ and $h \in \cK_1$,
\begin{equation} \label{eq:boundDP}
|i.u | |h| \leq | h.u |  \leq \sqrt 2 |u | |h| .
\end{equation}
Indeed, if $ u = s + i t$, $s , t \geq 0$, then $|h.u|^ 2 = \Re(h)^ 2 ( s + t)^2 + \Im(h)^ 2 ( s - t)^2$. This last expression is bounded 
from below by $ |h|^2 ( (s+ t)^ 2 \wedge (s- t)^ 2 ) = |h|^2  (s- t)^ 2 =
 |i.u |^2  |h|^2$. While it is bounded from above by $ |h|^ 2  ( ( s + t)^2 + ( s - t)^2 )= 2 |h|^2| u |^2$.

Now, using Jensen inequality and \eqref{eq:boundDP}, we find
$$
\left| \bE (   H.u )^{\frac \alpha 2} \right| \leq \bE  \left| (  H.u )^{\frac \alpha 2} \right| \leq (\sqrt 2  |u |)^{\frac \alpha 2}  \bE  | H | ^{\frac \alpha 2},
$$
whereas \eqref{fondin} and \eqref{eq:boundDP} imply
\begin{eqnarray*}
\left| \bE (   H.u )^{\frac \alpha 2}  -  \bE (   H.v )^{\frac \alpha 2} \right| & \leq &
  \bE \left| (   H.u )^{\frac \alpha 2}  -  (   H.v )^{\frac \alpha 2} \right| \\
& \leq & c\, \bE | H.(u - v) | ^ {\beta}  \left(  |  H.u | \wedge | H.v |  \right)^ { \frac  \alpha 2 - \beta } \\
& \leq & c\, 2^ {\frac \beta 2} \, | u - v | ^ {\beta}    \left(  |  i.u | \wedge | i.v |  \right)^ { \frac  \alpha 2 - \beta } \bE  | H | ^{\frac \alpha 2}. 
\end{eqnarray*}
This completes the proof with $\|\gamma\|_\beta\le  \left( \sqrt{2}^{\frac \alpha 2} +c2^{\frac{\beta}{2}} \right)  \bE  | H | ^{\frac \alpha 2}$.

\subsubsection{Proof of Lemma \ref{le:fpgamma} (Fixed point equation for fractional moments)} 

Write $u = u_1 + i u_2$ and $-iz = h \in \cK_1$. By definition 
\begin{eqnarray*}
\gamma_z ( u)  & = & \Gamma \left( 1 - \frac \alpha 2 \right) \bE \left(  \frac{u_1}{ h + \sum_k \xi_k H_k } + \frac{u_2}{  \bar h + \sum_k \xi_k \bar H_k }  \right)^{\frac \alpha 2 } \\
& = & \Gamma \left( 1 - \frac \alpha 2 \right) \bE \left(  \frac{\check  u . h + \sum_k \xi_k \check H_k.u  }{ \left| h + \sum_k \xi_k H_k \right|^2 } \right)^{\frac \alpha 2 } .
\end{eqnarray*}
We use the formulas, for all $w \in \cK_1$, $\gamma > 0$,
\begin{eqnarray*}
| w |^{-2 \gamma} =   ( \bar w )^{-\gamma}  ( w )^{-\gamma}  & = & \Gamma ( \gamma ) ^{-2} \int _{ [ 0,\infty) ^2} dx dy \, x^{\gamma - 1}  y^{\gamma - 1}  e^{ -x \bar w - y w} \\
& = & \Gamma ( \gamma ) ^{-2} 2^{ 1- \gamma}  \int_0 ^ {\frac \pi  2} d\theta  \sin ( 2 \theta) ^{ \gamma - 1}  \int_0 ^ \infty dr  \, r^{2\gamma - 1}  e^{ -r  e^{ i \theta}.w} .
\end{eqnarray*}
and for $0 < \gamma <  1$,
\begin{eqnarray*}
 w ^{\gamma}  & = & \gamma  \Gamma (  1- \gamma )^{-1} \int _{0}^\infty dx  \, x^{-\gamma - 1}  (1 - e^{ -x w} ). 
 \end{eqnarray*}
Formally, we find that $\gamma_z ( u) $ is equal to 
 \begin{align*}
& c_\alpha   \int_0 ^ {\frac \pi  2} d\theta  \sin ( 2 \theta) ^{ \frac \alpha 2 - 1} \int _{0}^\infty dx  \,  x^{-\frac \alpha 2 - 1} \\
& \quad \quad \times \int_0 ^ \infty dr  \, r^{\alpha - 1}  \bE \left( e^ { - r e^{ i \theta} . h - \sum_k \xi_k  r  e^{ i \theta} . H_k  } - e^ { -   ( r e^{ i \theta} + x \check u ) . h-  \sum_k \xi_k ( r  e^{ i \theta} + x  \check u ). H_k  } \right). 
\end{align*}
If we perform the change of variable $x = r y$ and apply Levy-Kintchine formula : we obtain the stated formula. The exchange of expectation and integrals is then justified by invoking Lemmas \ref{le:regFM} and \ref{le:DefFh}.

\subsubsection{A key auxiliary lemma}
The next lemma  will be used repeatedly. 

\begin{lemma}[Consequences of {\cite[Lemma 3.6]{BDG09}}]
\label{le:36BDGbis}
Let $0 < \alpha < 2$, $\gamma > 0$ and $0 < \kappa \leq 1$, there exists a constant $c = c ( \alpha, \gamma)  >0$ such that for all $h \in \cK_1$ and $x \in \cK_{\frac \alpha 2}$,
\begin{equation}\label{eq:36BDG}
\left| \int_0^\infty   r^{\gamma -1}
e^{- r h  } e^{-r^{\frac \alpha 2 } x }dr \right| \leq  c |h|^{Ê-\gamma}.\end{equation}
For all $h , k \in \cK_1$ and $x , y  \in \cK_{\frac \alpha 2}$, 
\begin{equation}\label{eq2}
 \left|Ê \int_0^\infty   r^{\gamma  -1} e^{- r h  } \left( e^{-r^{\frac \alpha 2 } x } - e^{-r^{\frac \alpha 2 } y } \right)  dr  \right| \leq  c |h|^{Ê-\gamma - \frac \alpha 2} |Êx - y |,
 Ê\end{equation}
and 
\begin{equation}\label{eq3}
\left|Ê \int_0^\infty   r^{\gamma  -1} e^{-r^{\frac \alpha 2 } x }  \left( e^{- r h  }  - e^{- r k  }   \right)  dr  \right| \leq  c  ( |h| \wedge |k| ) ^{Ê-\gamma - \kappa} |Êh - k |^{\kappa}.\end{equation}
For all $x_1, x_2, y_1, y_2 \in \cK_{\frac \alpha 2}$, 
\begin{align}
& \left|Ê \int_0^\infty   r^{\gamma -1}   e^{- r h }  \left(  \left( e^{-r^{\frac \alpha 2} x_1} - e^{-r^{\frac \alpha 2}  y_1}  \right) -  \left( e^{-r^{\frac \alpha 2 } x_2  }- e^{-r^{\frac \alpha 2 } y_2 }\right) \right) dr  \right|\label{qwe1} \\
&  \leq c   \left( | h |Ê^{- \gamma - \frac \alpha 2 }    | x_1 - x_2 - y_1 + y_2| + | h |Ê^{- \gamma - \alpha}  ( |Êx_1 - x_2 | + |Êy_1 - y_2 |   ) ( |Êx_1 - y_1 | +   |Êx_2 - y_2 |)  \right).\nonumber 
\end{align}
Moreover, for all $h , k  \in \cK_1$, $ x, y \in \cK_{\frac \alpha 2}$, 
for   $0 < \kappa \leq 1$, we have
\begin{align}
& \left|Ê \int_0^\infty   r^{\gamma -1} \left( e^{- r h  }  - e^ {- rk} \right)  \left( e^{-r^{\frac \alpha 2} x} - e^{-r^{\frac \alpha 2}  y} \right)  dr  \right|  \leq c (  | h |Ê\wedge |k| ) ^{- \gamma -  \frac \alpha 2 - \kappa}    | h - k |^{\kappa} | x - y|,\label{qwe2}
\end{align}
and finally, for all $0 \leq \kappa_1,  \kappa_2 \leq 1$, $h_1 , k_1 , h_2, k_2  \in \cK_1$, with $N =  | h_1 |Ê\wedge |k_1| \wedge |h_2| \wedge |k_2|$, 
\begin{align}
& \left|Ê \int_0^\infty   r^{\gamma -1} \left( e^{- r h_1  }  - e^{- r h_2  }  - e^ {- r k_1}  + e^ {- r k_2} \right)   e^{-r^{\frac \alpha 2} x}  dr  \right| \label{qwe3} \\
& \leq c  N ^{- \gamma - \kappa_1}       | h_1 - h_2 - k_1 + k_2|^{\kappa_1}   + c N^{- \gamma -\kappa_2 - \kappa_1} ( |Êh_1 - h_2 | + |Êk_1 - k_2 |   )^{\kappa_2} ( |Êh_1 - k_1 | +   |Êh_2 - k_2 |)^{\kappa_1} .\nonumber
\end{align}
\end{lemma}
\begin{proof}
The bound \eqref{eq:36BDG} is \cite[Lemma 3.6]{BDG09}.
 The bound \eqref{eq2} follows from \eqref{eq:36BDG} by taking derivative and using the convexity of  $\cK_{\frac \alpha 2}$. For \eqref{eq3}, assume for example that $|h| \leq |k|$. From \eqref{eq:36BDG}, we may assume  that $Ê|h - k | \leq |h| / 2$. Then taking derivative, we find
$$
\left|Ê \int_0^\infty   r^{\gamma -1} e^{-r^{\frac \alpha 2 } x }  \left( e^{- r h  }  - e^{- r k  }   \right)  dr  \right| \leq  c  N ^{Ê-\gamma-1} | h - k|,
$$
where $N = \min \{ | t h + (1 - t) k | : 0 \leq t \leq 1\}$. Since $Ê|h - k | \leq |h| / 2$ then $ N \geq |h| / 2$ and 
$$
\left|Ê \int_0^\infty   r^{\gamma -1} e^{-r^{\frac \alpha 2 } x }  \left( e^{- r h  }  - e^{- r k  }   \right)  dr  \right| \leq  c |h| ^{Ê-\gamma-1} | h - k| \leq c |h| ^{Ê-\gamma-\kappa} | h - k| ^ \kappa.
$$
To prove \eqref{qwe1}, \eqref{qwe2} and \eqref{qwe3}, we need to take
derivatives and use the first inequalities. Namely, for \eqref{qwe1}, we  
define the function on $[0,1]$, 
$$\varphi (t ) =  \int_0^\infty   r^{\gamma -1} e^{- r h } \left( e^{-r^{\frac \alpha 2 } x(t)  }  - e^{-r^{\frac \alpha 2 } y( t)  } \right) dr,
$$
with $x(t) =  t x_1 + (1 - t) x_2$ and $y(t) = t y_1 + (1 - t) y_2 $. We are looking for an upper bound on $|\varphi(1) - \varphi(0)|$. A straightforward computation yields 
\begin{eqnarray*}
 \varphi' (t)  & = &  \int_0^\infty   r^{\gamma + \frac \alpha 2 -1} e^{- r h } \left( (x_1 - x_2) e^{-r^{\frac \alpha 2} x (t) } -  (y_1 - y_2) e^{-r^{\frac \alpha 2}  y(t) } \right)  dr  \\
&  = &    (x_1 - x_2)  \int_0^\infty   r^{\gamma  + \frac \alpha 2  -1} e^{- r h  } \left( e^{-r^{\frac \alpha 2 } x(t)  }  - e^{-r^{\frac \alpha 2 } y(t) }  \right) dr \\
& & \quad + (x_1 - x_2 - y_1 + y_2)  \int_0^\infty   r^{\gamma   + \frac \alpha 2  -1} e^{- r h  } e^{-r^{\frac \alpha 2 } y(t) }   dr .
\end{eqnarray*}
By convexity, we note that $x(t), y(t)$ are in $\cK_{\frac \alpha 2}$. Hence using \eqref{eq:36BDG} and \eqref{eq2}, we may upper bound $|\varphi' (t)| $, up to a multiplicative constant, by  
\begin{align*}
& |h|^{-\gamma - \alpha}|x_1 - x_2| |x(t) - y(t)| +  |h|^{- \gamma   -\frac \alpha 2} |x_1 - x_2 - y_1 + y_2| \\
 & \leq |h|^{-\gamma - \alpha} |x_1 - x_2|  |x_1 - y_1| \vee |x_2 - y_2| +  |h|^{- \gamma   -\frac \alpha 2} |x_1 - x_2 - y_1 + y_2|. 
\end{align*}
This completes the proof of \eqref{qwe1}. For \eqref{qwe2}, we first first notice that splitting $h$ and $k$ and arguing as for the proof of \eqref{eq3}, it is sufficient to prove the statement for $|h-k| \leq (|h| \wedge |k| )/ 2$ and $\kappa= 1$. Then, with $h(t) = t h + (1 - t)k$, we consider this time the function
$$
\psi (t ) =  \int_0^\infty   r^{\gamma -1} e^{- r h(t)  } \left( e^{-r^{\frac \alpha 2 } x }  - e^{-r^{\frac \alpha 2 } y } \right) dr,
$$
and take the derivative. We find that it is proportional to $(h-k)(\varphi(1)-\varphi(0))$ with  the previous function $\varphi$ with $\gamma$ replaced by $\gamma+1$ and $x_1=x_2=x,y=y_1=y_2$. The bound is therefore clear. The proof is identical for the third statement for $\kappa_1 = \kappa_2 =1$. As above, we then generalize it to any $0 \leq \kappa_1,  \kappa_2 \leq 1$ by using the rough bound given by \eqref{eq3}.  \end{proof}

\subsection{Properties of the map $F$}
\begin{proof}[Proof of Lemma \ref{le:DefFh}]
We start by proving that for all $u \in S^1_+$,  for $h\in\cK_1$, $|h|\ge 1$,
\begin{equation}\label{eq:INFg}
| F_h ( g) (u ) |Ê\leq c |h|^{-\frac \alpha 2} ( \|Êg \|_\beta + 1 ).  
\end{equation}
By Lemma \ref{le:36BDGbis}, for $h \in \cK_1$, the map on $\cK_{\frac \alpha 2}$ %to $\cK_{\frac \alpha 2}$
given by 
 $$x \mapsto \int_0^\infty  \hspace{-3pt} dr \, r^{\frac{\alpha}{2}-1}
e^{- r h  } e^{-r^{\frac \alpha 2 } x } , $$
is bounded by $c |h|^{-\alpha/2}$ and Lipschitz with constant $c |h|^{-\alpha}$. Let $T > 0 $
to  be chosen later on. From \eqref{eq:boundDP} and \eqref{eq:36BDG}, for $\theta\in [0,\frac{\pi}{2}]$ and
$h\in  \cK_1$, $u\in S_1^+$, $g:\cK_1\rightarrow \cK_{\frac \alpha 2}$, we have
\begin{align}
&  Ê\int_{T}^\infty  \hspace{-3pt} dy\,  y^{-\frac{\alpha}{2}-1} \left|Ê \int_0^\infty  \hspace{-3pt} dr \, r^{\frac{\alpha}{2}-1}
e^{- r h .e^{Êi \theta }   }\left( e^{-r^{\frac \alpha 2 } g(e^{i\theta})} - e^{- y  r h .  u }e^{-r^{\frac \alpha 2}g(e^{i \theta} + y   u )}\right) \right|Ê \nonumber \\
&\leq Ê\int_{T}^\infty  \hspace{-3pt} dy\,  y^{-\frac{\alpha}{2}-1} 
\left( \left|Ê \int_0^\infty  \hspace{-3pt} dr \, r^{\frac{\alpha}{2}-1}
e^{- r h .e^{Êi \theta }   } e^{-r^{\frac \alpha 2 } g(e^{i\theta})}\right|+
\left|Ê \int_0^\infty  \hspace{-3pt} dr \, r^{\frac{\alpha}{2}-1}
e^{- r h .(e^{Êi \theta }+yu)    } e^{-r^{\frac \alpha 2 } g(e^{i \theta} + y   u )}
\right|\right) \nonumber \\
& \leq  c \int_{T}^\infty  \hspace{-3pt} dy\,  \frac{ y^{-\frac{\alpha}{2}-1} } {| h | ^{ \frac \alpha 2}  | i. e^{i\theta} | ^{ \frac \alpha 2} }  + c  \int_T^\infty  \hspace{-3pt} dy\,  \frac{ y^{-\frac{\alpha}{2}-1} } {| h | ^{ \frac \alpha 2}  | i.(e^{i\theta} + y u ) | ^{ \frac \alpha 2} } \nonumber \\
& \leq c' | h | ^{ - \frac \alpha 2}  \left| \theta - \frac \pi 4 \right| ^{ - \frac \alpha 2} T^{-\frac \alpha 2}, \label{eq:IN1I}
\end{align}
where we have used the fact that $ |  i.e^{i\theta} |  = |Ê\cos \theta - \sin \theta | \geq c | \theta - \pi / 4Ê|$ and the control, for any real $t , \delta$, $T>0$, any $\gamma_1<\gamma_2$, $\gamma_1\neq 0$, (here $\gamma_2=-\gamma_1=\alpha/2)$,
\begin{equation}\label{contint}
 \int_{T}^\infty  \frac{ y^{\gamma_1-1} } {  | y t - \delta | ^{ \gamma_2} } d y
= |\delta|^{\gamma_1-\gamma_2} |t|^{-\gamma_1}  \int_{  \frac{ÊT | t |}{ | Ê\delta |Ê}  } ^\infty  \hspace{-3pt} 
\frac{ x^{\gamma_1-1} } {  | x \pm 1 | ^{ \gamma_2} }
 dx 
\le c (T^{\gamma_1}|\delta|^{-\gamma_2} 1_{\gamma_1<0}+ |\delta|^{\gamma_1-\gamma_2} |t|^{-\gamma_1} 1_{\gamma_1>0})\,,\end{equation}
 where the sign depends on whether or not $t, \delta$ have a different sign. 
 
For the integration over $y$ in the interval $[0,T]$, we find similarly by \eqref{eq2} that
\begin{eqnarray*}
A&:=& \int_0^{T}  \hspace{-3pt} dy\,  y^{-\frac{\alpha}{2}-1} 
 \left|Ê \int_0^\infty  \hspace{-3pt} dr \, r^{\frac{\alpha}{2}-1}
e^{- r h .e^{Êi \theta }   }\left( e^{-r^{\frac \alpha 2 } g(e^{i\theta})} - e^{-r^{\frac \alpha 2}g(e^{i \theta} + y   u )}\right) \right|\\
 &\leq & c \int_0^{T}  \hspace{-3pt} dy\,  \frac{ y^{-\frac{\alpha}{2}-1} } {| h | ^{   \alpha}  |  i.e^{i\theta} | ^{ \alpha} }  |  g(e^{i\theta}) - g(e^{i \theta} + y   u )|\,. \end{eqnarray*}
Recalling that $g(z )=|z|^{\frac{\alpha}{2}}
g(\frac{z}{|z|})$ and using \eqref{fondin}-\eqref{eq:boundDP} we find that there exist
finite constants $C,C'$ so that  for all $z,z'\in {\mathcal K}_1^+$,
\begin{eqnarray}
|g(z)-g(z')|&\le& \|g\|_\beta\left( \left| |z|^{\frac{\alpha}{2}}-|z'|^{\frac{\alpha}{2}}\right|
+(|z|\wedge |z'|)^{\frac{\alpha}{2}} \left| \frac{z}{|z|}-\frac{z'}{|z'|}\right|^\beta
\left(|i.\frac{z}{|z|}|\wedge |i.\frac{z'}{|z'|}|\right)^{\frac{\alpha}{2}-\beta}\right)\nonumber\\
&\leq& C\|g\|_\beta\left( (|z|\wedge |z'|)^{\frac{\alpha}{2}-\beta}+
(|i.z|\wedge|i.z'|)^{\frac{\alpha}{2}-\beta}\right)|z-z'|^\beta \nonumber \\
&\leq& C' \|g\|_\beta  (|i.z|\wedge|i.z'|)^{\frac{\alpha}{2}-\beta} |z-z'|^\beta .
\label{contlipg}\end{eqnarray}
Using the fact that $|e^{i \theta} + y   u|\ge 1$ 
as $u, e^{i\theta}\in S_1^+, y\ge 0$, we find 
with  \eqref{eq:boundDP}
that 
\begin{align}
A& \leq c | h | ^{ -  \alpha}   |i. e^{i \theta} | ^{ -\alpha} \|Êg \|_\beta \left( \int_0^{T} dy\,  \frac{ y^{\beta-\frac{\alpha}{2}-1} }{ |  i.(  e^{i \theta} + y   u )| ^{ \beta - \frac \alpha  2} } +  \int_0^{T} dy\,  \frac{ y^{\beta-\frac{\alpha}{2}-1} }{|  i.e^{i \theta} | ^{ \beta - \frac \alpha  2}Ê} \right) \nonumber \\
& \leq  c' | h | ^{ -  \alpha}   | \theta - \frac \pi 4 | ^{- \beta -\frac \alpha 2 } \|Êg \|_\beta  T^{\beta -\frac{\alpha}{2}}, \label{eq:IN01}
\end{align}
where we have used that $\beta > \alpha /2$ to obtain a convergent integral.

In the integration over $y$ on the interval $[0,T]$, we have left aside the term
\begin{align}
&  \int_0^{T}  \hspace{-3pt} dy\,  y^{-\frac{\alpha}{2}-1} \int_0^\infty  \hspace{-3pt} dr \, r^{\frac{\alpha}{2}-1}
e^{- r h .e^{Êi \theta }   } e^{-r^{\frac \alpha 2}g(e^{i \theta} + y   u )} \left( 1 - e^{- y  r h.  u }  \right)Ê \nonumber . 
\end{align}
We  shall use this time the third statement \eqref{eq3} of Lemma \ref{le:36BDGbis} with $\kappa = 1$. We choose $T = |i.e^{i \theta} |Ê/2$ so that for all $y \in [0,T]$ from \eqref{eq:boundDP}
$$
| h . ( e^{i \theta} + y   u ) | \geq | h | |  i.e^{ i \theta} |Ê-   \sqrt 2   |Êh | T \geq ( 1 - \frac{Ê1 }{Ê\sqrt 2} ) | h | | i. e^{ i \theta} |. 
$$
For this choice of $T$, we get 
\begin{align}
&   \int_0^{T}  \hspace{-3pt} dy\,  y^{-\frac{\alpha}{2}-1}  \left|Ê  \int_0^\infty  \hspace{-3pt} dr \, r^{\frac{\alpha}{2}-1}
e^{- r h .e^{Êi \theta }   } e^{-r^{\frac \alpha 2}g(e^{i \theta} + y   u )} \left( 1 - e^{- y  r h.  u }  \right)\right| Ê \nonumber \\
& \leq c  \int_0^{T}  \hspace{-3pt} dy\,  \frac{Êy^{-\frac\alpha 2 }  }{Ê| h|^{ \frac \alpha 2} |Êi. e^{Êi \theta} |^{Ê\frac \alpha  2 + 1} Ê}   \leq c' | h|^{  - \frac \alpha 2} \left|Ê\theta - \frac \pi 4 \right|^{ - \frac \alpha 2 - 1}  T^{1 - \frac \alpha 2}Ê\label{eq:IN01bis}. 
\end{align}
Finally, using our choice of $T$, we deduce from (\ref{eq:IN1I}),(\ref{eq:IN01}),(\ref{eq:IN01bis}) that
$$
\int_{0}^\infty \hspace{-3pt} dy\,  y^{-\frac{\alpha}{2}-1}  \left|Ê  \int_0^\infty  \hspace{-3pt} dr \, r^{\frac{\alpha}{2}-1}
e^{- r h .e^{Êi \theta }   }\left( e^{-r^{\frac \alpha 2 } g(e^{i\theta})} - e^{- y  r h .  u }e^{-r^{\frac \alpha 2}g(e^{i \theta} + y   u )}\right) \right|Ê
$$
is bounded by $c | h | ^{Ê- \frac \alpha 2} | \theta - \frac \pi 4 | ^{Ê-\alpha} ( 1+ \|Êg \|_\beta )$ for $|h|\ge 1$. We obtain \eqref{eq:INFg} since $$
| \theta - \frac \pi 4 | ^{Ê-\alpha}( \sin 2 \theta ) ^{\frac \alpha 2 - 1} 
$$ 
is integrable over $[0, \pi /2]$. 

The proof of the lemma will be complete if we prove that for all $u, v \in S_1 ^ + $, 
\begin{equation}\label{eq:HNFg}
| F_h ( g) (u )  - F_h ( g) (v ) | \leq c |Êu -v | ^\beta ( | i. u  | \wedge | i. v  |  ) ^{  \frac \alpha  2 - \beta} ( 1 + \| g \|_\beta ) |h | ^{ -\frac \alpha 2}. 
\end{equation}
To do so, we fix $\theta \in [0, \pi /2]$ and assume for example that $|  i. (  e^{i \theta} + y   u ) |Ê \leq |  i. (  e^{i \theta} + y   v ) |$.  We first use the Lipschitz bound \eqref{eq2}, together with \eqref{contlipg},
 and write
\begin{align}
&  Ê\int_{0}^\infty  \hspace{-3pt} dy\,  y^{-\frac{\alpha}{2}-1} \left|Ê \int_0^\infty  \hspace{-3pt} dr \, r^{\frac{\alpha}{2}-1}
e^{- r h .e^{Êi \theta }   }\left( e^{- y  r h .  u }e^{-r^{\frac \alpha 2}g(e^{i \theta} + y   u )} - e^{- y  r h .  u }e^{-r^{\frac \alpha 2}g(e^{i \theta} + y   v )}\right) \right|Ê \nonumber \\
& \leq  c \int_{0}^\infty  \hspace{-3pt} dy\,  \frac{ y^{-\frac{\alpha}{2}-1} } {| h | ^{ \alpha}  | i . (e^{i\theta} + y u ) | ^ \alpha } |Êg(e^{i \theta} + y   u ) - g(e^{i \theta} + y   v )|Ê \nonumber \\
& \leq c | h | ^{ -  \alpha}  |Êu - v | ^\beta \|Êg \|_\beta \int_0^{\infty} dy\,  \frac{ y^{\beta-\frac{\alpha}{2}-1} }{ | i. (  e^{i \theta} + y   u ) | ^{ \beta+\frac \alpha 2} }  \nonumber \\
& \leq c' | h | ^{ - \alpha}   |Êu - v | ^\beta \|Êg \|_\beta \left| \theta - \frac \pi 4 \right| ^{ -  \alpha } ( | i. u  | \wedge | i. v  |  ) ^{  \frac \alpha  2 - \beta}, \label{eq:HN1I}
\end{align}
where we have used \eqref{contint} with $\gamma_1=\beta-\alpha/2>0$ and $\gamma_2=\kappa>\gamma_1$.

Now, in our control of 
$$
Ê\int_{0}^\infty  \hspace{-3pt} dy\,  y^{-\frac{\alpha}{2}-1}  \int_0^\infty  \hspace{-3pt} dr \, r^{\frac{\alpha}{2}-1} e^{- r h .e^{Êi \theta }   }\left( e^{- y  r h .  u }e^{-r^{\frac \alpha 2}g(e^{i \theta} + y   u )} - e^{- y  r h .  v }e^{-r^{\frac \alpha 2}g(e^{i \theta} + y   v )}\right) 
$$
we have so far left aside
$$
Ê\int_{0}^\infty  \hspace{-3pt} dy\,  y^{-\frac{\alpha}{2}-1}  \int_0^\infty  \hspace{-3pt} dr \, r^{\frac{\alpha}{2}-1} e^{- r h . ( e^{Êi \theta }   + y u ) } e^{-r^{\frac \alpha 2}g(e^{i \theta} + y   v )} \left( 1- e^{- y  r h .  (v-u) }\right) 
$$
where $|  i. (  e^{i \theta} + y   u ) |Ê \leq |  i. (  e^{i \theta} + y   v ) |$. By \eqref{eq3}
 applied to $\kappa = \beta$,
$$
\left|Ê \int_0^\infty  \hspace{-3pt} dr \, r^{\frac{\alpha}{2}-1} e^{- r h . ( e^{Êi \theta }   + y u ) } e^{-r^{\frac \alpha 2}g(e^{i \theta} + y   v )} \left( 1- e^{- y  r h .  (v-u) }\right) \right|Ê
$$
 is bounded up to multiplicative constant by 
$$
|h |^{Ê-\frac \alpha 2-\beta} |Êi .(e^{i \theta} + y u ) | ^{Ê - \frac \alpha 2 - \beta}  y^\beta |Êv - u |^ \beta.
$$
Using again \eqref{contint} with $0<\gamma_1=\beta-\alpha/2<\gamma_2=\beta+\alpha/2$ yields
\begin{align}
 & Ê\int_{0}^\infty \hspace{-3pt} dy\,  y^{-\frac{\alpha}{2}-1} \left|Ê \int_0^\infty  \hspace{-3pt} dr \, r^{\frac{\alpha}{2}-1} e^{- r h . ( e^{Êi \theta }   + y u ) } e^{-r^{\frac \alpha 2}g(e^{i \theta} + y   v )} \left( 1- e^{- y  r h .  (v-u) }\right) \right|Ê \nonumber \\
& \leq  c | h | ^{ -  \frac \alpha 2  -\beta}  |Êu - v |^\beta |Ê\delta |Ê^{Ê- \alpha }  [ | i.u|Ê ^{\frac{\alpha}{2} - \beta 
}+ | i.v|Ê ^{\frac{\alpha}{2} - \beta} ]\,.  \label{eq:finalHNe}\end{align}
 We may conclude the proof of \eqref{eq:HNFg}  by noticing that the bounds given by \eqref{eq:HN1I}-\eqref{eq:finalHNe} and mutliplied by $( \sin 2 \theta ) ^{\frac \alpha 2 - 1}$ are uniformly integrable on $[0, \pi /2]$.  \end{proof}

We can now build upon the proof of Lemma \ref{le:DefFh} to get proofs for Lemmas \ref{le:ContFh} and \ref{le:contFh}.

\begin{proof}[Proof  of Lemma \ref{le:ContFh}]
We shall now use the norm $\|.\|_{\beta,\e}$ for which we have the following analogue of \eqref{contlipg}:
if $0 \leq \e\le \beta-\frac{\alpha}{2}$, 
for all $z,z'\in {\mathcal K}_1^+$, 
\begin{eqnarray}
|f(z)|&\le & \|f\|_{\beta,\e} \frac {|z|^{\frac{\alpha}{2}+\e}}{|i.z|^\e}\label{lip1}\\
|f(z)-f(z')|&\le & C  \|f\|_{\beta,\e}\frac{(|z|\vee |z'|)^{\frac{\alpha}{2}+\e}}{(|i.z|\vee|i.z'|)^{\beta+\e}} |z-z'|^\beta\label{lip2}. \end{eqnarray}
We start by showing that for any $u \in S^1_+$, 
\begin{equation}\label{eq:FhgfU}
|F_h ( g) (u)  - F_h (f) (u) |  \leq c  |h|^{-\alpha } ( 1 + \| f \|_\beta + \| g \|_{\beta} )  \| f - g \|_{\beta,\e}  |i.u|^{ -\e}.
\end{equation}
The proof is similar to the argument in Lemma \ref{le:DefFh}. We notice that $F_h ( g) (u )   - F_h (f) (u)$ is equal to 
$$
  \int_0 ^{\frac \pi 2} \hspace{-3pt} d \theta (\sin 2\theta)^{ \frac \alpha 2 - 1} \int_0 ^ \infty \hspace{-3pt} dy\,  y^{-\frac{\alpha}{2}-1}  \int_0^\infty  \hspace{-3pt} dr \, r^{\frac{\alpha}{2}-1} Z(r,y,\theta), 
$$
where, with $g_u = g ( e^{i \theta} + y u )$, $f_u = f ( e^{i \theta} + y u )$, $h_u  = h . (e^{i\theta}  + y u ) $, 
\begin{eqnarray*}
Z & = &   e^{ - r h_0 } \left( e^{ - r^{\frac \alpha 2} g_0 } - e^{ - r^{\frac \alpha 2} f_0 } \right) - e^{ - r h_u } \left( e^{ - r^{\frac \alpha 2} g_u } - e^{ - r^{\frac \alpha 2} f_u } \right) \\
& = & \left( e^{ - r h_0 } - e^{ - r h_u } \right)  \left( e^{ - r^{\frac \alpha 2} g_u } -  e^{ - r^{\frac \alpha 2} f_u } \right) +  e^{ - r  h_0 }  \left(e^{ - r^{\frac \alpha 2} g_0 } - e^{ - r^{\frac \alpha 2} f_0 }  -  e^{ - r^{\frac \alpha 2} g_u} + e^{ - r^{\frac \alpha 2} f_u } \right). 
\end{eqnarray*}
We set $\delta  = - i .  e^{i\theta} $ and $ t  = i . u$. On the integration interval $[T, \infty)$ of $y$ we use the first form of $Z$ and treat the two terms separately. As in Lemma \ref{le:DefFh}, we use \eqref{eq2} and \eqref{lip1}
to find
\begin{align*}
&  Ê\int_{T}^\infty  \hspace{-3pt} dy\,  y^{-\frac{\alpha}{2}-1} \left|Ê \int_0^\infty  \hspace{-3pt} dr \, r^{\frac{\alpha}{2}-1} \left( 
e^{ - r h_0 } \left( e^{ - r^{\frac \alpha 2} g_0 } - e^{ - r^{\frac \alpha 2} f_0 } \right) - e^{ - r h_u } \left( e^{ - r^{\frac \alpha 2} g_u } - e^{ - r^{\frac \alpha 2} f_u } \right) \right) \right| Ê \nonumber \\
& \leq  c \int_{T}^{\infty }  \hspace{-3pt} dy\, 
 \frac{ y^{-\frac{\alpha}{2}-1} \| f- g \|_{\beta,\e} } {| h | ^{\alpha }  |\delta|^{\alpha+ \e} } + c \int_{T}^\infty  \hspace{-3pt} dy\,  \frac{ y^{-\frac{\alpha}{2}-1} y^{\frac{\alpha}{2}+\e}\vee 1   \| f- g \|_{\beta,\e}  } {| h | ^\alpha  |  t y - \delta | ^{  \alpha + \e} } \nonumber \\
& \leq 
c'   | h | ^{ - \alpha}  \| f - g \|_{\beta,\e}   (T ^{ -\frac \alpha 2}  \delta^{-  \e-\alpha}+ \delta^{-\alpha}|t|^{-\e})\,.
\end{align*}
The above computation requires the hypothesis 
 $\e>0$ to insure that the control of the integrals hold following \eqref{contint} with $\gamma_1\neq 0$.

For the integration  interval $[0, T)$ of $y$ we use the second form of $Z$ and choose $ T = | i . e^{i \theta} | / 2 = |\delta| /2$. We use \eqref{qwe2}, \eqref{lip1} and $\kappa =\alpha /2 + \e$,  we find
\begin{align*}
&\int_{0}^T  \hspace{-3pt} dy\,  y^{-\frac{\alpha}{2}-1} \left|Ê \int_0^\infty  \hspace{-3pt} dr \, r^{\frac{\alpha}{2}-1} \left( e^{ - r h_0 } - e^{ - r h_u } \right)  \left( e^{ - r^{\frac \alpha 2} g_u } -  e^{ - r^{\frac \alpha 2} f_u } \right) \right| Ê 
& \leq  c \int_{0}^{T}  \hspace{-3pt} dy\,  \frac{ y^{\kappa-\frac{\alpha}{2}-1} \| f- g \|_{\beta,\e} } {| h | ^{\alpha + \kappa}  |\delta|^{\alpha + \kappa + \e} } \nonumber \\
& \leq  c'   | h | ^{ -\frac{3\alpha}{2} - \e }  |\delta| ^{ - \frac{3\alpha}{2}-\e} \| f - g \|_{\beta,\e} . 
\end{align*}
Similarly, by  \eqref{qwe1} and \eqref{lip1}, \eqref{lip2} and \eqref{contlipg}, our choice of $T$ gives
\begin{align*}
&  Ê\int_{0}^T  \hspace{-3pt} dy\,  y^{-\frac{\alpha}{2}-1} \left|Ê \int_0^\infty  \hspace{-3pt} dr \, r^{\frac{\alpha}{2}-1}  e^{ - r  h_0 }  \left(e^{ - r^{\frac \alpha 2} g_0 } - e^{ - r^{\frac \alpha 2} f_0 }  -  e^{ - r^{\frac \alpha 2} g_u} + e^{ - r^{\frac \alpha 2} f_u } \right) \right| Ê \nonumber \\
& \leq  c \int_{0}^{T}  \hspace{-3pt} dy\,  \frac{ y^{\beta-\frac{\alpha}{2}-1} \| f- g \|_{\beta,\e} |\delta |^{-\e- \beta} } {| h | ^{\alpha}  |\delta|^{\alpha } } + c \int_{0}^{T}  \hspace{-3pt} dy\,  \frac{ y^{\beta-\frac{\alpha}{2}-1} ( \| f \|_\beta + \| g \|_\beta)  |\delta |^{\frac{\alpha}{2} - \beta}  \| f - g \|_{\beta,\e} |\delta|^{-\e}} {| h | ^{\frac{3\alpha}{2} }  |\delta|^{\frac{3\alpha}{2} }} 
\nonumber \\
& \leq  c'   | h | ^{ - \alpha  }   |\delta| ^{-\frac{3\alpha}{2} - \e} \| f - g \|_{\beta,\e} +  c'   | h | ^{ - \frac{3\alpha}{2} }   |\delta| ^{-\frac{3}{2} - \e} ( \| f \|_\beta + \| g \|_\beta)  \| f - g \|_{\beta,\e}. 
\end{align*}
Since $3 \alpha / 2 + \e < 1$, we may integrate our bounds over $\theta$ and obtain \eqref{eq:FhgfU}.

The proof of the lemma will be complete if we show that
 for any $u \ne v \in S_1^+$, with $|i.u| \leq |i.v|$, 
\begin{align}\label{eq:FhghH}
& |ÊF_h ( g) (u)  - F_h (f) (u)  - ÊF_h ( g) (v)  + F_h (f) (v) |  \\
& \quad \quad \leq c  |h|^{-\alpha }  ( 1 + \| f \|_\beta + \| g \|_{\beta} )  \| f - g \|_{\beta,\e}  | i.u | ^{- \beta - \e} | u - v | ^\beta. \nonumber
\end{align}
The proof is simpler than in the previous case as we do not need to consider
separatly the cases where $y$ are small or large. 
We set  $\delta  = - i .  e^{i\theta} $, $ t  = i . u$, $t' = i.v$, $|t| \leq |t'|$.  Using \eqref{qwe2}
  with $\kappa =\beta $  and \eqref{lip1}, we find 
\begin{align*}
&  Ê\int_{0}^\infty  \hspace{-3pt} dy\,  y^{-\frac{\alpha}{2}-1} \left|Ê \int_0^\infty  \hspace{-3pt} dr \, r^{\frac{\alpha}{2}-1} \left( e^{ - r h_v } - e^{ - r h_u } \right)  \left( e^{ - r^{\frac \alpha 2} g_u } -  e^{ - r^{\frac \alpha 2} f_u } \right) \right| Ê \nonumber \\
%& \leq  c \int_{0}^{\infty}  \hspace{-3pt} dy\,  \frac{ y^{\beta-\frac{\alpha}{2}-1} | u - v |^ \beta \| f- g \|_{\beta,\e}%  ( 1 \vee y^{\frac \alpha 2 + \e }  ) } {| h | ^{\alpha+\beta}  |t y -  \delta|^{\alpha + \beta + \e} } \nonumber \\
& \leq  c' | h| ^{-\alpha-\beta} | u - v |^ \beta \| f- g \|_{\beta,\e} (| \delta|^ { - \frac{3 \alpha} {2} - \e} |t|^ { \frac \alpha 2 - \beta}  +  |\delta|^ { -  \alpha } |t| ^ { - \beta - \e}   ).
\end{align*}
Moreover, using again \eqref{qwe1}, \eqref{lip1}, \eqref{lip2} and \eqref{contlipg} we find 
\begin{align*}
&  Ê\int_{0}^\infty  \hspace{-3pt} dy\,  y^{-\frac{\alpha}{2}-1} \left|Ê \int_0^\infty  \hspace{-3pt} dr \, r^{\frac{\alpha}{2}-1} e^{ - r  h_v }  \left(e^{ - r^{\frac \alpha 2} g_v } - e^{ - r^{\frac \alpha 2} f_v }  -  e^{ - r^{\frac \alpha 2} g_u} + e^{ - r^{\frac \alpha 2} f_u } \right)  \right| Ê \nonumber \\
%& \leq  c \int_{0}^{\infty}  \hspace{-3pt} dy\,  \frac{ y^{\beta-\frac{\alpha}{2}-1} | u - v |^ \beta   ( 1 \vee y^{\frac \%alpha 2 +  \e }  )  \| f- g \|_{\beta,\e}  |t y -  \delta|^{-\beta - \e} } {| h | ^{\alpha}  |t y -  \delta|^{\alpha}}  \\
& \quad \quad + c \int_{0}^{\infty}  \hspace{-3pt} dy\,  \frac{ y^{\beta-\frac{\alpha}{2}-1} | u - v |^ \beta |t y -  \delta|^{\frac { \alpha}{2} - \beta}  (  \| f \|_\beta  + \| g \|_\beta ) \| f- g \|_{\beta,\e}  ( 1 \vee y^{\frac \alpha 2 + \e }  )  |t y -  \delta|^{- \e}  } {| h | ^{\frac{3\alpha}{2} }  |t y -  \delta|^{\frac {3 \alpha}{2}} } \\
& \leq  c' | h| ^{-\alpha} | u - v |^ \beta  (1 + \| f \|_\beta + \| g \|_\beta) \| f- g \|_{\beta,\e} ( | \delta|^ { - \frac{3 \alpha} {2} - \e} |t|^ { \frac \alpha 2 - \beta}   +  | \delta |Ê^ { -  \alpha } |  t |Ê^ { - \beta - \e}   ). 
\end{align*}
Now, by assumption, $ 3 \alpha /2 + \e < 1$ and we may integrate our bounds over $\theta$ and obtain \eqref{eq:FhghH}. \end{proof}

\begin{proof}[Proof of Lemma \ref{le:contFh}] The proof is very close to the previous one,
and we simply outline it. 
We assume for example $| h | \leq |k|$. By Lemma \ref{le:DefFh}, we can also assume that $  |Êh - k | \leq |h|$ and in particular $|k| \leq 2 |h|$. We  first prove that for any $u \in S^1_+$, 
\begin{equation}\label{eq:FhFkU}
|ÊF_h ( g) (u)  - F_k (g) (u) |\leq c  |h|^{-\frac \alpha 2 - \kappa }  |Êh - k |^{\kappa} (1 + \|Êg \|_{\beta} ).
\end{equation}
The expression $F_h ( g) (u )   - F_k (g) (u)$ is equal to 
$$
  \int_0 ^{\frac \pi 2} \hspace{-3pt} d \theta (\sin 2\theta)^{ \frac \alpha 2 - 1} \int_0 ^ \infty \hspace{-3pt} dy\,  y^{-\frac{\alpha}{2}-1}  \int_0^\infty  \hspace{-3pt} dr \, r^{\frac{\alpha}{2}-1} Z(r,y,\theta), 
$$
where, with $g_u = g ( e^{i \theta} + y u )$, $h_u  = h . (e^{i\theta}  + y u ) $, $k_u = k . (e^{i\theta}  + y u ) $,
\begin{eqnarray*}
Z & = &   e^{ - r ^{\frac \alpha 2} g_0 } \left( e^{ - r  h_0 } - e^{ - r  k_0 } \right) - e^{ - r^{\frac \alpha 2} g_u } \left( e^{ - r  h_u } - e^{ - r k_u } \right) \\
& = & \left( e^{ - r ^{\frac \alpha 2} g_0  } - e^{ - r ^{\frac \alpha 2} g_u } \right)  \left( e^{ - r h_0 } -  e^{ -r  k_0 } \right) +  e^{ -  r ^{\frac \alpha 2} g_u }  \left(e^{ - r  h_0 } - e^{ - r k_0 }  -  e^{ - r h_u} + e^{ - r k_u } \right). 
\end{eqnarray*}
Let $T > 0$. We set $\delta  = - i .  e^{i\theta} $ and $ t  = i . u$. On the integration interval $[T,\infty)$ for $y$, we use the first form of $Z$. Then, from \eqref{eq3} in Lemma \ref{le:36BDGbis}, we find
\begin{align*}
&  Ê\int_{T}^\infty  \hspace{-3pt} dy\,  y^{-\frac{\alpha}{2}-1} \left|Ê \int_0^\infty  \hspace{-3pt} dr \, r^{\frac{\alpha}{2}-1} \left( e^{ - r ^{\frac \alpha 2} g_0 } \left( e^{ - r  h_0 } - e^{ - r  k_0 } \right) - e^{ - r^{\frac \alpha 2} g_u } \left( e^{ - r  h_u } - e^{ - r k_u } \right)  \right)  \right| Ê \nonumber \\
& \leq  c \int_{T}^{\infty}  \hspace{-3pt} dy\,  \frac{ y^{ -\frac{\alpha}{2}-1} | h- k | ^\kappa}{ |h|^{ \frac \alpha 2 + \kappa} |\delta|^{ \frac \alpha 2 + \kappa} }  +c \int_{T}^{\infty}  \hspace{-3pt} dy\,  \frac{ y^{-\frac{\alpha}{2}-1} ( 1 \vee y^\kappa)  | h- k |^\kappa }{ |h|^{ \frac \alpha 2 + \kappa} |t y - \delta|^{ \frac \alpha 2 + \kappa} }  \\
& \leq  c'   | h | ^{ -\frac{\alpha}{2} - \kappa }  | h- k |^\kappa (  |\delta| ^{ - \frac{\alpha}{2}-\kappa} T^{ - \frac \alpha 2}  +  |\delta| ^{ - \frac{\alpha}{2}-\kappa} T^{\kappa - \frac \alpha 2} ) ,
\end{align*}
where we have used that $\kappa < \alpha /2$. 
On the integration interval $[0,T)$ for $y$, we use the second form of $Z$. We choose $T = |i.e^{i\theta}|/2 = | \delta | /2$. For the first term,  by \eqref{qwe2} in Lemma \ref{le:36BDGbis},
\begin{align*}
&  Ê\int_{0}^T \hspace{-3pt} dy\,  y^{-\frac{\alpha}{2}-1} \left|Ê \int_0^\infty  \hspace{-3pt} dr \, r^{\frac{\alpha}{2}-1} \left( e^{ - r ^{\frac \alpha 2} g_0  } - e^{ - r ^{\frac \alpha 2} g_u } \right)  \left( e^{ - r h_0 } -  e^{ -r k_0 } \right)  \right| Ê \nonumber \\
& \leq  c \int_{0}^{T}  \hspace{-3pt} dy\,  \frac{ y^{\beta-\frac{\alpha}{2}-1} | h- k | ^\kappa |\delta|^{ \frac \alpha 2 - \beta} \| g \|_\beta  }{ |h|^{  \alpha  +  \kappa} |\delta|^{  \alpha + \kappa} }    \\
& \leq  c'   | h | ^{ -  \alpha - \kappa }  |\delta| ^{ - \alpha  - \kappa} | h- k |^\kappa \| g \|_\beta.
\end{align*}
The second term is easily bounded by \eqref{qwe3} with $\kappa_1=\kappa$ and $\kappa_2=0$.
Integrating our bounds over $\theta$, we obtain \eqref{eq:FhFkU}.
The proof that   for all $u \ne v \in S^+_1$, with $| i. u | \leq | i. v | $.
\begin{align}\label{eq:HNFghk}
& | F_h ( g) (u )   - F_k (g) (u) - F_h (g) (v)  + F_k ( g) (v ) | \\
&\quad \quad  \leq \; c |Êu -v | ^\beta  | i. u  | ^{  \frac \alpha  2 - \beta - \kappa} ( 1 + \| g \|_\beta) |h| ^{ -\frac \alpha 2}  | h - k |^{\kappa} \nonumber
\end{align}
is easier as it does not require to consider separatly small and large $y$; we leave it to the reader.
 \end{proof}

\subsubsection{Computation of characteristic function} 

With the notation of proof of Theorem \ref{th:unicityRDE}, we define for $z \in \bar \bC_+$, $u \in \cK_1^+$, 
$$
\chi_ z ( u  ) = \bE \exp ( - u . H(z)). 
$$
We note that the distribution of $R_0(z)$ is characterized by the value of $\chi_z$ on any open neighborhood in $\cK_1^+$. 
The next lemma asserts that the distribution of $R_0(z)$ is also characterized by the value of $\gamma_z$ on $\cK_1^+$
(that is on $S_+^1$ by homogeneity). 

\begin{lemma}[From fractional moment to characteristic function] \label{le:gammatochi}
Let $z \in \bar \bC_+$, $ 0 < \alpha  < 2$ and $R_0(z)$ solution of \eqref{eq:RDE} such that   $\bE | R_0(z) |^ {\frac \alpha 2 } < + \infty$. For all $u = u_1+ i u_2 \in  \cK_1^+ $, 
\begin{eqnarray*}
\chi_ z ( u ) & = & \int_{\bR_+^2} J_1 ( s ) J_1 (t) e^{ - \frac{(-iz)s^2 }{4 u_1} - \frac{(i\bar z)t^2}{4u_2} }  e^{ -  \gamma_ z \left( \frac{s^2}{4 u_1}+ i \frac{t^2}{4 u_2}  \right)Ê}  ds dtÊ \\
 &  & - \int_{0}^\infty  J_1 ( s ) e^{ -  \frac{(-iz)s^2 }{4 u_1}  }   e^{ - \Gamma (1 - \frac \alpha 2)   \gamma_z( \frac{s^2}{4 u_1})   }ds  -  \int_{0}^\infty  J_1 ( t ) e^{ -  \frac{(i\bar z)t^2 }{4 u_2}  }  e^{ - \Gamma (1 - \frac \alpha 2)   \gamma_z( i \frac{t^2}{4 u_2})
   }dt +1.
\end{eqnarray*}
where $J_1 (x) = \frac x 2 \sum_{k \geq 0} \frac { (-x^2 / 4) ^k } {k ! (k+1)!}$ is a Bessel function of the first kind. 
\end{lemma}
\begin{proof}
We use the formulas for $w \in \cK_1$,
$$
1 - e^{-w^{-1}} = \int_0 ^\infty J_1 ( s ) e^{ - \frac{ w s ^2}{4} } ds
$$
(see \cite{MR0167642}) and for $z,z' \in \bC$, 
\begin{equation*}\label{eq:ezz'}
e^{- z - z'} =  ( 1 - e^{-z}) ( 1 - e^{-z'}) - ( 1- e^{-z})  - ( 1-  e^{-z'}) +1 .
\end{equation*}  Then, it follows from \eqref{eq:RDE} that
\begin{eqnarray*}
e^{  - u_1 H - u_2 \bar H }&  \stackrel{d}{=} & \exp \left(    - \frac{Êu_1 }{Ê-i z +   \sum_{k \geq 1} \xi_k H_k } - \frac{Êu_2 }{Êi \bar z +   \sum_{k \geq 1} \xi_k \bar H_k } \right)  \\
&  \stackrel{d}{=} & \int_{\bR_+^2} J_1 ( s ) J_1 (t) e^{ -  \frac{(-iz)s^2 }{4 u_1} -  \frac{(i\bar z)t^2}{4u_2} }  e^{ - \sum_{k} \xi_k  \left( \frac{H_k s^2 }{4 u_1} +  \frac{\bar H_k t^2}{4u_2} \right) Ê}  ds dt \\
 & & -   \int_{0}^\infty  J_1 ( s ) e^{ -  \frac{(-iz)s^2 }{4 u_1}  }  e^{ - \sum_{k} \xi_k  \frac{H_k s^2 }{4 u_1} }ds -  \int_{0}^\infty  J_1 ( t ) e^{ -  \frac{(i\bar z)t^2 }{4 u_2}  }  e^{ - \sum_{k} \xi_k \frac{\bar H_k t^2 }{4 u_2} }dt + 1. 
\end{eqnarray*}
Since $J_1$ is bounded on $\bR_+$, we may safely take expectation. The conclusion follows from Levy-Khintchine formula. \end{proof}

\subsection{Proof of Theorem \ref{th:locvect}}\label{sec:prooflocvect}

We start with a simple lemma which relates $W_I (i)$ to the diagonal of resolvent.
\begin{lemma}[From eigenvectors to diagonal of resolvent] \label{lem-debnn} Let $ \alpha  > 0 $ and $I = [E - \eta , E+\eta]$ be an interval. Setting  $z =   E + i \eta \in \bC_+$, we have
$$
\frac 1 n \sum_{i=1}^n W_I (i)^{\frac \alpha 2} \leq \left( \frac { 2 n \eta}{ |Ê\Lambda_I |Ê } \right)^{\frac \alpha 2} \frac 1 n \sum_{i=1} ^n \left( \Im R ( z)_{ii} \right)^{\frac \alpha 2}.
$$
\end{lemma}
\begin{proof}
From the spectral theorem, we have
$$
\Im R ( z)_{ii} \geq \sum_{ v \in \Lambda_I} \frac{ \eta \langle v , e_i \rangle ^2 } {( \lambda_i ( A) - E )^2 + \eta^2 }  \geq \frac 1 {2\eta}   \sum_{ v \in \Lambda_I}  \langle v , e_i \rangle ^2 = \frac { |Ê\Lambda_I|  } {2 n \eta}  W_I ( i).$$
It remains to sum the above inequality. \end{proof}

At this stage, it should be clear that the proof of Theorem \ref{th:locvect} will rely on Theorem \ref{th:unicityRDE} and on an extension of the previous fixed 
point argument to finite $n$ system. The bottleneck in the proof will be on the lower bound of $  |Ê\Lambda_I |/n\eta$ which in particular requires
according to Lemma \ref{le:deconvolution} that $\mu_A(I)\le L |I|$.
This last control is difficult when $\alpha<1$ as in this case
$n^{-1}\sum_{i=1}^n(  \Im R(z)_{ii})^{\alpha/2}$ goes to zero like $\eta^{\alpha/2}$
so that arguments such as those
used in the proof of Theorem \ref{th:boundStieltjes} do not hold.
 It will be responsible for the restrictive condition $\eta \geq n ^{-\rho + o(1)}$ in the statement of Theorem \ref{th:locvect}. For completeness we will also prove in this subsection a vanishing upper bound on 
$$
 \frac 1 n \sum_{i=1} ^n \left( \Im R ( z)_{ii} \right)^{\frac \alpha 2}, 
$$
for $\eta$ of order $n^{  - 1/6}$ for all $ \alpha > 0$. More precisely, we have
\begin{theorem}[Vanishing fractional moment for the resolvent]\label{th-finn}
Let $ 0 < \alpha < 2/3 $, $0 < \e <   \frac{\alpha^2}{2 (4 - \alpha) }$, $\rho'  = \frac{2 + \alpha}{4 ( 3 +\alpha) }$ and $c_0 = \frac{ (2 + \alpha)^2} { 16 ( 3 + \alpha) }$.  There exist $c_1 = c_1 ( \alpha )$,  $c = c ( \alpha, \e) > 0$ such that if $n \geq 1$,  $z =   E + i \eta \in \bC_+$, $|z| \geq c$, $ n^{-\rho'} ( \log n )^ {c_0}   \leq \eta  \leq 1 $,
$$
\bE  \frac 1 n \sum_{i=1} ^n \left( \Im R ( z)_{ii} \right)^{\frac \alpha 2} \leq c  \eta^ { -\frac{\alpha ( 3 + \alpha) }{2 + \alpha} } n^{-\frac \alpha 2 + c_1 \e   } + c \eta^{\frac \alpha 2 - \e}. 
$$
Moreover, if  $n^{-\rho}  ( \log n )^{\frac{4}{2 + 3 \alpha} }  \leq \eta  \leq 1 $,
$$
\bE  \frac 1 n \sum_{i=1} ^n \left( \Im R ( z)_{ii} \right)^{\frac \alpha 2} \leq  c \eta^{\frac \alpha 2 - \e}.
$$
\end{theorem}
Theorem \ref{th:locvect} is a consequence of the second statement of Theorem \ref{th-finn} together  with Theorem \ref{th:boundStieltjes}, Lemma \ref{lem-debnn} and Lemma \ref{le:concres2} wich asserts that 
$$
\bP\left(   \frac 1 n \sum_{i=1} ^n \left( \Im R ( z)_{ii}   \right)^{\frac \alpha 2}  \geq \bE \frac 1 n \sum_{i=1} ^n \left( \Im R ( z)_{ii}   \right)^{\frac \alpha 2}   +  t \eta^{\frac \alpha 2 - \e} \right) \leq \exp ( - n \eta^{ 4 - \frac {4\e}{\alpha} }  t^ { \frac 1 \alpha} ) .
$$
We consider for $u\in\cK_1^+$,
$$\gamma^n_z(u):=\Gamma(1-\frac{\alpha}{2}) \bE \left[
\frac{1}{n}\sum_{k=1}^n \left( -iR(z)_{kk}.u\right)^{\frac \alpha 2}\right] =\Gamma(1-\frac{\alpha}{2}) \bE \left[
 \left( -iR(z)_{11}.u\right)^{\frac \alpha 2}\right]  \,.$$
\begin{lemma}[Bound on fractional moments of the resolvent]\label{lem-finn} Let $ 0 < \alpha / 2 \leq \beta < 2 \alpha /  ( 4 - \alpha)$ and $\rho', c_0$ as in Theorem \ref{th-finn}. There exists  $c> 0$ such that if $n \geq 1$,  $z =   E + i \eta \in \bC_+$, $|z| \geq 1$, $ \eta  \geq n^{-\rho'} ( \log n )^ {c_0}   $, then
\begin{equation}\label{bornegamman}
 \|\gamma^n_{z}\|_{\beta} \leq c. 
\end{equation}
\end{lemma}

The proof of  Theorem \ref{th-finn} provides also the local
 convergence of the fractional moments $\gamma_z^n$ for the norm
 $\|.\|_{\frac{\alpha}{2},\e}$. Indeed it is based again on an approximate 
fixed point argument for these quantities.

\begin{lemma}[Approximate fixed point for fractional moments of the resolvent]\label{lem-finn2}
Let $ 0 < \alpha < 2/3 $  and $\rho', c_0$ as in Theorem \ref{th-finn} and $G_z$ as in \eqref{defGz}. For  all $0 < \e <   \frac{\alpha^2}{2 (4 - \alpha) }$, there exists $c = c ( \alpha, \e) > 0$ such that if $n \geq 1$,  $z =   E + i \eta \in \bC_+$, $|z| \geq c$, $ n^{-\rho'} ( \log n )^ {c_0}   \leq \eta  \leq 1 $,
\begin{equation*}
 \|\gamma_{z}^n- G_z (  \gamma_{z}^n)\|
_{\frac{\alpha}{2}+\e, \e} \leq  c \eta^ { -\frac{\alpha ( 3 + \alpha) }{2 + \alpha} } n^{-\frac \alpha 4 }. 
\end{equation*} 
Moreover, if  $n^{-\rho}  ( \log n )^{\frac{4}{2 + 3 \alpha} }  \leq \eta  \leq 1 $,
$$
 \|\gamma_{z}^n- G_z (\gamma_{z}^n)\|_{\frac{\alpha}{2}+\e, \e} \leq \eta^ { -\frac{5 \alpha}{4} }  n ^{ - \frac \alpha 4}.
$$

\end{lemma}

We now check that the above two lemmas imply Theorem \ref{th-finn}. Note in the proof
below that they also imply the convergence of $\gamma_z^n$ to
$\gamma_z$ for $\eta\ge n^{-\rho'}(\log n)^{c_0}$.
\begin{proof}[Proof of Theorem \ref{th-finn}]Ê
We prove the first statement. Let $0 < \e <   \frac{\alpha^2}{2 (4 - \alpha) }$ and $\delta =  \eta^ { -\frac{\alpha ( 3 + \alpha) }{2 + \alpha} } n^{-\frac \alpha 4 }$. 
Now since $\|\gamma_{z}^n\|_{\frac \alpha 2 + \e} $ and $\| \gamma_{z} \|_{\frac \alpha 2 + \e}$ are uniformly bounded, we have by Lemma \ref{le:ContFh} and Lemma \ref{lem-finn2},
$$\|\gamma_{z}^n -\gamma_{z}\|_{\frac{\alpha}{2}+\e,\e}\le
c|z|^{-\alpha}\|\gamma_{z}^n -\gamma_{z}\|_{\frac{\alpha}{2}+\e,\e}
+c \delta $$
as long as $|z| \geq c$ with imaginary part $ n^{-\rho'} ( \log n )^ {c_0}   \leq \eta  \leq 1 $. Hence, if $|z|$ is large enough, $c|z|^{-\alpha}$ is less than $2$ and it follows that 
$$\|\gamma_{z}^n -\gamma_{z}\|_{\frac{\alpha}{2}+\e,\e}\le  2 c \delta.
$$
Now, we may argue as in the proof of Theorem \ref{th:unicityRDE}.  By Theorem \ref{th:unicityRDE}, for $|z|$ large enough, $| \gamma_z ( e^{i \frac \pi 4} ) | \leq c' \eta^{\frac \alpha 2 - \e} $, for some constant $c' > 0$. Then, for any $u \in S^1_+$, using Lemma \ref{lem-finn} and Lemma \ref{le:regFM}, 
\begin{align*}
& \Gamma(1-\frac{\alpha}{2}) \bE \Im  R(z)_{11} ^{\frac \alpha 2}    =  | \gamma^n_z   ( e^{i \frac \pi 4} ) |  \\
& \leq    | \gamma^n_z   ( e^{i \frac \pi 4} ) -  \gamma^n_z  (u  )  | + | \gamma^n_z   ( u ) -  \gamma_z ( u  )  |  +  |  \gamma_z  ( e^{i \frac \pi 4} ) -  \gamma_z  (u  ) |  + |\gamma_z ( e^{i \frac \pi 4} ) |  \\
& \leq   c'' | u - e^{i \frac \pi 4} |^{\frac \alpha 2} + \| \gamma_z    -  \gamma_z^n   \|_{\beta, \e} | i . u |Ê^{-\e}   + |\gamma_z ( e^{i \frac \pi 4} ) |  \\
& \leq  c'' | u - e^{i \frac \pi 4} |^{\frac \alpha 2} +  c''\delta | u - e^{i \frac \pi 4} |^{- \e} + c' \eta^{\frac \alpha 2 - \e}.
\end{align*}
Choosing $u$ such that $ | u - e^{i \frac \pi 4} |$ is of order $\delta ^{\frac{2Ê}{\alpha  + 2 \eÊ} } $, we deduce that for all $z = E + i \eta$ with $ |E| \geq E_{\alpha, \e}$, $\bE \Im  R(z) ^{\frac \alpha 2} $ is bounded up to a multiplicative constant, by 
$$
 \eta^ { -\frac{\alpha ( 3 + \alpha) }{2 + \alpha}   } n^{-\frac \alpha 4 +  O ( \e )  } + \eta^{\frac \alpha 2 - \e}.  
$$
 Since $\e > 0$ can be arbitrarily small, this concludes the proof for the case $ n^{-\rho'} ( \log n )^ {c_0}   \leq \eta  \leq 1 $. The proof for $n^{-\rho}  ( \log n )^{\frac{4}{2 + 3 \alpha} }  \leq \eta  \leq 1 $ is identical : we find,
$$
\bE  \frac 1 n \sum_{i=1} ^n \left( \Im R ( z)_{ii} \right)^{\frac \alpha 2} \leq c \eta^ { -\frac{5 \alpha}{4} }  n ^{ - \frac \alpha 4 +  O ( \e )} +   c \eta^{\frac \alpha 2 - \e}.
$$
It remains to notice that for $\e$ small enough, in our range of $\eta$, the second term dominates the first term.
\end{proof}

\begin{proof}[Proof of Lemma \ref{lem-finn}]
As in the proof of Lemma \ref{le:regFM}, it is sufficient to check that for some constant $c = c( \alpha, \beta)$, 
\begin{equation} \label{eq:Rbeta}
\bE  | R_{11} ( E  + i \eta) | ^ \beta    \leq c |E|^{-\beta}.
\end{equation}
As usual, from \eqref{eq:resolventformula}, we have
\begin{align*}
| R(z)_{11} |Ê =    \left|  \left( z - a_n^{-1} X_{11}  +   a_n ^{-2} \langle  X_1 , R^{(1)}X_1 \rangle \right)  \right|^{-1}. 
\end{align*}
We first get rid of the non-diagonal term in the scalar product $\langle  X_1 , R^{(1)}X_1 \rangle$. We perform this as in the proof of Lemma \ref{le:diagapprox}. Using the definition \eqref{defT} and  \eqref{fondin} with $\alpha / 2 = \beta$, we find 
$$
\left| \bE  | R ( E  + i \eta)_{11} | ^ \beta   -  \bE  \left|  z  +    a_n ^{-2} \sum_{i= 2} ^n R^{(1)}_{ii} X_{1i}^2    \right|^{-\beta} \right|  \leq c \eta^{-2 \beta}\bE | T(z)| ^ \beta
$$
In particular, since $ | z |^{-\beta} \leq  | \Re (z)  |^{-\beta}$, we find
$$
 \bE  | R( E  + i \eta)_{11} | ^ \beta   \leq  \bE  \left|  E  +    a_n ^{-2} \sum_{i= 2} ^n \Re  ( R^{(1)}_{ii} ) X_{1i}^2    \right|^{-\beta}  +  c \eta^{-2 \beta}\bE | T(z)| ^ \beta
$$
Now, we decompose the sum into a positive and  a negative part
$$
\sum_{i= 2} ^n \Re  ( R^{(1)}_{ii} ) X_{1i}^2  = \sum_{i= 2} ^n  \left( \Re  ( R^{(1)}_{ii} )  \right)_+  X_{1i}^2 - \sum_{i= 2} ^n \left( \Re  ( R^{(1)}_{ii} ) \right)_-  X_{1i}^2. 
$$
Note that, conditioned on $R^{(1)}$, the two sums are independent. We invoke Lemma \ref{le:formulealice}
$$
a_n ^{-2} \sum_{i= 2} ^n \Re  ( R^{(1)}_{ii} ) X_{1i}^2   \stackrel{d}{=} a S - b S',
$$
where, conditioned on $R^{(1)}$,  $a,b,S,S'$ are independent non-negative random variables, $S,S'$ being $\alpha/2$-stable random variables. Hence from what precedes,
$$
\bE  | R( E  + i \eta) _{11}| ^ \beta   \leq  \bE  \left|  E  +  a S - b S'   \right|^{-\beta}  +  c \eta^{-2 \beta}\bE | T(z)| ^ \beta.
$$
Assume for example that $E > 0$. Let $\cF$ be the filtration generated by $(R^{(1)}, a,b,S)$ and $\bE' = \bE [ \cdot | \cF]  $. Using Lemma \ref{momentinv} conditionnaly to $\cF$ yields
 that for some constant $c'> 0$, 
$$
 \bE  \left|  E  +  a S - b S'   \right|^{-\beta}  = \bE \left[ \bE'  \left|  E  +  a S - b S'   \right|^{-\beta}  \right] \leq  c' \,  \bE  \left|  E  + a S   \right|^{-\beta} \leq c'   E^{-\beta} . 
$$
If $ E < 0$, we repeat the same argument with the filtration generated by $(R^{(1)}, a,b ,S')$. 
%If $\e$ is as in Lemma \ref{le:diagapprox}, we deduce finally that
% \begin{equation*} 
%\bE  | R_{11} ( E  + i \eta) | ^ \beta  \leq  c \eta^{-\beta} \varepsilon + c'   \left|  E   \right|^{-\beta}. \end{equatio%n*} %Al:Ca me semble pas tout a fait exact (il faudrait cutoff etc)et c'est inutile donc je l'efface 

Now, if $0 < \beta < 2 \alpha /  ( 4 - \alpha)$, using the tail bound \eqref{boundTbis}, we find
\begin{equation}\label{eq:ETb}
\bE | T(z)| ^ \beta \leq c \left( n^{-\frac \beta  \alpha} + \left( \frac{M_n}{n} \right)^\frac{\beta}{2} \right).
\end{equation}
We now use the bound given by \eqref{eq:boundMn} on $M_n$ which is valid for all $\eta \geq n^{-\frac{\alpha +2 }{4}}$,
$$
\eta^{-2 \beta}\bE | T(z)| ^ \beta \leq c  \eta^{ - 2 \beta  -\frac{  2\beta} { 2 + \alpha} } n^{-\frac \beta 2 }  (   \log n ) ^ {\frac{ \beta ( 2 + \alpha) } {8}  }.  
$$
This concludes the proof of the lemma, since for $\eta \geq n^{-\rho'} ( \log n )^ { \frac{ (2 + \alpha)^2} { 16 ( 3 + \alpha) } }  $, the above expression is uniformly bounded. 
\end{proof}

Note that in the proof of Lemma \ref{lem-finn} we have used the bound \eqref{eq:boundMn} instead of the bound $M_n \leq c\eta^{-1}$ given by the proof of Proposition \ref{prop:boundStieltjes} because it is valid for a wider range of $\eta$. 

\begin{proof}[Proof of Lemma \ref{lem-finn2}]
Set $h  = -i z \in \cK_1$, $H_{k}(h)=-i R^{(1)}(ih)_{kk}$ and define
\begin{eqnarray*}
I^n_h(u) & =&\Gamma \left( 1 - \frac \alpha 2 \right) \bE \left(    \left( h +a_n^{-2} \sum_{k=2}^n X_{1k}^2 H_k \right)^{-1} . u  \right)^{\frac \alpha 2 } \\
& = & \Gamma \left( 1 - \frac \alpha 2 \right) \bE \left(  \frac{ h. \check u + a_n^{-2} \sum_{k=2}^n X^2_{1k} H_k(h). \check u  }{ \left| h +a_n^{-2} \sum_{k=2}^n X_{1k}^2 H_k \right|^2 } \right)^{\frac \alpha 2 } ,
\end{eqnarray*}
where we recall that $\check u = \Im (u) + i \Re (u)$.

\noindent{\em Step one : Diagonal approximation. }  In this first step, we generalize Lemma \ref{le:diagapprox}. We will upper bound the expression
$
\|\gamma_{ih}^n-I^n_{h}\|_{\beta,\e}.
$
Using the definition
 \eqref{defT}, we find that for any $u\in S_1^+$,
with $\eta=\Re(h)>0$,
$$|\gamma_{ih}^n(u)-I^n_{h}(u)|\le 
c \eta^{ - \alpha} \bE[|T(z)|^{\frac{\alpha}{2}}],$$
where we have used \eqref{fondin} with $\beta=\frac{\alpha}{2}$.  Using \eqref{eq:ETb} for $\beta = \alpha / 2$, we deduce that 
$$|\gamma_{ih}^n(u)-I^n_{h}(u)|\le c \eta^{ - \alpha}  \left(n^{-1/2}
+\left(\frac{M_n}{n}\right)^{\frac{\alpha}{4}}\right).$$
Whereas using \eqref{eq:boundMn} to bound $M_n$, we find for $n^{ - \frac{\alpha +2 }{4}}  \leq \eta\leq 1$ that 
\begin{equation}\label{contn}
|\gamma_{ih}^n(u)-I^n_{h}(u)|\le c n^{ - \frac \alpha 4 }\eta^ { - \frac{ \alpha ( 3 + \alpha)} { 2 + \alpha } }  ( \log n )^ { \frac{ \alpha ( 2 + \alpha)} { 8 } } .
\end{equation}
To bound  
$$\Delta_h^n(u,v):=|\gamma_{ih}^n(u)-\gamma_{ih}^n(v)-I^n_{h}(u)+I^n_h(v)|$$
we first observe that for $x_1,x_2,y_1,y_2\in \cK_1$, by using the 
standard interpolation trick, for $\kappa_1,\kappa_2,\beta \in [0,1]$,
we have
if $N=|x_1|\wedge |x_2|\wedge |y_1|\wedge |y_2|$,  and $\kappa_1+\kappa_2\ge \alpha/2$
$$|x_1^{\frac{\alpha}{2}}-x_2^{\frac{\alpha}{2}}-y_1^{\frac{\alpha}{2}}+y_2^{\frac{\alpha}{2}}|\le  N^{\frac{\alpha}{2}-\kappa_1-\kappa_2}
|x_1-y_1|^{\kappa_1}(|x_1-x_2|^{\kappa_2}+|y_1-y_2|^{\kappa_2})
+N^{\frac{\alpha}{2}-\beta}|x_1-x_2-y_1+y_2|^\beta. $$ 
We use this inequality  with
$$x_j= \frac{ h . \check  u + a_n^{-2} \sum_{k=2}^n X^2_{1k}  H_k(h).\check u  -i (j-1) T(z).\check  u }{ \left| h +a_n^{-2} \sum_{k=2}^n X_{1k}^2 H_k  -i (j-1) T(h)\right|^2 },
$$
and in $y_j$, $v$ replaces $u$. For $j \in \{ 1, 2\} $, one can check that, with $D_j=( h +a_n^{-2} \sum_{k=2}^n X_{1k}^2 H_k
-i(j-1) T(z))^{-1}$, 
$$|x_i-y_i|\le |D_i| |u-v| \;  , \quad
|x_1-x_2| \vee |y_1-y_2|  \le |D_1| |D_2| |T(z)|,$$
and
$$|x_1-x_2-y_1+y_2|\le |D_1| |D_2| |u-v||T(z)|\,.$$
Moreover, using \eqref{eq:boundDP}, we find $N\ge  ( |i.u|\wedge |i.v|  )  ( |D_1| \wedge |D_2| ) $. Recall finally that $|D_1|$ and $|D_2|$ are bounded by $\eta^{-1}$. Hence, choosing $\kappa_1=\beta,\kappa_2=\frac{\alpha}{2} +\e$ (with $\e$
small enough so that $\frac{\alpha}{2}+\e<\frac{2\alpha}{4-\alpha}$)
we deduce that, 
$$\Delta_h^n(u,v)\le \eta^{ - \beta -\frac  \alpha 2 }
(|i.u|\wedge |i.v|)^{-\beta-\e}
|u-v|^\beta \bE[|T(z)|^{\frac{\alpha}{2}+\e}]+\eta^{ - \beta - \frac{\alpha}{2}}
|u-v|^\beta \bE[|T(z)|^\beta].$$
We naturally choose $\beta=\frac{\alpha}{2}+\e < \frac{2\alpha}{4-\alpha}$.  From \eqref{boundTbis}, $T(z)
\in L^\beta$ and 
$$\bE[|T(z)|^\beta]\le c n^{-\frac{\beta}{\alpha}} +
c\left(\frac{M_n}{n}\right)^{\frac \beta 2}
\le  c \left(\eta^{\frac 4 {2 + \alpha} }  n ( \log n )^{ -  \frac{2 + \alpha} { 4} }  \right)^{-\frac \beta 2},$$
where we have finally assumed that $n^{-\frac{2 + \alpha }{4} }  \leq \eta\leq 1$ and used \eqref{eq:boundMn}. This gives  for $n^{-\frac{2 + \alpha }{4} }  \leq \eta\leq 1$  
\begin{equation*}
\|\gamma^n_{ih}-I^n_h\|_{\frac{\alpha}{2}+\e,\e}
\le c  \eta^{ - \beta - \frac{\alpha}{2} - \frac {2 \beta} {2 + \alpha} }  n^{-\frac \beta 2}  ( \log n )^{  \frac{\beta(2 + \alpha)} { 8} } 
\end{equation*}
Now it easy to check that for $\eta \geq n^{-\frac{2 + \alpha }{2 (4+\alpha)} } $, we have $\eta^{ - \beta - \frac{\alpha}{2} - \frac {2 \beta} {2 + \alpha} }  n^{-\frac \beta 2}  < \eta^ { -\frac{\alpha ( 3 + \alpha) }{2 + \alpha} } n^{-\frac \alpha 4 }$. It follows for $n^{-\rho' }  \leq \eta\leq 1$ and a new constant $c > 0$, depending on $\e$, that
 \begin{equation}\label{eq:tyy}
\|\gamma^n_{ih}-I^n_h\|_{\frac{\alpha}{2}+\e,\e}
\le c  \eta^ { -\frac{\alpha ( 3 + \alpha) }{2 + \alpha} } n^{-\frac \alpha 4 } .
\end{equation}

If instead we assume that $n^{-\rho}  ( \log n )^{\frac{4}{2 + 3 \alpha} }  \leq \eta  \leq 1 $, then, from the proof of Proposition \ref{prop:boundStieltjes}, we may use the stronger bound $M_n \leq c \eta^{-1}$ if $|z|$ large enough. We find instead
\begin{equation}\label{eq:tyybis}
\|\gamma^n_{ih}-I^n_h\|_{\frac{\alpha}{2}+\e,\e}
\le c \eta^{- \frac{ 3 \beta } {2} - \frac \alpha 2} n ^{ - \frac \beta 2} \leq   c \eta^ { - \frac{5 \alpha}{4} }  n ^{ - \frac \alpha 4 } .
\end{equation}
(where, for the last inequality, we have used the fact that $\eta \geq n ^{-1/3}$ for $n^{-\rho}  ( \log n )^{\frac{4}{2 + 3 \alpha} }  \leq \eta  \leq 1 $ and $n$ large enough).

\noindent{\em Step two : approximate fixed point equation. }
Next, we extend the proof of Proposition \ref{prop:fixpoint}.   We denote by 
$\bE_1 [Ê\cdot ]Ê$ and $\bP_1 ( \cdot) $ the conditional expectation and probability given $\cF_1$, the $\sigma$-algebra generated by the random variables $(X_{ij})_{ i \geq j \geq 2}$. We assume that  $\alpha / 2 < \beta < 1 - \alpha / 2 $ and $0 < \e < 1 - 3 \alpha / 2$. We first remark that by arguments similar to the proof
of Lemma \ref{le:fpgamma} and by Corollary \ref{cor:formulealice},
we have 
$$I^n_h(u)=\bE [ G_{z} (Z_n)(u)],$$
where, conditionned on $\cF_1$, $Z_n(u)=\frac{\kappa}{n}\sum_{k=2}^n ( H_k . u ) ^{\frac{\alpha}{2}}|g_k|^\alpha$, $g_k$ are i.i.d standard normal variables and $\kappa = \Gamma ( 1 - \alpha / 2 ) / \bE |g_1 | ^\alpha$. 
Note that,  from Lemma \ref{le:regFM}, $\|Z_n\|_\beta\le \frac{c}{n }\sum_{k=2}^n | H_k | ^{\frac{\alpha}{2}}
|g_k|^\alpha$ which belongs to $L^p$ for any $p >0$. Therefore, we can use Lemma \ref{le:ContFh} and the H\"older inequality
to insure that, 
\begin{equation}
\label{ty}
\|I^n_h- G_z (\gamma_z^n)\|_{\beta,\e}
\le c|h|^{-\alpha} \left(1+\|\gamma_z^n\|_\beta+\bE\left[\left(\frac{\kappa}{n}\sum_{k=2}^n
 | H_k | ^{\frac{\alpha}{2}} |g_k|^\alpha\right)^p\right]^{1/p}\right) \bE\left[ \|\gamma_z^n-Z_n\|_{\beta,\e}^q\right]^{1/q},\end{equation}
where $1 / p + 1 /q = 1$ and 
$$ \bE[ \|\gamma_z^n-Z_n\|_{\beta,\epsilon}^q]^{1/q}
\le  \bE[ \|\bE [ Z_n ] -Z_n\|_{\beta,\epsilon}^q]^{1/q}
+ \|\gamma_z^n-\bE [ Z_n ]  \|_{\beta,\epsilon}\,.$$
From the triangle inequality, 
$$
\bE[ \|\bE [ Z_n ]  -Z_n\|_{\beta,\epsilon}^q]^{1/ q } \leq \bE[ \|\bE [ Z_n ]  - \bE_1 [Z_n] \|_{\beta,\epsilon}^q]^{1/ q }  + \bE[ \|\bE_1 [ Z_n ]  -Z_n\|_{\beta,\epsilon}^q]^{1/ q }. 
$$ 
But, using Lemma \ref{le:devenet} from the appendix, 
\begin{align*}
&  \bE[ \|\bE_1 [ Z_n ]  -Z_n\|_{\beta,\epsilon}^q] \leq  \bE\left\|\frac{c }{n}\sum_{k=2}^n   ( H_k . u )^{\frac \alpha 2}   (|g_k|^\alpha-\bE |g_k|^\alpha)\right\|_{\beta, \e} ^q  \\
&   \leq c(q) ( \log  n )^ { \frac q 2}   ( \eta^{ \alpha }n ) ^{-\frac q 2}.
\end{align*}
Similarly, by Lemma \ref{le:devenet},
$$
\bE[ \|\bE [ Z_n ]  - \bE_1[  Z_n ] \|_{\beta,\epsilon}^q] \leq c^q \bE\left\| \bE \frac{1 }{n} \sum_{k=2}^n ( H_k . u )^{\frac \alpha 2}    - \frac{1 }{n} \sum_{k=2}^n ( H_k . u )^{\frac \alpha 2}   \right\|^q_{\beta, \e} \leq  c'(q) ( \log n )^{\frac {q \alpha} 4} ( \eta^ 2 n ) ^{ - \frac {q \alpha} 4 } . 
$$
Whereas using \eqref{eq:rankineq} as we did in the proof of Proposition \ref{prop:fixpoint}, we have, with $c_0 = 2 \Gamma ( 1 - \frac \alpha  2 )$, 
$$\|\gamma_z^n-\bE Z_n \|_{\beta,\epsilon}\le 
c_0 (n\eta)^{-\frac{\alpha}{2}}.$$
Hence, there exists a new constant $c(q)$ such that
$$
\bE\left[ \|\gamma_z^n-Z_n\|_{\beta,\e}^q\right]^{1/q}  \leq c(q) ( \log n )^{\frac { \alpha} 4} ( \eta^ 2 n ) ^{ - \frac { \alpha} 4 }.
$$
Similarly, using the triangle inequality at the first line,  \eqref{eq:rankineq} at the second line and the Jensen inequality at the third, 
\begin{align*}
&  \bE\left[\left(\frac{1}{n}\sum_{k=2}^n
 | H_k | ^{\frac{\alpha}{2}} |g_k|^\alpha\right)^p\right]^{1 / p}  \\
 & \quad \leq \bE\left[\left(\frac{1}{n}\sum_{k=2}^n
 | H_k | ^{\frac{\alpha}{2}} \bE |g_k|^\alpha  \right)^p\right]  ^{1 / p}  +  \bE\left[ \left|\frac{1 }{n}\sum_{k=2}^n | H_k | ^{\frac{\alpha}{2}}(|g_k|^\alpha-\bE |g_k|^\alpha)\right|^p \right]  ^{1 / p}   \\
 & \quad \leq \bE |g_1 |^\alpha \bE\left[\left(\frac{1}{n}\sum_{k=2}^n
 | R_k | ^{\frac{\alpha}{2}}  \right)^p\right]  ^{1 / p}  +  c_0(n\eta)^{-\frac{\alpha}{2}} +  c(p)^{ 1 / p}  ( \eta^{ \alpha }n ) ^{-\frac 1 2} \\
 & \quad \leq \bE |g_1 |^\alpha \left(\frac{1}{n}\sum_{k=2}^n
\bE  | R_k | ^{\frac{p \alpha}{2}}  \right)  ^{1 / p}  +  c_0 (n\eta)^{-\frac{\alpha}{2}} +  c(p)^{ 1 / p}  ( \eta^{ \alpha }n ) ^{-\frac 1 2}. 
\end{align*}
We choose $p > 1$ such that $p \alpha /2  <  2 \alpha / (4-\alpha)$ and we finally use \eqref{eq:Rbeta} and Lemma \ref{lem-finn}. Then, for our range of $\eta$, the right hand side of the above inequality is of order $1$. Putting these estimates in \eqref{ty}, we find finally that for any $\alpha / 2 < \beta < 2 \alpha / ( 4 - \alpha)$, any $0 < \e < 1 - 3 \alpha / 2$ and $n^{-\rho' } (\log n ) ^{c_0}  \leq \eta\leq 1$,  there exists a constant $c(\alpha, \beta, \e)$ such that  
$$
\|I^n_h- G_z(\gamma_z^n)\|_{\beta,\e} \leq c |Êz | ^{-\alpha} ( \log n )^{\frac { \alpha} 4} ( \eta^ 2 n ) ^{ - \frac { \alpha} 4 }. 
$$
Putting this together with \eqref{eq:tyy}, this conclude our proof (the above term is negligible compared to the right hand side of \eqref{eq:tyy} or \eqref{eq:tyybis}).  \end{proof}

\appendix

\section{Concentration of Gaussian measure}

In this paragraph, we recall a well-known concentration phenomenon of the Gaussian measure. The following classical result is contained in Ledoux \cite{ledoux01}. It is a consequence of the Logarithmic Sobolev inequality for the Gaussian measure and the Herbst argument.
\begin{theorem}[Concentration of Gaussian measure]\label{th:concnorm}
Let $F$ be a $1$-Lipschitz function on the Euclidean space $\bR^n$  and $G$ be a standard Gaussian vector in $\bR^n$ $N(0,I_n)$, then for every $ r \geq 0$,  
$$
\bP \left(  F ( G)  -  \bE [ÊF(G)Ê]Ê \geq r \right) \leq   e^{ - \frac{r^2}{2}},  
$$
where $m_F$ is the median of $F$ for $N(0,I_n)$.
 \end{theorem}
For $p , q  > 0$, we define for $x \in \bR^n$
$$
\|x \|_p = \left( \sum_{i =1}^n |x_i| ^p \right)^{\frac 1 p},
$$
and for a matrix $A$
$$
\|A \|_{p \to q} =  \sup_{\|x \|_p = 1 } \| A x \|_q.
$$
(this is a norm for $p,q \geq 1$) The usual operator norm is denoted by 
$$
\|A \| = \|A \|_{2 \to 2} = \sup_i  | s_i|
$$
where the $s_i$'s are the singular values of $A$. Recall that if $0 < p \leq 2$,
$$
\|I_n \|_{2  \to p} = n^{\frac 1 p - \frac 1 2}.
$$

\begin{corollary}\label{cor:concnorm}
Let $A$ be a $n\times n$ non-negative matrix, $0\le  p \leq 2$ and $G$ be a standard Gaussian vector in $\bR^n$ $N(0,I)$. There exist positive constants $c, \delta > 0$  depending  only on $p$, such that if
$
\left( \tr A^{p} \right) ^\frac 1 p \geq c   \|A \|n^{ \frac 1 p  - \frac 1 2}
$
then
$$
\|A G \|_ p \geq \delta \left( \tr A^p \right) ^\frac 1 p,
$$
with probability at least 
$$
1 -   \exp \left\{ - \delta   \left( \frac { \left( \tr A^{p} \right) ^\frac 1 p }{\|A \| n^{\frac 1   p - \frac 1  2} }\right)^2 \right\}.
$$
\end{corollary}

\begin{proof}
We first consider the case $1 \le p\le 2$.
 We define $F (x) = \|A x \|_p$. From the triangle inequality (valid for all $p\ge 1$)
 $$
 |F ( x) - F( y) | \leq  F( x- y) = \|A ( x - y)  \|_p \leq \|x - y\|_2 \| A \|_{2 \to p}. 
 $$
 Since $\| A \|_{2 \to p } \leq \| A \|_{2 \to 2} \|I_n \|_{ 2 \to p}$, we deduce that $F$ is Lipschitz with constant 
 $$
 \sigma = \|A \|n^{ \frac 1 p  - \frac 1 2}.
 $$ 
It follows by Theorem \ref{th:concnorm} that for every $ r \geq 0$,  
\begin{equation}\label{eq:concAGp}
\bP \left(     \|A G \|_ p - \bE  \|A G \|_ p   \leq r \right) \leq    e^{ - \frac{r^2}{2\|A \|^2_{2 \to p}}} \leq   e^{ - \frac{r^2}{2\|A \|^2 n^{2 / p  - 1}}}.  
\end{equation}
The corollary will follow by applying the above inequality to $r =\bE  \|A G \|_ p  /2$ and by showing that, for some constant $c_0>0$, 
\begin{equation}\label{eq:concAGp1}
\bE  \|A G \|_ p \geq c_0 \left( \tr A^{p} \right) ^\frac 1 p.
\end{equation}
From \eqref{eq:concAGp}, for some $c_1 >0$, 
$$
\bE \left|  \|A G \|_ p -  \bE  \|A G \|_ p  \right|^p \leq  (c_1 \sigma ) ^p.
$$
Hence 
\begin{equation}\label{eq:concAGp2}
\bE  \|A G \|_ p \geq \left( \bE \|A G \|^p _ p\right)^{\frac 1 p} - c_1 \sigma. 
\end{equation}
Now, let $(\lambda_k,u_k)_{1 \leq k \leq n}$ be the eigenvalues and normalized eigenvectors of $A$. We note that 
$$
(A G)_i   \stackrel{d}{=} \sum_{k=1}^n \lambda_k \langle u_k , e_i \rangle G_k .$$
In particular, $(A G)_i$ has distribution $N ( 0, \sum_k \lambda_k^2 \langle u_k , e_i \rangle^2)$ and for some $c_2 >0$, 
$$
\bE \|A G \|^p _ p =  \sum_{i=1}^n  \bE \left|\sum_{k=1}^n \lambda_k \langle u_k , e_i \rangle G_k \right|^p
 = c_2 \sum_{i=1}^n \left(  \sum_{k=1}^n \lambda_k^2 \langle u_k , e_i \rangle^2 \right)^{\frac p 2}.
$$
For $0 < p \leq 2$ and $\sum_{k=1}^n \langle u_k , e_i \rangle^2=1$ for all $i\in \{1,\ldots,n\}$, we may use the Jensen inequality:
\begin{equation}\label{eq:AGpp}
\bE \|A G \|^p _ p  \geq c_2 \sum_{i=1}^n  \sum_{k=1}^n \lambda_k^p \langle u_k , e_i \rangle^2 = c_2 \tr A^p. 
\end{equation}
Then, from \eqref{eq:concAGp2} and the value of $\sigma$, we deduce that \eqref{eq:concAGp1} holds with $c_0 = c_2 ^{ 1/ p} /2 $ if $c$ is chosen large enough so that $ c^{1/p}_2 c \geq  2 c_1$.

We next consider the case $0 \le p\le 1$. We denote,
for $R=(\frac{\kappa}{n} \tr A^p)^{1/p}$ with some positive constant $\kappa$ to be chosen later, $\phi$ a 
 non-negative Lipschitz function which is lower bounded uniformly
by $|x|^p$, is equal to $|x|^p$  on $|x|\ge R$, and Lipschitz constant bounded by $R^{p-1}$. In particular, the $\bR^n \to \bR_+$ function $x \mapsto \sum_i \phi (x_i)$ is Lipschitz with constant bounded by $ \sqrt n R^{p-1}$. 
It follows that the $\bR^n \to \bR_+$ function $F (x) = \sum_i \phi ((A x)_i)$  is Lipschitz with constant bounded by $ \| A \|Ê\sqrt n R^{p-1}$.
Hence,  by Theorem \ref{th:concnorm}, for any $r>0$, 
$$\bP\left(  \sum_{i=1}^n \phi(\langle AG, e_i\rangle)-\bE\sum_{i=1}^n \phi(\langle AG, e_i\rangle) \le - r \right)\le e^{-\frac{ r^2}{Ê2 \|ÊA \|^2  n R^{2p-2}}}. $$

Now, we observe that
$$\|AG\|_p^p\ge   \sum_{i=1}^n  \phi(\langle AG, e_i\rangle)$$
and also,
$$ \bE \sum_{i=1}^n \phi(\langle AG, e_i\rangle)\geq \bE \|AG\|_p^p  -  \kappa \tr(A^p).$$
Therefore from \eqref{eq:AGpp},
if we choose $\kappa<c_2/4$,
$$\bE \sum_{i=1}^n \phi(\langle AG, e_i\rangle)\ge \frac{3c_2}{4} \tr(A^p).$$
Finally, we set $r= c_2\tr(A^p)/4$ and conclude that
$$\bP\left(\|AG\|_p^p\le \frac{c_2}{2} \tr(A^p)
\right)\le \bP\left(\sum_{i=1}^n \phi(\langle AG, e_i\rangle)\le 
 \frac{c_2}{2}\tr(A^p)\right)\le e^{-c n
\left(\frac{\tr (A^p)}{n \|A\|^p}\right)^{\frac{2}{p}}}.$$
\end{proof}

\begin{remark} \label{re:deloc}
Consider the special case where $A$ is the projector on a vector space $W$ of dimension $d$. Then $A^p = A$ for all $p >0$. Corollary
 \ref{cor:concnorm} gives a lower bound for $\|A G \|_p$ of order $d^{\frac 1 p}$ when $d\ge c^p n^{1-p/2}$. However, if $(u_1, \cdots, u_d)$ is an orthonormal basis of $W$ such that $\langle u_k , e_i \rangle^2 \geq \e^2 / n$ then for $p\le 1$ we have a lower bound  for $\|A G \|_p$ of order $\e n^{ \frac 1 p - \frac 1 2} d^{\frac 1 2}$ which can be significantly larger. Hence, we expect that Corollary \ref{cor:concnorm} is sharp if $W$ has a localized basis and not sharp if $W$ has a delocalized basis.
\end{remark}

\section{Stable distributions}

In this paragraph, we give some properties of stable distributions. 

Let $\sigma >0$, $0 < \alpha < 2$ and $\beta \in [-1,1]$. A real random variable $X$ has $\alpha$-stable distribution $\stab_\alpha(\beta,\sigma)$ if its Fourier transform is given  for all $t \in \bR$, by
\begin{equation}
\label{eq:stablefourier}
\bE \exp (it X)  = \exp\left[  - \sigma^\alpha | t|^\alpha\,(1\!-\!i \beta\,\textrm{sgn}(t) u_\alpha)~\right]
\end{equation}
where $\textrm{sgn}(t)$  is the sign of $t$ and $
    u_\alpha =\tan(\pi \alpha/2)\,
$
for all $\alpha$ except $\alpha = 1$ in which case
$
    u_1=-(2/\pi)\log|t|.\,
$

If $0 < \alpha < 1$ and $\beta  =1$, the distribution $\stab_\alpha(1,\sigma)$ has support $\bR_+$ and its Laplace transform is conveniently given for all $t \in \bR_+$, by
\begin{equation}
\label{eq:stablelaplace}
\bE \exp (- t X)  = \exp\left[  - \sigma^\alpha  t ^\alpha v_\alpha~\right],
\end{equation}
with $v_\alpha = \frac 2 \pi\sin \left( \frac {\pi \alpha}{2} \right)  \Gamma ( 1 - \alpha) \Gamma ( \alpha )   $.

\begin{lemma}[Decomposition of quadratic form]\label{le:formulealice}
Let $X = (X_i)_{1 \leq i \leq n}$ be iid  symmetric $\alpha$-stable random variables with distribution $\stab_\alpha(0,\sigma)$. Let $A$ be a $n\times n$ positive definite matrix, then
$$
\langle X , A X\rangle \stackrel{d}{=}  \|A^{1/2} G \|_\alpha^{2} S,
$$ 
where $G$ is a standard gaussian vector $N( 0 , I ) $ independent of $S$, a positive $\alpha / 2$-stable  $\stab_{\frac \alpha 2}(1, 2 \sigma^2 v_{\frac \alpha 2}^{-\frac 2 \alpha})$.
\end{lemma}
\begin{proof}
We use the identity, for $y \in \bR^n$, $$
\exp ( - \frac{t^2}{2}  \langle y , y\rangle) = \bE \exp (   i t \langle  y , g \rangle ). 
$$
Applied to $y = A^{1/2} X$, we get, for $t \geq 0$, 
$$
\bE \exp (- t  \langle X , A X\rangle  ) = \bE \exp ( i \sqrt{2  t}  \langle A^{1/2} X ,  G \rangle  ) = \bE \exp ( i \sqrt{2  t}  \langle  X ,  A^{1/2} G \rangle  ).
$$
Then, since $X$ is stable vector,  $\langle  X ,  A^{1/2} G \rangle$ has distribution $\stab_\alpha (0 , \sigma \| A^{1/2} G \|_\alpha)$. From \eqref{eq:stablefourier}, it follows
$$
\bE \exp (- t  \langle X , A X\rangle  ) =  \bE \exp ( -  (2  t )^{\frac{\alpha}{2}} \sigma^\alpha \|A^{1/2} G \|^{\alpha}_\alpha ).
$$
Then, we conclude by applying \eqref{eq:stablelaplace}. \end{proof}

\begin{corollary}[Sum of weighted squares]\label{cor:formulealice}
Let $X = (X_k)_{1 \leq k \leq n}$ be iid  symmetric $\alpha$-stable random variables with distribution $\stab_\alpha(0,\sigma)$ and let $(w_k)_{1 \leq k \leq n} \in \bC^n_+$. Then 
$$
\bE \exp \left(  i \sum_{k=1}^n w_k X_k^2 \right) = \bE \exp \left( - (-2i)^{\frac \alpha 2} \sigma^\alpha \sum_{k=1}^n w_k ^{\frac \alpha 2} |g_k|^\alpha \right),
$$ 
where $G = (g_1, \cdots, g_n)$ is a standard gaussian vector $N( 0 , I ) $.
\end{corollary}

\begin{proof}
We set $\rho_k = - i w_k$, we shall prove that 
\begin{equation}\label{eq:corformuleD}
\bE \exp \left(  - \sum_{k=1}^n \rho_k X_k^2 \right) = \bE \exp \left( - 2^{\frac \alpha 2} \sigma^\alpha \sum_{k=1}^n \rho_k ^{\frac \alpha 2} |g_k|^\alpha \right). 
\end{equation}
We  write $\rho_k =  i ( a_k - b_k )  + c_k$, where $a_k, b_k$ are the negative and positive parts of $\Re (w_k)$ and $c_k = \Im (w_k) >0$. We set $\rho_k (t,s) =  t a_k +   s  b_k +  c_k$, $D = \{z \in \bC : \Re (z ) > 0 \} = - i \bC_+$ and  $D_\e= \{z \in \bC : \Re (z ) >  - \e, | \Im (z) | < 2 \}$, where $2 \e = \min ( c_k / ( a_k + b_k) )$. Then, the $D^2_\e  \to D$ function $(t,s) \mapsto  \sum_{k=1}^n \rho_k (t,s) X_k^2$ is analytic in each of its coordinates. Since the function $z \mapsto \exp ( -  z )$  is analytic and bounded on $D$, from Montel's Theorem, we deduce that the $D_\e^2  \to \bC$ function 
$$
\varphi : (t,s) \mapsto \bE \exp \left(  - \sum_{k=1}^n \rho_k (t,s) X_k^2 \right)
$$ 
is analytic in each of its coordinates in $D_\e$. However,  for $s,t \in \bR_+$, we notice that $\rho_k(s,t) \in \bR_+$. Hence  by Lemma \ref{le:formulealice} applied to a diagonal matrix, we have
$$
\varphi (t,s) = \bE \exp \left( - 2^{\frac \alpha 2} \sigma^\alpha \sum_{k=1}^n \rho_k(s,t) ^{\frac \alpha 2} |g_k|^\alpha \right). 
$$
The $D \to D$ function $z \mapsto z^{\alpha / 2}$ is analytic. We may thus again apply Montel theorem and deduce that the right hand side of the above identity is analytic in $(s,t)$ on $D_\e^2$. So finally, the above equality holds true for all $(s,t) \in D^2_\e$. Applied to $(s,t) = (i,-i)$, we obtain precisely \eqref{eq:corformuleD}.
  \end{proof}

The next lemma looks at the behavior of a positive stable random variable near $0$. 
\begin{lemma}[Tail of inverse positive stable variable] \label{le:tailS}
Let $\sigma >0$, $0 < \alpha <1$ and $S$ be a positive $\alpha $-stable $\stab_{\alpha}(1, \sigma)$ random variable. There exists a positive constant $c_0 (\alpha)$ such that for all $0 < c <   \sigma^{\frac{\alpha}{1 - \alpha}} c_0 ( \alpha )$,
$$
\bE \exp ( c S^{- \frac{\alpha}{1 - \alpha} } ) < \infty, 
$$
while the above is infinite for $c >  \sigma^{\frac{\alpha}{1 - \alpha}} c_0 ( \alpha )$.
\end{lemma}

\begin{proof}
From the identity, for $m >0$, $x > 0$, 
$$
x^{-m} = \frac{1}{\Gamma ( m) } \int_0^\infty t^{m - 1} e^{-xt} dt, 
$$
we deduce that, for $p \geq 0$, 
$$
\exp ( c  x^{-p}) =  \sum_{k \geq 0} \frac{c^k}{ \Gamma ( k p ) \Gamma ( k+1) }  \int_0^\infty t^{kp  - 1} e^{-xt} dt\,.
$$
In particular, from \eqref{eq:stablelaplace}, with $\hat \sigma = \sigma v_{\alpha}^{1/\alpha}$, 
and Fubini's theorem,
\begin{eqnarray*}
\bE \exp ( c S^{-p} ) &=&  \sum_{k \geq 0} \frac{c^k }{ \Gamma ( k p ) \Gamma ( k+1) }  \int_0^\infty t^{kp  - 1} e^{-t^\alpha \hat \sigma^\alpha} dt \\
& = & \alpha^{-1}  \sum_{k \geq 0}  c^k \hat \sigma^{-kp} \frac{\Gamma(\frac{k p}{\alpha} )  }{ \Gamma ( k p ) \Gamma ( k+1) } .
\end{eqnarray*}
The conclusion follows easily from Stirling's formula,  $\Gamma(x) \sim_{x \to \infty} \sqrt{\frac{2 \pi}{ x}}~{\left( \frac{x}{e} \right)}^x$. 
\end{proof}

\begin{lemma}[Negative fractional moments of smooth random variable]\label{momentinv}
Let $ \alpha > 0$ and  $S$ be a real-valued random variable with law which has a uniformly bounded density 
on $[-1,1]$ and is bounded by $c |x|^{-\alpha -1}$ on $[-1,1]^c$ for
some finite positive constant $c$.
Then, for any $0 < \beta<1$, there exists a finite constant $C $
so that  for any $x\in \mathbb R$, any $\sigma\ge 0$,
we have 
\begin{equation*}\label{control1}
\bE[ |x-\sigma S|^{-\beta}]\le  C |x|^{-\beta}\,.
\end{equation*}
\end{lemma}
\begin{proof}
Let us first assume that $\sigma\ge 2|x|$. If $C$ is a bound on the density of the law of $S$
on $[-1,1]$, for $T\ge (2/\sigma)^\beta$,
$$\bE[|x-\sigma S|^{-\beta}]\le T+\int_T^\infty\bP\left( |x-\sigma S|\le t^{-1/\beta}\right) dt
\le T+ C (1-\beta)^{-1} T^{1-\frac{1}{\beta}}\sigma^{-1} \,.$$
Choosing $T=(2/\sigma)^\beta\le |x|^{-\beta}$ provides the desired estimate.
In the case $\sigma\le 2|x|$ and $t^{-1/\beta}\le |x|/2$, we have
$$\bP\left( |x-\sigma S|\le t^{-1/\beta}\right)\le C \left(\frac{x}{\sigma}\right)^{-\alpha -1} t^{-\frac{1}{\beta}}\sigma^{-1}$$
Therefore if $\sigma\le 2|x|$ and $T=(2/|x|)^\beta$,
$$\bE[|x-\sigma S|^{-\beta}]\le T+C \sigma^{\alpha} x^{-\alpha-1} (1-\beta)^{-1} T^{1-\frac{1}{\beta}}
\le C' |x|^{-\beta} +C'\sigma^{\alpha} |x|^{-\alpha-\beta}\le C'(1+2^\alpha)|x|^{-\beta}
\,,$$
which completes the proof of the lemma.
\end{proof}

\section{Concentration of random matrices with independent rows}

The total variation norm of $f:\bR \to\bR $ is
\[
\|f \| _\textsc{TV}:=\sup \sum_{k \in \bZ} | f(x_{k+1})-f(x_k) |, \, 
\]
where the supremum runs over all sequences $(x_k)_{k \in \bZ}$ such that
$x_{k+1} \geq x_k$ for any $k \in \bZ$. If $f = \ind_{(-\infty,s]}$ for
some real $s$ then $\|f \|_\textsc{TV}=1$, while if $f$ has a derivative in
$\mathrm{L}^1(\bR)$, we get
\[
\|f \| _\textsc{TV}=\int |f'(t)|\,dt.
\] 
\begin{lemma}[Concentration for spectral measures \cite{BCC_heavygirko}]\label{le:concspec}
  Let $A$ be an $n\times n$ random Hermitian matrix. Let us assume that the
  vectors $(A_i)_{1 \leq i \leq n}$, where $A_i := (A_{ij})_{1 \leq j \leq i}
  \in \bC^i$, are independent. Then for any $f:\bR\to\bC$ such that
  $\| f \|_\textsc{TV}\leq1$ and $\bE |\int\!f\,d\mu_A |<\infty$, and every
  $t\geq0$,
  \[
  \bP \left( \left| \int\!f\,d\mu_A -\bE\int\!f\,d\mu_A \right|  \geq t \right) %
  \leq 2 \exp\left({-\frac{n t^2}{2}}\right).
  \]
\end{lemma}

The next lemma is an easy consequence of Cauchy-Weyl interlacing Theorem. It is an ingredient of the proof of Lemma \ref{le:concspec}. 
\begin{lemma}[Interlacing of eigenvalues]\label{le:Cauchy}
Let $A$ be an $n \times n$ hermitian matrix and $B$ a principal minor of $A$. Then for any $f:\bR\to\bC$ such that
  $\| f \|_{TV} \leq 1$ and $\lim_{|x| \to \infty} f(x) = 0$, 
$$
\left|\sum_{i=1}^n f ( \lambda_i ( A) )  - \sum_{i=1}^{n-1} f ( \lambda_i ( B) ) \right|\leq  1.
$$
\end{lemma}

The Lipschitz norm of $f:\bC \to\bC $ is
$$
\|f \|_{L} = \sup_{ x \ne y }\frac{ |f(x) - f(y) | }{ |x - y |}.
$$

\begin{lemma}[Concentration for the diagonal of the resolvent]\label{le:concres}
  Let $A$ be an $n\times n$ random Hermitian matrix and consider its resolvent matrix $R (z) = (A - z)^{-1} $, $z \in \bC_+$. Let us assume that the
  vectors $(A_i)_{1 \leq i \leq n}$, where $A_i := (A_{ij})_{1 \leq j \leq i}
  \in \bC^i$, are independent. Then for any $f:\bC\to\bR$ such that
  $\| f \|_{L} \leq 1$, and every
  $t\geq0$,
  \[
  \bP \left( \left|\frac 1 n  \sum_{k=1}^n  f (R (z) _{kk} ) -\bE \frac 1 n  \sum_{k=1}^n  f (R (z) _{kk} )  \right|  \geq t \right) %
  \leq 2 \exp\left({-\frac{n  \Im (z)^{2 } t^2}{8}}\right).
  \]
\end{lemma}

\begin{proof}
The proof is close to the proof of Lemma \ref{le:concspec} as done in \cite{BCC_heavygirko} and relies on the method of bounded martingale difference.  We start by showing that for every $n\times n$
  deterministic Hermitian matrices $B$ and $C$ and any measurable function $f$
  with $\| f \|_{L} \leq 1$,
  \begin{equation}\label{eq:rankineq}
  \left|\frac 1 n  \sum_{k=1}^n  f (R_B (z) _{kk} ) -\frac 1 n  \sum_{k=1}^n  f (R_C (z) _{kk} )  \right|  \leq 2 \left( { n \Im (z) } \right)^{-1}  \mathrm{rank}(B-C),
  \end{equation}
  where $R_B = (B - z)^{-1} $ and $R_C = (C - z)^{-1} $ are their resolvent matrices. Indeed, by assumption
  $$
    \left|  \sum_{k=1}^n  f (R_B (z) _{kk} )  -  \sum_{k=1}^n  f (R_C (z) _{kk} )  \right|  \leq   \sum_{k=1}^n    \left| R_B (z) _{kk}  -    R_C (z) _{kk}  \right| .
    $$ 
The resolvent identity asserts that 
  $$
 M :=  R_B - R_C = R_B ( C - B ) R_C.
  $$
 It follows that  $r = \mathrm{rank}(M)\leq \mathrm{rank}(B - C)$. We notice also that $ \|M \| \leq 2 \Im (z)^{-1}$. Hence, in the singular value decomposition of $M = U D V$, at most $r$ entries of $D = \mathrm{diag} ( s_1, \cdots, s_n)$ are non zero and they are bounded by $\| M \|$.  We denote by $u_1, \cdots, u_r$ and $v_1, \cdots, v_r$ the associated orthonormal vectors so that 
  $$
  M = \sum_{i = 1}^r s_i u_i v_i ^ * ,
  $$
and
  \begin{eqnarray*}
  \left| R_B (z) _{kk}  -    R_C (z) _{kk}  \right|  = |M_{kk}|  =     \left| \sum_{i=1}^r s_i \langle u_i , e_k \rangle  \langle v_i , e_k \rangle   \right|  \leq \| M \|   \sum_{i=1}^r |\langle u_i , e_k \rangle | | \langle v_i , e_k \rangle |. 
  \end{eqnarray*}
  We obtain from Cauchy-Schwarz,
   \begin{eqnarray*}
   \frac 1 n  \sum_{k=1}^n    \left| R_B (z) _{kk}  -    R_C (z) _{kk}  \right| & \leq &\|M \| \sum_{i= 1} ^r  \sqrt{ \frac 1 n  \sum_{k=1}^n   |\langle u_i , e_k \rangle |^{2 } }\sqrt{ \frac 1 n  \sum_{k=1}^n | \langle v_i , e_k \rangle |^{2 }  }\\
    &= & r \|M \|  n^{-1}. 
    \end{eqnarray*}
    Equation \eqref{eq:rankineq} is thus proved. 
    
  Next, for any
  $x = (x_1, \ldots, x_n) \in\cX:= \{(x_i)_{1 \leq i \leq n} : x_i \in
  \bC^{i-1} \times \bR\}$, let $B(x)$ be the $n \times n$ Hermitian
  matrix given by $B(x)_{ij} = x_{i,j}$ for $1 \leq j \leq i \leq n$ and $R_x (z) = (B(x) - z) ^{-1}$. We thus have
  $R (z)  = R _{(A_1, \ldots, A_n)} (z) $. For all $x \in \cX$ and $x'_i \in
  \bC^{i-1} \times \bR $, the matrix
  \[ 
  B(x_1,\ldots,x_{i-1}, x_i , x_{i+1}, \ldots , x_n) %
  - B(x_1,\ldots,x_{i-1}, x'_i , x_{i+1}, \ldots , x_n) 
  \]
  has only the $i$-th row and column possibly different from $0$, and thus
  \[
  \mathrm{rank}\left(B(x_1,\ldots,x_{i-1}, x_i , x_{i+1}, \ldots , x_n) %
    - B(x_1,\ldots,x_{i-1}, x'_i , x_{i+1}, \ldots , x_n)\right) \leq 2 .
  \]
  Therefore from \eqref{eq:rankineq}, we obtain, for every
  $f:\bR \to\bR $ with $\| f \|_{L} \leq 1$,
  \[
 \left| \frac 1 n \sum_{k=1} ^n f ( R_{(x_1,\ldots,x_{i-1}, x_i , x_{i+1}, \ldots , x_n)}  (z) _{kk} )  %
    -  \frac 1 n \sum_{k=1} ^n f ( R_{(x_1,\ldots,x_{i-1}, x'_i , x_{i+1}, \ldots , x_n)}  (z) _{kk}  )  \right| %
  \leq 4 \left(   n \Im (z)  \right)^{-1}.
  \]
  The desired result follows now from the Azuma--Hoeffding inequality, see
  e.g.\ \cite[Lemma 1.2]{mcdiarmid}.  \end{proof}

\begin{lemma}[Concentration for the diagonal of the resolvent]\label{le:concres2}
  Let $A$ be an $n\times n$ random Hermitian matrix and consider its resolvent matrix $R (z) = (A - z)^{-1} $, $z \in \bC_+$. Let us assume that the
  vectors $(A_i)_{1 \leq i \leq n}$, where $A_i := (A_{ij})_{1 \leq j \leq i}
  \in \bC^i$, are independent. Then for any $\gamma\in [0,1]$, there exists a positive constant $c$ so that
for  every
  $t\geq0$,
  \begin{equation}\label{control2}
  \bP \left( \left|\frac 1 n  \sum_{k=1}^n   (R (z) _{kk} )^\gamma -\bE \frac 1 n  \sum_{k=1}^n  (R (z) _{kk} )^\gamma  \right|  \geq t \right) %
  \leq 2 \exp\left(-c n\Im (z)^{2 } t^{\frac{2}{\gamma}}\right).
  \end{equation}
\end{lemma}
\begin{proof}
Let $\varepsilon$ be a positive real number 
and  $\phi_\varepsilon:\bC\to\bC$ be equal to one on $|z|\ge 2 \varepsilon$, vanishing on $|z|\le \varepsilon$
and growing linearly with the modulus in between. Thus, $\phi_\varepsilon$ is Lipschitz with constnat bounded by $1/\varepsilon$.
We decompose $x^\gamma$ as
$$x^\gamma=x^\gamma \phi_\varepsilon(x)+x^\gamma(1-\phi_\varepsilon(x))\,.$$
By definition, $x^\gamma(1-\phi_\varepsilon(x))$ has modullus uniformly bounded above by $(2\varepsilon)^\gamma$ so that
if we choose $\varepsilon>0$ so that $(2\varepsilon)^\gamma=t/4$ then with $f(x)=x^\gamma \phi_\varepsilon(x)$
we have
\begin{align*}
& \bP \left( \left|\frac 1 n  \sum_{k=1}^n   (R (z) _{kk} )^\gamma -\bE \frac 1 n  \sum_{k=1}^n  (R (z) _{kk} )^\gamma  \right|  \geq t \right) \\
&\quad  \le  \; \bP \left( \left|\frac 1 n  \sum_{k=1}^n  f (R (z) _{kk} ) -\bE \frac 1 n  \sum_{k=1}^n  f(R (z) _{kk} ) \right|  \geq t /2\right)\,.
\end{align*}
On the other hand, $f$ is Lipschitz with constant bounded by $2 \varepsilon^{\gamma-1}=2^{\frac{2}{\gamma}-\gamma} t^{1-\frac{1}{\gamma}}$.
Hence, Lemma \ref{le:concres} yields
\begin{eqnarray*}
\bP \left( \left|\frac 1 n  \sum_{k=1}^n   (R (z) _{kk} )^\gamma -\bE \frac 1 n  \sum_{k=1}^n  (R (z) _{kk} )^\gamma  \right|  \geq t \right)&\le& 2\exp\left(-\frac{n \Im(z)^2 t^2}{8 4^{\frac{2}{\gamma}-\gamma} t^{2(1-\frac{1}{\gamma})}}\right)\\
&\le& 2\exp\left(-\frac{n \Im(z)^2 t^{\frac{2}{\gamma}}}{8 4^{\frac{2}{\gamma}-\gamma}}\right)\,.\end{eqnarray*}
This concludes the proof. \end{proof}

We conclude this appendix by deviations inequalities for the norm $\| \cdot \|_{\beta, \e}$ introduced in Section \ref{sec:RDE}. 
\begin{lemma}[Deviation of the $\beta$-norm]\label{le:devenet}
Let $0 < \alpha \leq 1$, $\e \geq 0$, $\alpha / 2  < \beta  < 1$ and assume that  $\alpha / 2 + \beta + \e \leq 1$. Let $(g_k), 1 \leq k \leq n,$ be iid standard Gaussian variable and $(h_k), 1 \leq k \leq n \in \cK_1$ with $|h_k| \leq \eta^{-1}$. Define, $\gamma (u) = \frac 1 n \sum_{k= 1} ^ n (h_k. u )^{\frac \alpha 2}    (  |g_k|^\alpha  - \bE | g_k |^ \alpha ) $. Then there exists constant $c_0,c_1$, such that for all $t \geq 1$, all $n \geq 2$,
$$
\bP \left( \left\| \gamma  \right\|_{\beta, \e}  \geq \frac { t } { ( \eta^ \alpha n ) ^{\frac 1 2 } }  \right) \leq c_1 n ^ {\frac  2  \alpha }   \exp ( - c_0 t^2 ).
$$
Similarly,   let $A$ be an $n\times n$ random Hermitian matrix and consider its resolvent matrix $R (z) = (A - z)^{-1} $, $z = E +i \eta \in \bC_+$. Let $H_k = - i R  _{kk} (z)$ be as above and  $\gamma' (u) = \frac 1 {n} \sum_{k= 1} ^ n (H_k. u )^{\frac \alpha 2}   - \bE \frac 1 {n} \sum_{k= 1} ^ n (H_k. u )^{\frac \alpha 2}  $, for all $t \geq 1$, all $n \geq 2$,
$$
\bP \left( \left\| \gamma'  \right\|_{\beta, \e}  \geq \frac { t } { ( \eta^ 2 n ) ^{\frac \alpha 4 } }  \right) \leq c_1 n   \exp ( - c_0 t^{4/\alpha } ).
$$
\end{lemma}
\begin{proof}
Set $L = \frac 1 n \sum_{k=1} ^ n |  |g_k|^\alpha  - \bE | g_k |^ \alpha |$. We first use a net argument. For any $u, v \in S_+^ 1$, from \eqref{fondin}, for some constant $c = c(\alpha)$. 
$$
\frac{ | \gamma ( u ) - \gamma(v) | } { | u - v | ^ \beta} ( | i . u | \wedge |i . v| ) ^{\beta - \frac \alpha 2 }  \leq c L \eta^{-\frac \alpha 2}. 
$$
In particular, setting, for integer $m$ and $1 \leq k \leq m$, $u_k = e^{ i 2 \pi k / m}$, we find
$$
| \gamma  (u)|   \leq   c L  ( m \eta ) ^{-\frac \alpha 2} +  \max_{k}| \gamma (u_k ) |  .
$$
Notice also that if $| u - v | \leq 4 / m $, then, with $\beta' = \alpha / 2 + \beta + \e \leq 1$, 
\begin{eqnarray*}
\frac{ | \gamma (u ) - \gamma (v)  | } { | u - v |^\beta } ( | i . u | \wedge |i . v| ) ^{\beta + \e}  & \leq &  \frac{ c L \eta^{-\frac \alpha 2} | u - v |^ {\beta'} ( | i . u | \wedge |i . v| ) ^{\frac \alpha 2  - \beta'}  } { | u - v |^\beta } ( | i . u | \wedge |i . v| ) ^{\beta + \e} \\
& \leq & 4 c L (  m \eta )^{-\frac \alpha  2}.  
\end{eqnarray*}
While if $| u - v | \geq 4 / m $, we denote by $u_*$ and $v_*$, the element of $\{ u_k : 1 \leq k \leq m\} $ at distance at most $1 / m$ of $u$ and $v$ and with $| i . u_* | \geq |i . u|$, $| i . v_* | \geq |i . v|$. We get $| u - v | \geq 2 | u_* - v_* |$ and 
\begin{eqnarray*}
\frac{ | \gamma (u ) - \gamma (u_*)  | } { | u - v |^\beta } ( | i . u | \wedge |i . v| ) ^{\beta + \e}  & \leq &  \frac{ c L \eta^{-\frac \alpha 2} | u - u_* |^ {\beta'} ( | i . u | \wedge |i . v| ) ^{\frac \alpha 2  - \beta'}  } { | u - v |^\beta } ( | i . u | \wedge |i . v| ) ^{\beta + \e} \\
& \leq & c' L (  m \eta )^{-\frac \alpha  2}.  
\end{eqnarray*}
We deduce that, for some constant $c_0 \geq 1$,  
\begin{equation} \label{eq:enet}
\| \gamma \|_{\beta , \e} \leq c_0 L (  m \eta )^{-\frac \alpha  2} + c_0 \max_{k}| \gamma (u_k ) |  + c_0 \max_{ k \ne \ell }  \frac{ | \gamma ( u_k  ) - \gamma( u_\ell ) | } { | u_k  - u_\ell  | ^ \beta} ( | i . u_k  | \wedge |i . u_\ell | ) ^{\beta - \frac \alpha 2 }.
\end{equation}
On the other hand, since $0 < \alpha \leq 1$, the random variable $|g_k|^\alpha$ is sub-gaussian. It follows from Hoeffding's inequality, that for any $s \geq 0$, 
$$
\bP ( L \geq \bE L + s ) \leq \exp ( - c n s^2),  
$$
and for any $u, v \in S_+^ 1$, 
\begin{align*}
& \bP \left( | \gamma (u ) |  \geq  s\right) \leq 2 \exp ( - c n s^2 \eta^{\alpha } ) \quad \hbox { and } \\
& \quad \quad \bP \left( \frac{ | \gamma (u ) - \gamma (v)  | } { | u - v |^\beta } ( | i . u | \wedge |i . v| ) ^{\beta - \frac \alpha 2 }  \geq  s\right) \leq 2 \exp ( - c n s^2 \eta^{\alpha } ).  
\end{align*}
From the union bound, we get from \eqref{eq:enet}, 
$$
 \bP \left( \left\| \gamma  \right\|_{\beta, \e}  \geq c_0 \left( \frac{t}{(\eta^\alpha n)^{\frac 1 2}}
  +  ( \bE L + s ) ( m \eta)^{-\frac \alpha 2} \right) \right) \leq \exp ( - c n s^2) + m^2  \exp ( - c t^2 ). 
$$
We take $m = [n^{1/\alpha} (  ( \bE L  + s ) )  ^ { 2 / \alpha} t ^ { - 2 / \alpha}]$ and $ s = t $, we find for all $t \geq 2 / c_0$, 
$$
 \bP \left( \left\| \gamma  \right\|_{\beta, \e}  \geq \frac {2 c_0 t } {  ( \eta^ \alpha n ) ^{\frac 1 2 } }  \right) \leq c ' n ^ { 2 / \alpha }   \exp ( - c t^2 ). 
$$
This prove the first statement. For the second statement, the proof is similar.  First, the above net argument gives that \eqref{eq:enet} holds for $\gamma'$ with $L= 1$. Also the proof of Lemma \ref{le:concres2} implies  that for any $u, v \in S_+^ 1$, 
\begin{align*}
& \bP \left( | \gamma' (u ) |  \geq  s\right) \leq 2 \exp ( - c n \eta^{2}  s^{\frac 4 \alpha} ) \quad \hbox { and } \\
& \quad \quad \bP \left( \frac{ | \gamma' (u ) - \gamma' (v)  | } { | u - v |^\beta } ( | i . u | \wedge |i . v| ) ^{\beta - \frac \alpha 2 }  \geq  s\right) \leq 2 \exp ( - c n \eta^{2}  s^{\frac 4 \alpha} ).  
\end{align*}
From the union bound, we deduce that for all $s \geq 0$,
$$
 \bP \left( \left\| \gamma'  \right\|_{\beta, \e}  \geq c_0 \left( s  -  ( m \eta)^{-\frac \alpha 2} \right) \right) \leq  m^2 \exp ( - c n \eta^{2}  s^{\frac 4 \alpha} ). 
$$
Taking $ s = t  / ( \eta^ 2 n ) ^{\frac \alpha 4 } $ and  $ m  = \sqrt n t ^{- 2 / \alpha} (2 c_0 ) ^{ 2 / c_0} $, this concludes the proof.
\end{proof}

\section*{Acknowledgments}

The authors thank G\'erard Ben Arous, Amir Dembo and Terence Tao for a stimulating discussion on the topic held at the Institute of Pacific for Pure and Applied Mathematics, Los Angeles. We are very grateful to Michael Aizenman, Djalil Chafa\"i and Benjamin Schlein,  as this work  benefited  a lot from their comments.

 \bibliographystyle{amsplain}
\bibliography{bib}

\vfill

\end{document}